\documentclass[11pt,reqno,a4paper,draft]{amsart}
\usepackage[T1]{fontenc} 
\usepackage[utf8]{inputenc} 

\usepackage{lmodern}
\usepackage{microtype}
\usepackage[scrtime]{prelim2e}
\usepackage{geometry}
\usepackage{enumerate}
\usepackage[final]{hyperref}
\usepackage[abbrev,backrefs]{amsrefs}

\usepackage[usenames,dvipsnames,svgnames,table]{xcolor}

\definecolor{Blu}{rgb}{0,0,1}

\newcommand{\R}{\mathbb{R}}
\newcommand{\Rset}{\mathbb{R}}
\newcommand{\N}{\mathbb{N}}
\newcommand{\Z}{\mathbb{Z}}

\newcommand{\C}{\mathbb{C}}
\newcommand{\e}{\mathrm{e}}
\newcommand{\defeq}{:=}
\newcommand{\dif}{\,\mathrm{d}}
\newcommand{\norm}[1]{\left\lVert #1 \right\rVert}              %
\newcommand{\bignorm}[1]{\bigl\lVert #1 \bigr\rVert}              %
\newcommand{\abs}[1]{\lvert #1 \rvert}                     %
\newcommand{\bigabs}[1]{\bigl\lvert #1 \bigr\rvert}			%
\newcommand{\weakto}{\rightharpoonup}
\newcommand{\dualprod}[2]{\langle #1, #2 \rangle}
\newcommand{\scalprod}[2]{( #1 \vert #2 )}
\newcommand{\bigscalprod}[2]{\bigl( #1 \big\vert #2 \bigr)}
\newcommand{\floor}[1]{\lfloor #1 \rfloor}

\newcommand{\Lin}{\mathrm{Lin}}
\newcommand{\st}{\;:\;}

\DeclareMathOperator{\sign}{sign}

\DeclareMathOperator{\Div}{div}
\DeclareMathOperator{\trace}{tr}

\newtheorem{proposition}{Proposition}[section]
\newtheorem{theorem}{Theorem}
\newtheorem{lemma}[proposition]{Lemma}

\newtheorem{claim}{Claim}
\theoremstyle{definition}

\theoremstyle{remark}
\newtheorem{remark}{Remark}

\newenvironment{proofclaim}[1][Proof of the claim]{\begin{proof}[#1]}{\end{proof}}
\newcommand{\resetclaim}{\setcounter{claim}{0}}

\begin{document} 

\title[Groundstates of magnetic nonlinear Schr\"odinger equations]{Properties of groundstates of nonlinear Schr\"odinger equations under a weak constant magnetic field}

\author{Denis Bonheure}
\address{Denis Bonheure 
\newline \indent Département de Mathématiques, Université Libre de Bruxelles,
\newline \indent CP 214, Boulevard du triomphe, B-1050 Bruxelles, Belgium,
\newline \indent and INRIA- Mephisto team.}
\email{Denis.Bonheure@ulb.ac.be}

\author{Manon Nys}
\address{Manon Nys 
\newline \indent Dipartimento di Matematica Giuseppe Peano, Università degli Studi di Torino (UNITO), 
\newline \indent Via Carlo Alberto, 10 10123 Torino, Italy.}
\email{manonys@gmail.com}

\author{Jean Van Schaftingen}
\address{Jean Van Schaftingen 
\newline \indent Institut de Recherche en Mathématique et Physique, Université Catholique de Louvain,
\newline \indent Chemin du Cyclotron 2 bte L7.01.01, 1348 Louvain-la-Neuve, Belgium.}
\email{Jean.VanSchaftingen@uclouvain.be}

\thanks{D. Bonheure and J. Van Schaftingen were partially supported by the Projet de Recherche (Fonds de la Recherche Scientifique--FNRS) n. T.1110.14 ``Existence and asymptotic behavior of solutions to systems of semilinear elliptic partial differential equations''. D. Bonheure is partially supported by  INRIA~-- Team MEPHYSTO,
MIS F.4508.14 (FNRS)  \& ARC AUWB-2012-12/17-ULB1-IAPAS.
M. Nys is partially supported by the project ERC Advanced Grant 2013 n. 339958: ``Complex
Patterns for Strongly Interacting Dynamical Systems -- COMPAT''. M. Nys wishes to thank the
Belgian Fonds de la Recherche Scientifique -- FNRS}

\date{\today}

\subjclass[2010]{%
35J61 %
(35B07, %
35B40, %
35J10, %
35Q55%
)}

\keywords{Semilinear elliptic problem; magnetic Schr\"odinger operators; convergence of spectrum; asymptotic decay of solutions.}

\begin{abstract}
We study the qualitative properties of groundstates of 
the time-inde\-pen\-dent magnetic semilinear Schr\"odinger equation 
\[
 -(\nabla  + i A)^2 u + u = \vert u \vert^{p - 2} u \qquad \text{in $\mathbb{R}^N$}
\]
where the magnetic potential $A$ induces a constant magnetic field.
When the latter magnetic field is small enough, we show that the groundstate solution is unique
up to magnetic translations and rotations in the complex phase space,
that groundstate solutions share the rotational invariance of the magnetic field and that the
presence of a magnetic field induces a Gaussian decay.
In this small magnetic field r\'egime, the correponding ground-energy is a convex differentiable function of the magnetic field.
\end{abstract}

\maketitle

\tableofcontents

\section{Introduction}

We are interested in the \emph{time-independent magnetic semilinear Schr\"odinger equation}
\begin{equation}%
  \label{eqNLSEMag}
  -\Delta_A u + u = \abs{u}^{p - 2} u \qquad \text{in \(\R^N\)},
\end{equation}
in dimension \(N \ge 2\), 
with a \emph{linear magnetic potential} \(A \in \Lin(\R^N, \bigwedge^1 \R^N)\)
which allows to define the \emph{magnetic Laplacian}
\[
 - \Delta_A  := -\Delta -2i \, A\cdot\nabla  - i\Div A\,  +\abs{A}^{2},
\]
and a subcritical power \(p\) in the nonlinearity, i.e. \(2 < p < \frac{2N}{N-2}\).

Infinitely many solutions have been constructed for the magnetic nonlinear Schr\"o\-dinger equation \eqref{eqNLSEMag}, see e.g.\ \cite{EstebanLions1999}.
In this work we are interested in the qualitative properties of the \emph{groundstates} (or \emph{least-energy solutions}) of the problem \eqref{eqNLSEMag}, which can be obtained and characterized as minimizers of the variational problem
\[
   \inf \,\bigl\{\mathcal{I}_{A}(u) \st u \in H^1_A (\R^N, \mathbb{C}) \setminus \{0\} \, \text{ and } \,  \mathcal{I}'_A(u) = 0 \bigr\}.
\]
Here the \emph{magnetic Sobolev space} $H^1_A(\mathbb{R}^N, \C)$ is a \emph{real} Hilbert space given by
\[
  H^1_{A}(\R^N, \mathbb{C}) \defeq \bigl\{u \in L^2(\R^N, \mathbb{C}) \st  D_A u \in L^2(\R^N) \bigr\},
\]
and the functional $\mathcal{I}_A : H^1_A (\Rset^N,\C)  \to \R $ is defined for each function \(u \in H^1_A (\Rset^N,\C)\) by
\[
 \mathcal{I}_{A}(u) \defeq \frac{1}{2} \int_{\R^N} \bigl( \abs{D_A u}^2 + \abs{u}^2 \bigr) - \frac{1}{p} \int_{\R^N} \abs{u}^p,
\]
where the \emph{magnetic covariant derivative} is defined by
\[
  D_A u = D u + i Au.
\]
Critical points of the functional \(\mathcal{I}_A\) correspond to weak solutions of the equation \eqref{eqNLSEMag}. For more details about those objects, we refer to \S \ref{section:preliminaries}.

The aim of the present work is to understand the symmetry properties of the groundstates of the magnetic nonlinear Schr\"odinger equation \eqref{eqNLSEMag}, their asymptotic decay at infinity and their dependence on the magnetic field \(B = dA \in \bigwedge^2 \mathbb{R}^N \).
In order to alleviate the statement of the results, 
we recall how the problem can be simplified by gauge fixing, i.e., by choosing a specific gauge
and how the problem is invariant under magnetic translations.

The \emph{gauge invariance of the magnetic Hamiltonian}
means that if for some function \(\psi \in C^1 (\Rset^N)\), we set 
\begin{equation}
\label{eq:gauge}
\Tilde{A} = A + d \psi\ \text{ and }\ \Tilde{u}  = \e^{-i \psi} u,
\end{equation}
then 
\[
 D_{\Tilde{A}} \Tilde{u} = \e^{-i \psi} D_A u.
\]
In particular, if \(u \in H^1_A (\Rset^N,\C)\), then \(\mathcal{I}_{\Tilde{A}} (\Tilde{u}) = \mathcal{I}_A (u)\) and therefore solutions of the equation \eqref{eqNLSEMag} with \(A\) and \(\Tilde{A}\) can be related to each other using the relation \eqref{eq:gauge}. Since \(d \Tilde{A} = d A\) and \(\abs{\Tilde{u}} = \abs{u}\), the gauge invariance means that 
only the \emph{magnetic field} \(dA\) plays a role in the physical behavior of the solutions of \eqref{eqNLSEMag}.
When, as in the present work, the magnetic field $dA$ is constant, one of the simplest 
gauge choice is to assume that \(A\) is linear and skew-symmetric. 
Equivalently, this means that \(A\) is represented by an antisymmetric matrix. 
If \(A\) is linear, such a choice can be made by setting \(\psi (x) = -A (x)[x]/2\) in \eqref{eq:gauge}.
This choice is equivalent to the choice of the Coulomb gauge with a transversal boundary condition at infinity (or, equivalently, divergence free with a Neumann boundary condition at infinity), i.e., 
\begin{equation}
\label{eqCondGaugeFixing}
  \left\{
  \begin{aligned}
    \Div A & = 0 & &\text{in \(\Rset^N\)},\\
    \frac{A (x)[x]}{\abs{x}^2} &\to 0 & & \text{as \(\abs{x} \to \infty\)}.
  \end{aligned}
  \right.
\end{equation}
In particular, if \(B \in \bigwedge^2 \Rset^N\) is a constant skew-symmetric form, there exists a unique \(A \in \Lin(\Rset^N, \bigwedge^1 \Rset^N)\) satisfying \(dA = B\) and \eqref{eqCondGaugeFixing}.  
This potential \(A\) is defined for each \(x \in \Rset^N\) by
\begin{equation}\label{eq:Bgauge}
  A (x)[v] = \frac{1}{2} B[x, v].
\end{equation}
As \(\mathcal{I}_{\Tilde{A}} (\Tilde{u}) = \mathcal{I}_A (u)\) when \eqref{eq:gauge} holds, the precise choice \eqref{eq:Bgauge} allows to define the \emph{ground-energy function} \(\mathcal{E} : \bigwedge^2 \Rset^N \to \R\) by
\begin{equation}\label{eq:E(B)}
   \mathcal{E} (B) = \mathcal{E} (d A) \defeq \inf \, \bigl\{\mathcal{I}_{A}(v) \st v \in H^1_A (\R^N, \mathbb{C}) \setminus \{0\} \, \text{ and } \,  \mathcal{I}'_A(v) = 0 \bigr\}.
\end{equation}

Because of the presence of the magnetic potential, the magnetic nonlinear Schr\"odinger equation \eqref{eqNLSEMag} is not invariant under translations in \(\mathbb{R}^N\). 
However, it is still invariant under \emph{magnetic translations with respect to the connection} $D_A$. 
For \(a \in \Rset^N\) and \(u \in H^1_A (\Rset^N,\C)\), that magnetic translation is defined by 
\begin{equation*} 
 \tau^A_a u (x) = \e^{-i A (a) [x]} u (x - a).
\end{equation*}
This definition depends on the gauge fixing made above. This magnetic translation commutes with the covariant derivative \(D_A\), i.e., \(D_A \circ \tau^A_a = \tau^A_a \circ D_A\).
Together with multiplications by complex numbers in the unit circle, the magnetic translations form a Lie group. This will be explained in more details in \S \ref{sbsec:magnetic-translations}.

Our starting point to study the properties of the groundstates is to establish that, when the magnetic field \(dA\) is sufficiently small, then the groundstate of  \eqref{eqNLSEMag} is unique up to the action of the Lie group that we have described. 

\begin{theorem}[Uniqueness up to magnetic translations and multiplications by a complex phase of groundstates]%
\label{theoremUniqueness}
For every \(N \ge 2\) and \(p \in (2, \frac{2N}{N - 2})\), there exists $\varepsilon > 0$ such that if $A \in \Lin(\R^N,\bigwedge^1 \R^N)$ satisfies $\abs{d A} \le \varepsilon$, if \(u\) and \(v\) are solutions of \eqref{eqNLSEMag} satisfying \(\mathcal{I}_A (u) \le \mathcal{E} (0) + \varepsilon\) and if \(\mathcal{I}_A (v) \le \mathcal{E} (0) + \varepsilon\), then \(u = \e^{i\theta} \tau^A_a v\) for some \(a \in \R^N\) and \(\theta \in \R\).
\end{theorem}

The main idea of the proof of Theorem~\ref{theoremUniqueness} in \S \ref{section:uniquenessandsymmetry}
is to take advantage of the well-known uniqueness and non-degeneracy of the 
solutions of \eqref{eqNLSEMag} under a vanishing magnetic field \(A = 0\), see e.g.\ \citelist{\cite{Kwong1989}\cite{Weinstein1985}},
and to extend the uniqueness by an implicit function argument.
The main difficulty in this proof consists in the fact that the natural function space \(H^1_A (\Rset^N,\C)\) for the functional
\(\mathcal{I}_A\) \emph{depends} on the magnetic field: the norm and the elements of the space differ in general for different magnetic fields.
This rules out a straightforward application of classical implicit function theory.
Instead, we prove the uniqueness by relying on the arguments of the uniqueness part of the 
proof of the implicit theorem. Those do not rely on the completeness of the function space and spares us with the study of completeness across the scale of the spaces \(H^1_A (\Rset^N,\C)\).

A first consequence of Theorem~\ref{theoremUniqueness} is that the solutions inherit the symmetries of the magnetic potential in a sense explained below.

\begin{theorem}[Symmetry and monotonicity of groundstates]  \label{theoremSymmetry}
Let \(N \ge 2\), \(p \in (2, \frac{2N}{N - 2})\) and $\varepsilon > 0$ be as in Theorem~\ref{theoremUniqueness}. If $A \in \Lin(\R^N,\bigwedge^1 \R^N)$ is skew-symmetric and satisfies $\abs{d A} \le \varepsilon$ and if $u$ is a solution of \eqref{eqNLSEMag} such that  $\mathcal{I}_{A} (u) \le \mathcal{E} (0) +\varepsilon$, then there exists $a \in \R^N$ such that for every linear isometry $R$ of $\R^N$ satisfying \(\abs{A \circ R}^2 = \abs{A}^2\), one has 
\[
  u (R (x + a) - a) = \e^{iA(a)[R(x + a)- (x + a)]} u(x).
\]
Moreover, the function \(u\) is nondecreasing along any ray starting from the point \(a\).
\end{theorem}

Since the magnetic semilinear Schr\"odinger equation \eqref{eqNLSEMag} is invariant under magnetic translations, see \S \ref{sbsec:magnetic-translations},
Theorem~\ref{theoremSymmetry} implies the existence of a unique groundstate \(u\) such that its conclusion holds with \(a = 0\), that is, for every linear isometry $R$ of $\R^N$ such that \(\abs{A \circ R}^2 = \abs{A}^2\), one has \(u \circ R = u\).

Alternatively, Theorem~\ref{theoremSymmetry} states that a groundstate can be translated in such a way that it only depends monotonically on the norms of the projections on the eigenspaces of the quadratic form \(\abs{A}^2\). Also, because of the antisymmetric structure of \(A\), the group of isometries such that \(\abs{A \circ R}^2 = \abs{A}^2\) can be written, up to an isometry of the Euclidean space, as a product of orthogonal groups \(O (2n_1) \times O (2 n_2) \times \dotsb \times O (2 n_k) \times O (N {- 2n_1} {- 2 n_2} - \dotsb - 2n_k)\), with \(n_1, n_2, \dotsc, n_k \in \N\); when \(N = 3\) and \(A \ne 0\), it is always of the form \(O (2) \times O (1)\), corresponding to a decomposition in the transversal and longitudinal directions with respect to the magnetic field.

Whereas this answers positively the question about the symmetry of groundstates of the magnetic nonlinear Schr\"odinger equation \eqref{eqNLSEMag} when the magnetic field \(\abs{dA}\) is small, the question about the symmetry of groundstates for an arbitrary large magnetic field \(\abs{dA}\) remains completely open. In the planar case \(N = 2\), when the magnetic field \(\abs{dA}\) is small, the groundstates of \eqref{eqNLSEMag} correspond to the groundstates of the decoupled equation \eqref{eq:neweq}, which are non-degenerate \cite[Section 7.3]{ShiojiWatanabe2016}. This implies that  no symmetry breaking can appear by bifurcation from the radial groundstate.

In order to prove Theorem~\ref{theoremSymmetry}, we first exploit the uniqueness of the solution to prove symmetry with respect to a large subgroup of symmetries. 
Next, we note that in view of this partial symmetry, the solutions can be viewed as groundstates 
of a nonlinear Schr\"odinger equation without a magnetic field and with a quadratic electric potential \((1 + \abs{A}^2)\) (a nonlinear harmonic oscillator) and we deduce the symmetry and the monotonicity by applying classical tools for such problems.

In general, the groundstates of the nonlinear Schr\"odinger equation in the absence of a magnetic field are known to decay exponentially to \(0\) at infinity. 
By the Kato inequality, if \(u\) is a solution to the magnetic nonlinear Schr\"odinger equation \eqref{eqNLSEMag}, then its modulus \(\abs{u}\) is a subsolution to the nonlinear Schr\"odinger equation without a magnetic field and decays thus at least exponentially. One can in fact expect a better decay at infinity.
In the two-dimensional case \(N = 2\), the solutions have a Gaussian decay, see e.g.\ \citelist{\cite{Erdos1996}\cite{Shirai2008}}, similarly to Landau levels in a symmetric gauge.

Our next result is that in any dimension and for a small magnetic field, 
solutions have an improved decay rate that can be related to an exterior problem 
with a quadratic potential $(1 + |A|^2)$.

\begin{theorem}[Asymptotics of groundstates at infinity]%
\label{thm:asymptotic}%
Let \(N \ge 2\), \(p \in (2, \frac{2N}{N - 2})\) and $\varepsilon > 0$ be as in Theorem~\ref{theoremUniqueness}. 
If $A \in \Lin(\R^N,\bigwedge^1 \R^N)$ is skew-symmetric and satisfies $\abs{d A} \le \varepsilon$ 
and if $u$ is a solution of \eqref{eqNLSEMag} such that $\mathcal{I}_{A} (u) \le \mathcal{E} (0) +\varepsilon$,
then there exists \(a \in \Rset^N\) such that 
\[
 0 < \liminf_{\abs{x} \to \infty} \frac{\abs{u (x)}}{v (x - a)} 
 \le \limsup_{\abs{x} \to \infty} \frac{\abs{u (x)}}{v (x - a)} < \infty,
\]
where \(v \in C^2 (\Rset^N \setminus B_R)\) is a positive solution to 
\[
\left\{
\begin{aligned}
 -\Delta v (x) + (1 + \abs{A(x)}^2) v(x) &= 0 &&\text{for \(x \in \Rset^N \setminus B_R\)},\\
 v (x) & \to 0 & & \text{as \(x \to \infty\)}.
\end{aligned}
\right.
\]
\end{theorem}

In particular, when \(N = 2\), if $A (x)[v] = \frac{B}{2} x \wedge v$, $\abs{B} \le \varepsilon$ ($\varepsilon > 0$ given in Theorem~\ref{theoremUniqueness}) and if $u$ is a solution of \eqref{eqNLSEMag} such that $\mathcal{I}_{A} (u) \le \mathcal{E} (0) + \varepsilon$,
then there exist \(a \in \Rset^2\) and \(c \in \C \setminus \{0\}\) such that, as \(\abs {x} \to \infty\),
\[  
u (x)
  = \frac{\exp \bigl(- \frac{\abs{B}\,\abs{x - a}^2}{4}\bigr)}{\abs{x - a}^{\frac{1}{2}} \bigl(1 + \abs{B} \, \abs{x - a}\bigr)^{\frac{1}{2}+\frac{1}{\abs{B}}}} \bigl(c + o (1)\bigr).
\]
This refines for a constant magnetic field and a vanishing electric field the Gaussian decay that was already known in \citelist{\cite{Erdos1996}\cite{Shirai2008}}. The Gaussian decay is reminiscent of the decay of Landau states. 

In higher dimensions, the linear problem is in general anisotropic and we do not hope
having such an explicit expression of the asymptotics. Theorem~\ref{thm:asymptotic} allows however 
to obtain some Gaussian asymptotic \emph{upper bounds}. For example in the three-dimensional case \(N = 3\), we have for any small magnetic field \(B \in \Rset^3 \simeq \bigwedge^2 \Rset^3\),
\begin{equation}
\label{eq3dDecay}
  u (x)  = O \biggl( \exp \Bigl(- \frac{\abs{B \times (x - a)}^2}{4 \abs{B}} \Bigr)\biggr),
\end{equation}
where the typical Gaussian decay of Landau states can again be recognized.
These estimates are far from optimal: they do not enforce any decay in the \(B\)-direction whereas it follows also from Theorem~\ref{thm:asymptotic} that \(u\) decays at least exponentially in all the directions. Theorem~\ref{thm:asymptotic} follows by looking at the solutions of \eqref{eqNLSEMag} as solutions to a modified problem without magnetic field and analyzing the decay with comparison arguments. 

The last question that we consider is the dependence of the ground-energy function \(\mathcal{E}\) (see \eqref{eq:E(B)}) on the magnetic field \(B\) and in particular its monotonicity.

\begin{theorem} \label{thm:energy}
If \(N \ge 2\) and \(p \in (2, \frac{2N}{N - 2})\), then the ground-energy function \(\mathcal{E}\) is continuously differentiable and convex in a neighbourhood of \(0\),
where it achieves a global minimum.
Moreover, 
\[
 \mathcal{E} (B) = \mathcal{E} (0) + \frac{\abs{B}^2}{4 N} \int_{\R^N} \abs{x}^2 \abs{u_0 (x)}^2\dif x + o (\abs{B}^2),
\]
where $u_0$ is a radial groundstate of \eqref{eqNLSEMag} with $A = 0$.
\end{theorem}

While it is clear from the diamagnetic inequality (see \eqref{eqDiamagnetic} below) that \(\mathcal{E} (t B) \ge \mathcal{E} (0)\), the monotonicity is much more subtle. Theorem~\ref{thm:energy} shows this monotonicity for a small field. Indeed, by convexity we have 
$\mathcal{E} (t B) \le (1- t) \mathcal{E} (0) + t \mathcal{E} (B)$
and therefore 
$\mathcal{E} (t B) \le \mathcal{E} (B)$
for all $t\in[0,1]$.

The asymptotic expansion in Theorem~\ref{thm:energy} also shows that for small magnetic fields in large dimensions, that is \(N \ge 4\), the different components of the magnetic field \(B\) simply add up their contributions.

In a work about the semiclassical limit of nonlinear Schr\"odinger equation in the strong magnetic field r\'egime, it had been shown that the ground-energy was subdifferentiable and that it had a subdifferential that was allowing to recover the Lorentz electromagnetic force acting on a magnetic dipole \cite{DiCosmoVanSchaftingen2015}*{Proposition 3.5} (see also \citelist{\cite{FournaisLeTreustRaymondVanSchaftingen}\cite{FournaisRaymond}}).
Theorem~\ref{thm:energy} reinforces this analysis.

We prove Theorem~\ref{thm:energy} by applying a direct argument to show the differentiability of the solutions with respect to variations of the magnetic field.

\smallbreak

The article is organized as follows. In \S \ref{section:preliminaries}, we recall some definitions about the magnetic Sobolev spaces, for which we prove some general theorems. We also give in more details the definition of the groundstates of \eqref{eqNLSEMag}, with and without magnetic field, and we recall some of their known properties. In \S \ref{section:continuitygroundstateenergy}, we prove the continuity of the ground-energy function $\mathcal{E}$.
The uniqueness (Theorem~\ref{theoremUniqueness}) is proved in \S \ref{section:uniquenessandsymmetry} and the symmetry (Theorem~\ref{theoremSymmetry}) in \S \ref{section:symmetryof}. 
The asymptotics of groundstates (Theorem~\ref{thm:asymptotic}) are studied in \S \ref{section:asymptotic} whereas the properties of the function \(\mathcal{E}\) in a neighbourhood of \(0\) are studied  in \S \ref{sec:diff}.

\section{Preliminaries}  \label{section:preliminaries}

\subsection{Magnetic Sobolev spaces}

In this section we begin by reviewing the definitions of covariant derivative, magnetic Sobolev spaces, diamagnetic inequality and magnetic translations.
We then study the problem of convergence of sequences of functions taken in varying
magnetic Sobolev spaces.

\subsubsection{Definition of magnetic Sobolev spaces}
Magnetic Sobolev spaces are a natural framework for the magnetic semilinear Schr\"odinger equation \eqref{eqNLSEMag}.
If \(A \in L^2_{\mathrm{loc}} (\Rset^N)\), which is the case for $A \in \Lin(\mathbb{R}^N, \bigwedge^1 \R^N)$, for \(u \in W^{1, 1}_{\mathrm{loc}} (\Rset^N, \C)\), the covariant derivative is  defined by 
\[
  D_A u \defeq Du + i  Au \, : \,  \Rset^N \to \Lin (\Rset^N, \C) \simeq \C \otimes \textstyle\bigwedge^1 \Rset^N.
\]
We define for $F : \mathbb{R}^N \rightarrow \mathbb{C}^N$ the \emph{covariant divergence} by
\[
 \Div_A F \defeq \Div F + i A[F],
\]
that is, for every \(x \in \Rset^N\) and \(z \in \C\),
\[
 \scalprod{z}{\Div_A F (x)} = \scalprod{z}{\Div  F(x)}  + \scalprod{z}{i A (x)[F (x)]}.
\]
Here and in the sequel $\scalprod{\cdot}{\cdot}$ denotes the canonical \emph{real scalar product} of vectors in $\C$ (on the left-hand side) and in $\Lin(\R^N,\C)$ (on the right-hand side).
We note that \(D_A u \in L^1_{\mathrm{loc}} (\Rset^N)\) is characterized by the fact that, for every test function \(\varphi \in C^1_c (\Rset^N, \C^N)\), the following integration by parts formula is satisfied
\begin{equation}
\label{eqIntegrationParts}
  \int_{\Rset^N} \scalprod{D_A u}{\varphi} = - \int_{\Rset^N} \scalprod{u}{\Div_A \varphi}.
\end{equation}
The \emph{magnetic Sobolev space} \(H^1_A (\Rset^N,\C)\), defined as
\begin{equation*}
 H^1_{A}(\R^N, \mathbb{C}) \defeq \bigl\{u \in L^2(\R^N, \mathbb{C}) \, : \,  D_A u \in L^2(\R^N) \bigr\},
\end{equation*}
is a Hilbert space endowed with the Euclidean norm
\begin{equation*}
\norm{u}_{H^1_A (\R^N, \mathbb{C})}^2 = \int_{\R^N} \abs{D_A u}^2 + \abs{u}^2,
\end{equation*}
deriving from the \emph{real scalar product}
\[
 \scalprod{u}{v}_{H^1_A (\Rset^N,\C)} = \int_{\R^N} \scalprod{D_A u}{D_A v} + \scalprod{u}{v}.
\]
Finally, the space of compactly supported smooth functions $C^{\infty}_c(\R^N,\C)$ is dense in the magnetic Sobolev space $H^1_A(\R^N,\C)$, see e.g. \citelist{\cite{LiebLoss}*{Theorem 7.22}\cite{EstebanLions1999}*{Proposition 2.1}}.

In the following sections, we will be interested in the case where $A \in \Lin(\R^N, \bigwedge^1 \R^N)$, that is \(A\) is a linear map from \(\R^N\) to linear forms on \(\R^N\). In this case \(A\) is continuous, and thus in particular is in \(L^2_{\mathrm{loc}} (\R^N, \bigwedge^1 \R^N)\).

\subsubsection{Diamagnetic inequality}

We first recall that the connexion $D_A$ is compatible with the Euclidean norm on $\C$. Indeed, if $A \in L^2_\mathrm{loc} (\R^N)$ and $u \in H^1_A (\R^N, \mathbb{C})$, then  $u \in W^{1, 1}_\mathrm{loc} (\R^N, \mathbb{C})$, and by the chain rule for vector-valued functions \cite{AmbrosioDalMaso1990},
\begin{equation} \label{eqDerivativeModulus}
  D \abs{u} = \scalprod{\sign (u)}{D_A u},
\end{equation}
where the sign of a complex number $z \in \C$ is defined by 
\begin{equation*}
 \sign (z) = \left\{ \begin{aligned} &  \frac{z}{\abs{z}} \quad & z \ne 0,\\
                     &   0  \quad & z = 0.
             \end{aligned} \right.
\end{equation*}
In particular the identity \eqref{eqDerivativeModulus} implies the diamagnetic inequality
\begin{equation}\label{eqDiamagnetic}
  \abs{D \abs{u}} \le \abs{D_A u},
\end{equation}
with equality if and only if $D_A u = \sign (u) D\abs{u}$, see for example \cite{LiebLoss}*{theorem 7.21}.

\subsubsection{Magnetic translations}\label{sbsec:magnetic-translations}
The magnetic translations correspond to a parallel transport with respect to the connection \(D_A\).
If \(A \in \Lin(\R^N, \bigwedge^1 \R^N)\) is skew-symmetric, then, for every \(x, v \in \Rset^N\), \(A (x)[v] = - A (v)[x]\), and if \(u : \Rset^N \to \C\), we have 
\[
 \tau^A_a u (x) \defeq \e^{-iA (a)[x]} u (x - a).
\]

The magnetic translations are compatible with the connection, i.e., for every \(u \in H^1_A (\Rset^N,\C)\), Leibnitz's rule implies
\[
\begin{split}
  D_A (\tau^A_a v) (x) & = \e^{-i A (a) [x]}\bigl(D v (x - a) - i A (a) v (x - a) + i A (x) v (x - a) \bigr)\\
                       & = \e^{-i A (a) [x]} \bigl(D v (x - a) + i A (x - a) v (x - a)\bigr) \\
                       & = \tau^A_a D_A v (x),
\end{split}
\]
for every \(x \in \Rset^N\), so that 
\[
 D_A \circ \tau^A_a = \tau^A_a \circ D_A .
\]
We observe that in general, magnetic translations do not commute.
Indeed, one has 
\[
 \tau^A_b \tau^A_a u (x) = \e^{-iA(b)[x]-iA(a)[x - b]} u (x - a - b)
 = \e^{i A (a)[b]} \tau^A_{a + b} u (x - (a + b)), 
\]
for each \(x \in \Rset^N\) and therefore 
\begin{equation} 
  \label{eq:composition-translations}
  \tau^A_b \circ \tau^A_a = \e^{i A(a) [b]} \tau^A_{a+b}.
\end{equation}
If we consider the space \(\R^N \times \R\), endowed with the product \(\star\) defined for \((a, t), (b, s) \in \R^N \times \R\) by
\[
  (b, s) \star (a, t) = (a + b, t + s + i A (a)[b]), 
\]
we see that \((a, t) \mapsto \e^{i t} \tau^A_a \) defines a group action whose kernel is \(\{0\} \times 2 \pi \Z\). This group is isomorphic to \(\mathbb{H}^k \times \Rset^{N - 2k}\), where \(\mathbb{H}^k\) is the \(2 k + 1\)--dimensional \(k\)-th order Heisenberg group and \(2 k\) is the rank of the matrix \(A\).
In particular, for every $s, t \in \mathbb{R}$,
\begin{equation*}
\tau^A_{s a} \circ \tau^A_{ta} = \tau^A_{(t+s)a},
\end{equation*}
so that the translations in the same direction form a group.

\subsubsection{Convergence across magnetic Sobolev spaces} \label{section:magneticspaces}

In this section, we develop some counterparts of classical results in Sobolev spaces to study convergence of sequences of maps belonging to different magnetic Sobolev spaces. The common assumption of the next statements is that the sequence of vector potentials converges e.g.\ in $L^2_{\mathrm{loc}}(\mathbb{R}^N)$.

\begin{lemma}[Weak closure across magnetic Sobolev spaces]\label{lemma:WeakClosure}
Assume that for every \(n \in \N\),  $u_n \in H^1_{A_n} (\R^N,\mathbb{C})$, where $(A_n)_{n \in \N}$ is a sequence in $L^2_\mathrm{loc} (\R^N, \bigwedge^1 \R^N)$. If $A_n \to A$ strongly in $L^2_\mathrm{loc}(\R^N)$, $u_n \weakto u$ weakly in $L^2 (\R^N)$, and $D_{A_n} u_n \weakto g$ weakly in $L^2 (\R^N)$  as $n \to \infty$, then $u \in H^1_A (\R^N, \mathbb{C})$ and $D_A u = g$. 
\end{lemma}

\begin{proof}
By definition of the weak covariant derivative $D_{A_n} u_n$, for every $\varphi \in C^\infty_c (\R^N, \mathbb{C}^N)$ and $n \in \N$, in view of \eqref{eqIntegrationParts}, we have
\[
 \int_{\R^N} \scalprod{D_{A_n} u_n}{\varphi} = - \int_{\R^N} \scalprod{u_n}{\Div_{A_n} \varphi }.
\]
Since $\Div_{A_n} \varphi \to \Div_{A} \varphi$ as \(n \to \infty\) in $L^2 (\R^N,\mathbb{C})$, because of the strong convergence of $A_n$ to $A$ in $L^2_{\mathrm{loc}}(\mathbb{R}^N)$, and using the facts that $u_n \weakto u$ and $D_{A_n} u_n \weakto g$ as \(n \to \infty\) in $L^2(\mathbb{R}^N)$, we conclude that 
\[
\int_{\R^N} \scalprod{g}{\varphi} = - \int_{\R^N} \scalprod{u}{\Div_A \varphi} = \int_{\mathbb{R}^N} \scalprod{D_A u}{\varphi},
\]
where the last equality follows again from \eqref{eqIntegrationParts}. This shows that $g = D_A u = D u + i A u$, using \cite{EstebanLions1999}*{Proposition 2.1}.
\end{proof}

The next lemma deals with bounded sequences in distinct magnetic Sobolev spaces.

\begin{lemma}[Weak sequential compactness across magnetic Sobolev spaces] \label{lemma:WeakCompactness}
Assume that for every \(n \in \N\),  $u_n \in H^1_{A_n} (\R^N,\mathbb{C})$, where $(A_n)_{n \in \N}$ is a sequence in $L^2_\mathrm{loc} (\R^N, \bigwedge^1 \R^N)$. If $A_n \to A$ strongly as \(n \to \infty\) in $L^2_\mathrm{loc}(\R^N)$ and
\[
\liminf_{n \to \infty} \int_{\R^N} \abs{D_{A_n} u_n}^2 + \abs{u_n}^2 <\infty,
\]
there exist $u \in H^1_A (\R^N, \mathbb{C})$ and a subsequence $(n_\ell)_{\ell \in \N}$ such that 
$u_{n_\ell} \weakto u$ weakly in $L^2( \mathbb{R}^N)$ and $D_{A_{n_\ell}} u_{n_\ell} \weakto D_A u$ weakly in $L^2 (\R^N)$ as $\ell \to \infty$.
\end{lemma}

\begin{proof}
This follows from the standard weak sequential compactness criterion in Hilbert spaces and Lemma~\ref{lemma:WeakClosure}.
\end{proof}

Weakly converging sequences across Sobolev spaces converge strongly in Lebesgue spaces on compact subsets.

\begin{lemma}[Rellich's Theorem across magnetic spaces] \label{lemmaRellich}
Assume that for every \(n \in \N\), $u_n \in H^1_{A_n} (\R^N,\mathbb{C})$, where $(A_n)_{n \in \N}$ is a sequence in $L^2_\mathrm{loc} (\R^N, \bigwedge^1 \R^N)$. 
If $A_n \to A$ strongly in $L^2_\mathrm{loc}(\R^N)$, and $u_n \weakto u$ weakly in $L^2_\mathrm{loc} (\R^N)$, $D_{A_n} u_n \weakto D_A u$ weakly in $L^2_\mathrm{loc} (\R^N)$, then $u_n \rightarrow u$ strongly in $L^p_\mathrm{loc} (\R^N)$ as $n \to \infty$, for $1 \leq p < \frac{2N}{N-2}$.
\end{lemma}

\begin{proof}
Let $R > 0$. First, we will prove the strong convergence of $u_n$ to $u$ in $L^p(B_R)$, for every $p \in [1, \frac{N}{N-1})$. Since $A_n \to A$ strongly in $L^2(B_R)$ and $u_n \weakto u$ weakly in $L^2 (B_R)$ as $n \to \infty$, the sequence
$(A_n u_n)_{n \in \N}$ is bounded in $L^1 (B_R)$. Thus, since $D_{A_n} u_n \weakto D_A u$ weakly in $L^2(B_R)$, we have that $(D u_n)_{n \in \N}$ is also bounded in $L^1 (B_R)$. Therefore, Rellich's compactness theorem in $W^{1, 1} (B_R)$ tells us that $u_n \to u$ strongly in $L^p (B_R)$, for every $1 \leq p < \frac{N}{N-1}$, as $n \to \infty$.

Next, if $1 \leq q \leq \frac{2N}{N-2}$, there exists a constant $C > 0$ such that for every $n \in \N$, in view of the diamagnetic inequality \eqref{eqDiamagnetic} and of the classical Sobolev embedding,  
\[
  \Bigl(\int_{B_R} \abs{u_n}^q\Bigr)^\frac{2}{q} \le C \int_{B_R} \abs{D \abs{u_n}}^2 + \abs{u_n}^2
  \le C \int_{B_R} \abs{D_{A_n} u_n}^2 + \abs{u_n}^2.
\]
Then, $(u_n)_{n \in \N}$ is bounded in $L^q (B_R)$, for $q \in [1, \frac{2N}{N-2}]$.
By a standard interpolation argument, we conclude that $u_n \to u$ as \(n \to \infty\) in $L^p (B_R)$, for every $p \in [1, \frac{2N}{N-2})$.
\end{proof}

In the case of the strong convergence of a sequence of maps \((u_n)\)across Sobolev spaces, the moduli converge also strongly in a Sobolev space.

\begin{lemma}[Continuity of the modulus across Sobolev spaces]\label{lemmaModulus}
Assume that for every \(n \in \N\), $u_n \in H^1_{A_n} (\R^N,\mathbb{C})$, where $(A_n)_{n \in \N}$ is a sequence in $L^2_\mathrm{loc} (\R^N, \bigwedge^1 \R^N)$. 
If $A_n \to A$ strongly in $L^2_\mathrm{loc}(\R^N)$, and $u_n \to u$ strongly in $L^2 (\R^N)$, $D_{A_n} u_n \to D_A u$ strongly in 
$L^2 (\R^N)$ as $n \to \infty$, then $\abs{u_n} \to \abs{u}$ strongly in $H^1 (\R^N)$ as $n \to \infty$.
\end{lemma}

\begin{proof}
It is clear that $\abs{u_n} \to \abs{u}$ in $L^2 (\R^N)$. For the derivative, observe that, by \eqref{eqDerivativeModulus},
\[
 D \abs{u_n} = \scalprod{\sign (u_n)}{D_{A_n} u_n}.
\]
Since $u_n \to u$ in $L^2 (\R^N)$ and $D_A u = 0$ almost everywhere on $u^{-1}(\{0\})$, we conclude that 
\[
 D \abs{u_n} = \scalprod{\sign (u_n)}{D_{A_n} u_n}  \to \scalprod{\sign (u)}{D_{A} u} = D \abs{u}
\]
in measure. By the diamagnetic inequality \eqref{eqDiamagnetic}, we know that
\[
 \abs{D\abs{u_n}}^2 \le \abs{D_{A_n} u_n}^2. 
\]
This last inequality together with the strong convergence of $(D_{A_n} u_n)_{n \in \N}$ in $L^2 (\R^N)$ implies that $D \abs{u_n} \to D \abs{u}$ in $L^2 (\R^N)$ as $n \to \infty$ by Lebesgue's dominated convergence theorem.
\end{proof}

The next lemma shows that the strong convergence across magnetic Sobolev spaces implies the convergence in the Lebesgue spaces in which the magnetic Sobolev spaces are embedded. When $A_n = A$ for all $n\in\N$, this follows from the classical scalar Sobolev embedding and the diamagnetic inequality, see \cite{Shirai2008}*{Lemma 3.1}.

\begin{lemma}[Continuous Sobolev embedding across magnetic Sobolev spaces] \label{lemmaSobolev} 
Let $(u_n)_{n \in \N}$ be a sequence in $H^1_{A_n} (\R^N,\mathbb{C})$, where $(A_n)_{n \in \N}$ is a sequence in $L^2_\mathrm{loc} (\R^N, \bigwedge^1 \R^N)$.
If $A_n \to A$ strongly in $L^2_\mathrm{loc}(\R^N)$, and $u_n \to u$ strongly in $L^2 (\R^N)$, $D_{A_n} u_n \to D_A u$ strongly in $L^2 (\R^N)$, then $u_n \to u$ strongly in $L^p (\R^N)$ for $2 \leq p \leq \frac{2N}{N-2}$, as $n \to \infty$.
\end{lemma}

\begin{proof}
First observe that by Lemma~\ref{lemmaRellich}, $\abs{u_n - u}^p \to 0$ locally in measure as $n \to \infty$. In view of Lemma~\ref{lemmaModulus}, $\abs{u_n} \to \abs{u}$ in $H^1 (\R^N)$ and thus, by the Sobolev embeddings, $\abs{u_n} \to \abs{u}$ in $L^p (\R^N)$, for $2 \leq p \leq \frac{2N}{N-2}$, and $\abs{u_n}^p \to \abs{u}^p$ in $L^1 (\R^N)$ as $n \to \infty$. Moreover, we have that
\[
 \abs{u_n - u}^p \le 2^p \bigl(\abs{u_n}^p + \abs{u}^p\bigr).
\]
Applying Lebesgue's dominated convergence theorem, we infer that $\abs{u_n - u}^p \to 0$ in $L^1 (\R^N)$ and we therefore conclude that $u_n \to u$ in $L^p (\R^N)$, as $n \to \infty$.
\end{proof}

\subsection{Groundstates}

Here we recall the known properties of the groundstates of the nonlinear Schr\"odinger equation \eqref{eqNLSEMag}, with or without magnetic potential.

\subsubsection{Existence of groundstates and characterization of the ground-energy}
A function $u \in H^1_{A}(\R^N, \mathbb{C})$ is a weak solution of the nonlinear Schr\"odinger equation \eqref{eqNLSEMag} if, for every $v \in H^1_A (\R^N, \mathbb{C})$,
\[
 \int_{\R^N} \scalprod{D_A u}{D_A v} + \scalprod{u}{v} = \int_{\R^N} \abs{u}^{p - 2}\scalprod{u}{v}.
\]
This follows from the integration by parts formula \eqref{eqIntegrationParts} and from the identity,
\[ 
  - \Delta_A u =   -\Div_A \nabla_A u,
\]
where the \emph{magnetic gradient} \(\nabla_A u : \Rset^N \to \C^N\) is defined so that for every \(x \in \Rset^N\), \(v \in \Rset^N\) and \(z \in \C\)
\[
  \scalprod{\nabla_A u(x)}{zv} = \scalprod{z}{D_A u (x)[v]},
\]
where the scalar product on the left-hand side is the canonical real scalar product on \(\C^N\) and the scalar product on the right-hand side is the canonical scalar product on \(\C\).
Weak solutions of the nonlinear Schr\"odinger equation \eqref{eqNLSEMag} are also critical points of the functional \(\mathcal{I}_{A}\) defined for each \(u \in H^1_A (\R^N)\) by 
\begin{equation*}
\mathcal{I}_{A}(u) \defeq \frac{1}{2} \int_{\R^N} \bigl( \abs{D_A u}^2 + \abs{u}^2 \bigr) - \frac{1}{p} \int_{\R^N} \abs{u}^p.
\end{equation*}
We recall that the ground-energy function $\mathcal{E} : \bigwedge^2 (\Rset^N) \to \R$ is defined by \eqref{eq:E(B)}.
Since for any constant $B \in \bigwedge^2 \R^N$ there is a unique skew-symmetric $A \in \Lin(\R^N, \bigwedge^1 \R^N)$ such that $d A = B$, the function \(\mathcal{E}\) is well-defined. 

The function $u \in H^1_{A}(\R^N, \mathbb{C})$ is a \emph{groundstate or a least-energy solution} of the magnetic nonlinear Schr\"odinger equation \eqref{eqNLSEMag} if \(u\) is a weak solution of the equation \eqref{eqNLSEMag} such that
\begin{equation*}
\mathcal{I}_{A}(u)  = \mathcal{E} (dA).
\end{equation*}
The next lemma is standard, we sketch the proof for completeness. 

\begin{lemma}%
[Existence and characterization of groundstates]%
\label{lemmaExistenceGroundstate}
For every magnetic potential $A \in \Lin(\R^N, \bigwedge^1 \R^N)$, there exists $u \in H^1_A(\mathbb{R}^N, \mathbb{C}) \setminus \{0\}$ such that 
\[
 \mathcal{I}_A (u) = \mathcal{E} (dA) \text{ and }\  \mathcal{I}_A' (u) = 0.
\]
Moreover, 
\[
 \mathcal{E} (dA) = \bigl(\tfrac{1}{2} - \tfrac{1}{p}\bigr) \inf_{v \in H^1_A (\R^N, \mathbb{C}) \setminus \{0\}} \mathcal{Q}_A (v)^\frac{p}{p - 2},
\]
where the functional $\mathcal{Q}_A : H^1_A (\R^N, \mathbb{C}) \setminus \{0\} \to \R$ is defined by
\[
 \mathcal{Q}_A (v) \defeq \frac{\displaystyle \int_{\R^N} \abs{D_A v}^2 + \abs{v}^2}{\displaystyle \Bigl(\int_{\R^N} \abs{v}^p \Bigr)^\frac{2}{p}}.
\]
\end{lemma}

\begin{proof}
The existence of a minimizer of $\mathcal{Q}_A$ has been proved by M. Esteban and P.-L. Lions \cite{EstebanLions1999}*{theorem 3.1}. By homogeneity, this minimizer can be chosen to satisfy the condition $\mathcal{I}_A' (v) = 0$, or equivalently, the magnetic nonlinear Schr\"odinger equation \eqref{eqNLSEMag}.

Finally, if $v \in H^1_A( \mathbb{R}^N, \mathbb{C})$ satisfies $\mathcal{I}'_A (v) = 0$, then 
\[
 \mathcal{Q}_A (v) = \bigl(\tfrac{2p}{p - 2} \mathcal{I}_A (v) \bigr)^{\frac{p - 2}{p}}.
\]
We can then conclude that
\[
 \mathcal{E} (dA) = \bigl(\tfrac{1}{2} - \tfrac{1}{p}\bigr) \inf_{v \in H^1_A (\R^N, \mathbb{C})} \mathcal{Q}_A (v)^\frac{p}{p - 2}.\qedhere
\]
\end{proof}

The following lemma gives us some basic properties of the ground-energy. First, for any constant $B \in \bigwedge^2 \R^N$, $A \in \bigwedge^1\R^N$, and for any linear isometry $R \in \Lin (\R^N, \R^N)$, we recall that the pull-back of $B$ and $A$ by $R$ are given respectively by
\[
 R_\# B[v, w] = B[R(v), R (w)],
\]
and
\[
R_\# A(x)[v] = A(R(x))[R(v)],
\]
for $v,w\in \R^N$.

\begin{lemma}[Invariance under isometries of \(\mathcal{E}\)]
\label{lemma:propertyGSE}
Let $B \in \bigwedge^2 \R^N$ be constant.
If $R \in \Lin (\R^N , \R^N)$ is an isometry, then
\[
\mathcal{E}(R_\# B) = \mathcal{E} (B).
\]
\end{lemma}

\begin{proof}
Assume that \(A \in \Lin (\Rset^N, \bigwedge^1 \Rset^N)\) is antisymmetric and \(dA = B\). 
The invariance under isometries follows from the fact that 
\[
\begin{split}
  D_{R_\# A} (u \circ R) (x)[v] &= D u (R (x)) [R(v)] + i u (R (x)) A (R (x))[R (v)]\\
  &= (D_A u)(R (x))[R (v)] = R_\# (D_A u)(x)[v].\qedhere
\end{split}
\]
\end{proof}

\subsubsection{Groundstates without a magnetic field} \label{subsec:groundstate0}

We recall some well established results for the problem without a magnetic field
\begin{equation} \label{eq:limitproblem}
- \Delta u + u = \abs{u}^{p - 2} u , \qquad \text{ in } \mathbb{R}^N.
\end{equation}
This problem is a natural limit for \eqref{eqNLSEMag} when \(A \to 0\).
The next result states that the groundstate of \eqref{eq:limitproblem} in $H^1(\mathbb{R}^N, \mathbb{C})$ is unique up to rotations in $\C$ and translations in $\R^N$. 

\begin{proposition}[Uniqueness up to rotations in $\mathbb{C}$ and translations in $\mathbb{R}^N$] \label{propositionDiamagneticUniqueness}
If $u, v \in H^1(\R^N, \C)$ satisfy $\mathcal{I}_0 (u) = \mathcal{E} (0)$ and $\mathcal{I}^\prime_0 (u) = 0$, then there exist $\theta \in \R$ and $a \in \R^N$ such that $v= \e^{i \theta} \tau^0_a u$.
\end{proposition}

\begin{proof}
We first observe by the diamagnetic inequality \eqref{eqDerivativeModulus} that
\[
  \mathcal{I}_0 (\abs{u}) \le \mathcal{I}_0 (u) 
\]
with equality if and only if $D u = \sign (u) D \abs{u}$.
In view of the characterization of the groundstate of Lemma~\ref{lemmaExistenceGroundstate}, we have that $\mathcal{I}_0 (\abs{u}) = \mathcal{I}_0 (u) = \mathcal{E}(0)$, and therefore $D u = \sign (u) D \abs{u}$ almost everywhere in \(\Rset^N\). The function $\abs{u}$ is thus a real positive solution of the equation
\[
  -\Delta \abs{u} + \abs{u} = \abs{u}^{p - 1}.
\]
By classical regularity theory, $\abs{u} \in C^2(\R^N)$ and by the strong maximum principle $\abs{u} > 0$. Therefore, $\abs{u}^{-1} \in L^\infty_\mathrm{loc}$ so that $\sign(u) \in H^1_\mathrm{loc}(\R^N, \C)$. We can then compute its derivative  
\[
  D \sign (u) = \frac{D u - \sign (u) D \abs{u}}{\abs{u}}.
\]
Since $Du = \sign (u) D \abs{u}$, we conclude that $D \sign (u) = 0$ almost everywhere in \(\R^N\). Hence, there exists a real number $\varphi \in \R$ such that $\sign (u) = \e^{i \varphi}$ almost everywhere in \(\R^N\), so that $u = \e^{i \varphi} \abs{u}$. Similarly for the function $v$, there exists $\psi \in \R$ such that $v = \e^{i \psi} \abs{v}$. 
Since $\abs{u}$ and $\abs{v}$ are positive solutions of the semilinear problem \eqref{eqNLSEMag}, by the uniqueness up to translations of such solutions \citelist{\cite{Coffman1972}\cite{Kwong1989}\cite{McLeodSerrin1987}}, there exists $a \in \R^N$ such that $\abs{v} = \tau^0_a \abs{u}$. We have thus proved that $v = \e^{i (\psi - \varphi)} \tau^0_a u $.
\end{proof}
It clearly follows from the previous proposition that there exists a unique real, positive and radially symmetric groundstate of \eqref{eq:limitproblem}, that we denote by $u_0$. The next proposition states the non-degeneracy property due to M.\thinspace{}I.\thinspace{}Weinstein \cite{Weinstein1985} and Y.-G.\thinspace{}Oh \cite{Oh1990}.

\begin{proposition}[Non-degeneracy of the groundstates in absence of magnetic field] \label{propositionNonDegeneracyWithoutMagneticField}
Assume that $u  \in H^1 (\R^N, \C)$ satisfies $\mathcal{I}_0 (u) = \mathcal{E} (0)$ and $\mathcal{I}_0' (u) = 0$. If $w \in H^1 (\R^N, \C)$ satisfies
\begin{equation} \label{eq:NonDegeneracy}
 -\Delta w + w = \abs{u}^{p - 2} w + (p - 2) \abs{u}^{p - 4} \scalprod{u}{w} u, 
\end{equation}
then there exist $y \in \R^N$ and $\lambda \in \R$ such that 
\begin{equation} \label{eq:w}
  w = D u [y] + \lambda i u. 
\end{equation}
\end{proposition}

In particular if \(u\) is a solution of the equation \eqref{eq:limitproblem} and if \(w\) is a solution of its linearised problem \eqref{eq:NonDegeneracy} given by \eqref{eq:w}, then \(u\) and \(w\) are orthogonal in the space \(H^1 (\Rset^N, \C)\), i.e.
\begin{equation}
\label{eq:Orthogonality}
\int_{\mathbb{R}^N} \scalprod{Du}{Dw} + \scalprod{u}{w} = 0.
\end{equation}
This can be either deduced from Proposition~\ref{propositionNonDegeneracyWithoutMagneticField} or proved directly, by testing the equation \eqref{eq:limitproblem} of a groundstate $u$ on $w$ 
and \eqref{eq:NonDegeneracy} against on \(u\), namely 
\[
\int_{\mathbb{R}^N} \scalprod{Du}{Dw} + \scalprod{u}{w} = \int_{\mathbb{R}^N} \abs{u}^{p - 2} \scalprod{u}{w},
\]
and
\[
\int_{\mathbb{R}^N} \scalprod{Dw}{Du} + \scalprod{w}{u} = (p-1) \int_{\mathbb{R}^N} \abs{u}^{p - 2} \scalprod{w}{u}.
\]
Since \(p > 2\), the orthogonality identity \eqref{eq:Orthogonality} follows.

For every groundstate $u \in H^1(\mathbb{R}^N, \C)$ of \eqref{eqNLSEMag}, we can rewrite equation \eqref{eq:NonDegeneracy} as an eigenvalue equation in the following way
\[
  L_u w = \lambda w,  \qquad w \in H^{1}(\mathbb{R}^N, \C),
\]
where the operator $L_u : H^1(\mathbb{R}^N, \C) \to H^1(\mathbb{R}^N, \C)$ is given by
\begin{equation*}
L_u w \defeq \bigl( - \Delta + 1 \bigr)^{-1} \bigl( \abs{u}^{p - 2} w + (p - 2) \abs{u}^{p - 4} (u \otimes u)[ w] \bigr),
\end{equation*}
with 
\begin{equation}\label{eq:otimes}
(u \otimes u)[ w] \defeq  \scalprod{u}{w}  u.
\end{equation}

It is standard that the operator $L_u$ is compact. Indeed, it is known that the groundstates $u$ of \eqref{eq:limitproblem} decays as $|x|^{-(N-1)/2} \exp (- |x|)$, see e.g. \cite{AmbrosettiMalchiodiRuiz}*{p. 332}, so that they are in $L^q(\mathbb{R}^N)$ for every $q \ge 1$. For completeness Lemma~\ref{lemma:compact-L} below gives a more general result including the one above.

It is also standard, see e.g.\ \cite{AmbrosettiMalchiodi}*{Remark 4.2}, to check directly that the groundstate $u$ is the first eigenfunction of eigenvalue $\lambda_1(L_u) = (p-1) > 1$, while the functions $w$ given in \eqref{eq:w} are the following eigenfunctions corresponding to the eigenvalues $\lambda_i(L_u) = 1$, $i = 2, \dotsc, N+2$. Finally, $\lambda_i(L_u) < 1$ for $i > N+2$.

\section{Continuity of the ground-energy} \label{section:continuitygroundstateenergy}

In this section we study the continuity of the groundstates with respect to $A$ (or equivalently with respect to $B$ since we choose $A$ to be skew-symmetric). 

\subsection{Palais-Smale type condition across magnetic spaces}

Our crucial tool will be a Palais-Smale type condition across magnetic Sobolev spaces. We recall that we assume $A$ to be skew-symmetric.

\begin{lemma}[Palais--Smale condition across magnetic spaces] \label{lemmaPalaisSmale}
Let $(u_n)_{n \in \N}$ be a sequence in $H^1_{A_n} (\R^N,\mathbb{C})$, where $(A_n)_{n \in \N}$ is a sequence in $L^2_\mathrm{loc} (\R^N, \bigwedge^1 \R^N)$.
If $A_n \to A$ strongly in $L^2_{\mathrm{loc}}(\mathbb{R}^N)$ as $n \rightarrow \infty$, and if $u_n$ satisfies
\[
 \lim_{n \to \infty} \norm{\mathcal{I}'_{A_n} (u_n)}_{\left(H^1_{A_n} (\R^N, \mathbb{C})\right)^\prime} =  0
\ 
\text{
and }
\
\limsup_{n \to \infty} \mathcal{I}_{A_n} (u_n) < \infty,
\]
then there exist $u \in H^1_A (\R^N, \mathbb{C}) \setminus \{0\}$, $(a_{n})_{n \in \N}$ in $\mathbb{R}^N$ and a subsequence $(n_\ell)_{\ell \in \N}$ such that $\tau^{A_{n_\ell}}_{a_{n_\ell}} u_{n_\ell} \weakto u$, $D_{A_{n_\ell}} (\tau^{A_{n_\ell}}_{a_{n_\ell}} u_{n_\ell}) \weakto D_A u$ weakly in $L^2 (\R^N)$ as $\ell \to \infty$ and 
\[
 \liminf_{\ell \to \infty} \mathcal{I}_{A_{n_\ell}} (u_{n_\ell}) \ge \mathcal{I}_A (u).
\]
Moreover, if $\mathcal{I}_A (u) \ge \limsup_{\ell \to \infty} \mathcal{I}_{A_{n_\ell}} (u_{n_\ell})$, the convergences are strong in $L^2(\R^N)$.
\end{lemma}

First, we note that the convergence $A_n \to A$ in $L^2_{\mathrm{loc}}(\mathbb{R}^N)$ is equivalent to the convergence $d A_n \to d A$ in the finite-dimensional space \(\bigwedge^2 \Rset^N\).
In this statement, the standard choice of the norm on $\bigl( H^1_{A_n} (\R^N, \mathbb{C}) \bigr)'$ is essential
\begin{multline*}
 \norm{\mathcal{I}'_{A_n} (u_n)}_{\bigl( H^1_{A_n} (\R^N, \mathbb{C}) \bigr)^\prime}\\
 = \sup\, \Bigl\{ \dualprod{\mathcal{I}'_{A_n} (u_n)}{v} \st v \in H^1_{A_n} (\R^N, \mathbb{C}) \, \text{ and } \, \int_{\R^N} \abs{D_{A_n} v}^2 + \abs{v}^2 \le 1 \Bigr\},
\end{multline*}
where
\[
\dualprod{\mathcal{I}'_{A} (u)}{v} = \int_{\mathbb{R}^N} \scalprod{D_A u}{D_A v} + \scalprod{u}{v} - |u|^{p-2} \scalprod{u}{v}.
\]

\begin{proof}
The proof will be divided into four claims. 

\begin{claim} \label{claimBounded}
The sequence is bounded, i.e.
\[
 \limsup_{n \to \infty} \int_{\R^N} \abs{D_{A_n} u_n}^2 + \abs{u_n}^2 < \infty.
\]
\end{claim}

\begin{proofclaim}
By direct computation and by our assumptions we have
\[
\begin{split}
 \bigl(\tfrac{1}{2}-\tfrac{1}{p}\bigr)\int_{\R^N} \abs{D_{A_n} u_n}^2 + \abs{u_n}^2 
 & = \mathcal{I}_{A_n} (u_n) - \frac{1}{p} \dualprod{\mathcal{I}'_{A_n} (u_n)}{u_n} \\
 & = O (1) +  o \Bigl(\sqrt{\int_{\R^N} \abs{D_{A_n} u_n}^2 + \abs{u_n}^2} \Bigr).
\end{split}
\]
Since $p>2$, the conclusion follows.
\end{proofclaim}

\begin{claim} \label{claimWeakConvergence}
There exist a sequence $(a_n)_{n \in \N}$ in $\R^N$, $u \in H^1_A (\R^N, \mathbb{C}) \setminus \{0\}$ and a subsequence $(n_\ell)_{\ell \in \N}$ such that $\tau^{A_{n_\ell}}_{a_{n_\ell}} u_{n_\ell} \weakto u$ and $D_{A_{n_\ell}} (\tau^{A_{n_\ell}}_{a_{n_\ell}} u_{n_\ell}) \weakto D_A u$ weakly in $L^2 (\R^N)$ as $\ell \to \infty$. Moreover, $\mathcal{I}_A' (u) = 0$.
\end{claim}

\begin{proofclaim}
By an inequality of P.-L. Lions, see e.g. \citelist{\cite{Lions1984CC2}*{lemma I.1}\cite{Willem1996}*{lemma 1.21}\cite{MorozVanSchaftingen2013-2}*{lemma 2.3}\cite{VanSchaftingen2014}}, and by the diamagnetic inequality \eqref{eqDiamagnetic}, for every $n \in \N$,
\[
\begin{split}
  \int_{\R^N} \abs{u_n}^{p}
& \le C \Bigl(\sup_{y \in \R^N} \int_{B_1 (y)} \abs{u_n}^{p}\Bigr)^{1 - \frac{2}{p}}  \int_{\R^N} \abs{D \abs{u_n}}^2 + \abs{u_n}^2\\
& \le C \Bigl(\sup_{y \in \R^N} \int_{B_1 (y)} \abs{u_n}^{p}\Bigr)^{1 - \frac{2}{p}}  \int_{\R^N} \abs{D_{A_n} u_n}^2 + \abs{u_n}^2.
\end{split}
\]
On the other hand,
\[
 \int_{\R^N} \abs{D_{A_n} u_n}^2 + \abs{u_n}^2 = \int_{\R^N} \abs{u_n}^{p}
 + \dualprod{\mathcal{I}_{A_n}' (u_n)}{u_n} = \int_{\R^N} \abs{u_n}^{p} + o \Bigl(\sqrt{\int_{\R^N} \abs{D_{A_n} u_n}^2 + \abs{u_n}^2} \Bigr).
\]
Putting the two previous contributions together, we get
\begin{equation*}
\int_{\mathbb{R}^N} \abs{u_n}^{p} \leq C \Bigl(\sup_{y \in \R^N} \int_{B_1 (y)} \abs{u_n}^{p}\Bigr)^{1 - \frac{2}{p}} \biggl( \int_{\R^N} \abs{u_n}^{p} + o \Bigl(\sqrt{\int_{\R^N} \abs{D_{A_n} u_n}^2 + \abs{u_n}^2} \Bigr) \biggr).
\end{equation*}
Since $p > 2$, this leads to
\[
  \liminf_{n \to \infty} \Bigl(\sup_{y \in \R^N} \int_{B_1 (y)} \abs{u_n}^{p}\Bigr)^{1 - \frac{2}{p}} > 0.
\]
Therefore there exists a sequence $(a_n)_{n \in \N}$ in $\R^N$ such that 
\[
 \liminf_{n \to \infty} \int_{B_1 (a_n)} \abs{u_n}^{p} > 0.
\]
With this first result, we deduce that there is no evanescence of the sequence $u_n$, up to a translation by $a_n$. Because of the presence of the magnetic potential we cannot consider the translation alone, we have to take in consideration the magnetic translation $\tau^{A_n}_{a_n}$ compatible with the connexion $D_{A_n}$. We then have
\[
 \liminf_{n \to \infty} \int_{B_1 (0)} \abs{\tau^{A_n}_{a_n} u_n}^{p} > 0,
\]
and by Claim~\ref{claimBounded}
\[
 \liminf_{n \to \infty} \int_{\R^N} \abs{D_{A_n} (\tau^{A_n}_{a_n} u_n)}^2 + \abs{\tau^{A_n}_{a_n} u_n}^2 = \liminf_{n \to \infty} \int_{\R^N} \abs{D_{A_n} u_n}^2 + \abs{u_n}^2 <\infty.
\]
By Claim~\ref{claimBounded} and the weak compactness Lemma across magnetic spaces of Lemma~\ref{lemma:WeakCompactness}, we conclude that there exist $u \in H^1_A(\R^N,\C)$ and a subsequence $(n_\ell)_{\ell \in \N}$ such that, as \(\ell \to \infty\), $\tau^{A_{n_\ell}}_{a_{n_\ell}} u_{n_\ell} \weakto u$ and $D_{A_{n_\ell}} (\tau^{A_{n_\ell}}_{a_{n_\ell}} u_{n_\ell}) \weakto D_A u$ weakly in $L^2 (\R^N)$. Moreover, applying the compactness across magnetic Sobolev spaces of Lemma~\ref{lemmaRellich}, we infer that
\[
 \int_{B_1 (0)} \abs{u}^{p} = \lim_{\ell \to \infty} \int_{B_1 (0)} \abs{\tau^{A_{n_\ell}}_{a_{n_\ell}} u_{n_\ell}}^{p}
 \ge \liminf_{n \to \infty} \int_{B_1 (a_n)} \abs{u_n}^{p} > 0,
\]
so that $u \not\equiv 0$.

Finally, since, for every $\varphi \in C^1_c (\R^N, \C)$, we have $D_{A_n} \varphi \to D_A \varphi$ in $L^2 (\R^N)$, we deduce from the compactness property (Lemma~\ref{lemmaRellich}) that
\[
\begin{split}
 0 & = \lim_{\ell \to \infty} \int_{\R^N} \scalprod{D_{A_{n_\ell}} (\tau^{A_{n_\ell}}_{a_{n_\ell}} u_{n_\ell})}{D_{A_{n_\ell}} \varphi} + \scalprod{\tau^{A_{n_\ell}}_{a_{n_\ell}} u_{n_\ell}}{\varphi} - \abs{\tau^{A_{n_\ell}}_{a_{n_\ell}} u_{n_\ell}}^{p - 2} \scalprod{\tau^{A_{n_\ell}}_{a_{n_\ell}} u_{n_\ell}}{\varphi}\\
 &=\int_{\R^N} \scalprod{D_{A} u}{D_{A} \varphi} + \scalprod{u}{\varphi} - \abs{u}^{p - 2} \scalprod{u}{\varphi}.
\end{split}
\]
This means that $\mathcal{I}_A' (u) = 0$.
\end{proofclaim}

\begin{claim}%
\label{claimLiminf}%
One has
\[
 \mathcal{I}_A (u) \le \liminf_{\ell \to \infty} \mathcal{I}_{A_{n_\ell}} (u_{n_\ell}).
\]
\end{claim}

\begin{proofclaim}
For every $n \in \N$,
\[
\mathcal{I}_{A_n} (u_n) = \bigl(\tfrac{1}{2} - \tfrac{1}{p}\bigr)\int_{\R^N} \abs{u_n}^p + \tfrac{1}{2} \dualprod{\mathcal{I}'_{A_n} (u_n)}{u_n}.
\]
Therefore, since by assumption
\[\norm{\mathcal{I}'_{A_n} (u_n)}_{ \left( H^1_{A_n} (\R^N, \mathbb{C}) \right)'} \to 0\]  and the sequence $(\norm{u_n}_{H^1_{A_n} (\R^N, \mathbb{C})})_{n \in \N}$ is bounded by Claim 1, we have that
\[
 \liminf_{n \to \infty} \mathcal{I}_{A_n} (u_n) = \bigl(\tfrac{1}{2} - \tfrac{1}{p}\bigr) \liminf_{n \to \infty} \int_{\R^N} \abs{u_n}^p.
\]
Moreover, by weak lower semi-continuity, we deduce that
\[
 \int_{\R^N} \abs{u}^p \leq \liminf_{\ell \to \infty} \int_{\R^N} \abs{\tau^{A_{n_\ell}}_{a_{n_\ell}} u_{n_\ell}}^p = \liminf_{\ell \to \infty} \int_{\R^N} \abs{u_{n_\ell}}^p ,
\]
and by Claim~\ref{claimWeakConvergence}
\[
 \mathcal{I}_{A} (u) = \bigl(\tfrac{1}{2} - \tfrac{1}{p}\bigr)\int_{\R^N} \abs{u}^p.
\]
We conclude that 
\[
\liminf_{\ell \to \infty} \mathcal{I}_{A_{n_\ell}} (u_{n_\ell}) \ge \mathcal{I}_{A} (u).
\] 
\end{proofclaim}

\begin{claim} \label{claimLimsup}
If moreover
\[
 \limsup_{\ell \to \infty} \mathcal{I}_{A_{n_\ell}} (u_{n_\ell}) \leq \mathcal{I}_A (u) ,
\]
then $\tau^{A_{n_\ell}}_{a_{n_\ell}} u_{n_\ell} \rightarrow u$ and $D_{A_{n_\ell}} \tau^{A_{n_\ell}}_{a_{n_\ell}} u_{n_\ell} \rightarrow D_A u$ strongly in $L^{2}(\mathbb{R}^N)$.
\end{claim}

\begin{proofclaim}

By assumption, we have $\limsup_{\ell \to \infty} \mathcal{I}_{A_{n_\ell}} (u_{n_\ell}) \le \mathcal{I}_{A} (u)$ so that in view of Claim~\ref{claimLiminf} we conclude that $\lim_{l \to \infty} \mathcal{I}_{A_{n_\ell}}(u_{n_\ell}) = \mathcal{I}_A (u)$. Then, since 
\[
  \bigl(\tfrac{1}{2}-\tfrac{1}{p}\bigr) \int_{\R^N} \abs{D_{A_{n_\ell}} (\tau^{A_{n_\ell}}_{a_{n_\ell}} u_{n_\ell})}^2 + \abs{\tau^{A_{n_\ell}}_{a_{n_\ell}} u_{n_\ell}}^2 = \mathcal{I}_{A_{n_\ell}} (u_{n_\ell})- \frac{1}{p} \dualprod{\mathcal{I}'_{A_{n_\ell}} (u_{n_\ell})}{u_{n_\ell}}
\]
and 
\[
  \bigl(\tfrac{1}{2}-\tfrac{1}{p}\bigr) \int_{\R^N} \abs{D_{A} u}^2 + \abs{u}^2 = \mathcal{I}_{A} (u),
\]
we conclude that $\tau^{A_{n_\ell}}_{a_{n_\ell}} u_{n_\ell} \to u$ and $D_{A_{n_\ell}}(\tau^{A_{n_\ell}}_{a_{n_\ell}} u_{n_\ell}) \to D_A u$ strongly in $L^2 (\R^N)$. 
\end{proofclaim}
\end{proof}

\subsection{Continuity of groundstates}

We are now ready to state and prove the main result of this section.

\begin{proposition}[Continuity of \(\mathcal{E}\) and of the groundstates] \label{propositionContinuityGroundState}
The ground-energy function $\mathcal{E}:\bigwedge^2 \Rset^N \to \R$ is continuous. 
Moreover, if $(A_n)_{n \in \N}$ is a sequence in $L^2_\mathrm{loc} (\R^N, \bigwedge^1 \R^N)$ such that $A_n \to A$ in $L^2_{\mathrm{loc}}(\mathbb{R}^N)$ as $n \to \infty$ and if the sequence $(u_n)_{n \in \N}$ in $H^1_{A_n} (\R^N,\mathbb{C})$ satisfies $\mathcal{I}_{A_n}' (u_n) = 0$ and $\mathcal{I}_{A_n} (u_n) = \mathcal{E} (dA_n)$, then there exist $u \in H^1_A (\R^N, \mathbb{C})$ with $\mathcal{I}_A (u) = \mathcal{E} (dA)$ and $\mathcal{I}'_{A} (u) = 0$, a sequence $(a_n)_{n \in \N}$ in $\R^N$ and a subsequence such that $\tau^{A_{n_\ell}}_{a_{n_\ell}} u_{n_\ell} \to u$ and $D_{A_{n_\ell}} (\tau^{A_{n_\ell}}_{a_{n_\ell}} u_{n_\ell}) \to D_{A} u$ strongly in $L^2 (\R^N)$ as $\ell \to \infty$.
\end{proposition}

\begin{proof}
We first prove the upper semicontinuity of $\mathcal{E}$. For this, we observe that, by density of $C^1_c (\R^N, \C)$ in $H^1_{A} (\R^N, \mathbb{C})$ (see for example \cite{LiebLoss}*{Theorem 7.22}) and by the characterization of Lemma~\ref{lemmaExistenceGroundstate}, we have
\[
  \mathcal{E} (dA) = \bigl(\tfrac{1}{2} - \tfrac{1}{p}\bigr) \inf_{v \in C^1_c (\R^N,\C) \setminus \{0\}} \bigl(\mathcal{Q}_{A}(v)\bigr)^\frac{p}{p - 2}.
\]
Since for every $v \in C^1_c (\R^N, \C)$, the function $A \mapsto \mathcal{Q}_A (v)$ is continuous, $\mathcal{E}$ is upper semicontinuous as an infimium of upper semicontinuous functions. 

Let us now prove that $\mathcal{E}$ is lower semicontinuous. By Lemma~\ref{lemmaExistenceGroundstate}, there exists $u_n \in H^1_{A_n} (\R^N, \mathbb{C})$ such that $\mathcal{I}_{A_n} (u_n) = \mathcal{E} (dA_n)$ and $\mathcal{I}_{A_n}' (u_n) = 0$. By Lemma~\ref{lemmaPalaisSmale}, there exist $u \in H^1_{A} (\R^N, \mathbb{C}) \setminus \{0\}$ with $\mathcal{I}_A'(u)=0$, a sequence $(a_n)_{n \in \N}$ in $\mathbb{R}^N$ and a subsequence $(n_\ell)_{\ell \in \N}$ in $\N$ such that 
$\tau^{A_{n_\ell}}_{a_{n_\ell}} u_{n_\ell} \weakto u$, $D_{A_{n_\ell}} (\tau^{A_{n_\ell}}_{a_{n_\ell}} u_{n_\ell}) \weakto D_A u$ weakly in $L^2 (\R^N)$ as $\ell \to \infty$ and 
\[
\lim_{\ell \to + \infty} \mathcal{I}_{A_{n_\ell}} (u_{n_\ell}) = \liminf_{n \to \infty} \mathcal{I}_{A_n} (u_n) \ge \mathcal{I}_A (u),
\]
if we choose well the subsequence. Since $\mathcal{I}_A (u) \ge \mathcal{E} (dA)$, it follows that the function $\mathcal{E}$ is lower semicontinuous. Moreover, since $\mathcal{E}$ is upper semicontinuous, we conclude that $\tau^{A_{n_\ell}}_{a_{n_\ell}} u_{n_\ell} \to u$ and $D_{A_{n_\ell}} (\tau^{A_{n_\ell}}_{a_{n_\ell}} u_{n_\ell}) \to D_A u$ strongly in $L^2 (\R^N)$ as $\ell \to \infty$.
\end{proof}

\section{Uniqueness up to magnetic translations and rotations in $\mathbb{C}$ of the groundstates} \label{section:uniquenessandsymmetry}

In this section, we prove that when the magnetic field $dA$ is small enough, the groundstate of the magnetic nonlinear Schr\"odinger equation \eqref{eqNLSEMag} is essentially unique. Uniqueness holds up to the invariances of the problem, namely magnetic translations $\tau_a^A$ and rotations in $\C$. To deduce this result, we start from the case $A = 0$, for which we know that there is uniqueness of the groundstate up to translations and rotations in $\C$, as proved in Proposition~\ref{propositionDiamagneticUniqueness}, and we analyse the problem when $dA$ is small as a perturbation of it. 

A natural tool to perform this perturbation analysis is the implicit function theorem. However, an application of the implicit function theorem would face the difficulty of finding the right framework to work in variable magnetic Sobolev spaces. We rely instead on a spectral approach that consists in proving directly the uniqueness by a sort of local injectivity theorem reminiscent of Bonheure, Bouchez, Grumiau and Van Schaftingen, see \cite{BonheureBouchezGrumiauVanSchaftingen2008} (see also \cite{BonheureBouchezGrumiau2009}). Compared to their work, we face the additional difficulty that the linear operator acts on variable functional spaces.

\subsection{Spectral theory across magnetic Sobolev spaces}

Our main tool is to prove the stability of the spectrum of a magnetic operator under perturbations of the potentials. If the magnetic potential is constant, this is quite classical. The difficulty  when the magnetic potential varies is that its variations are not controlled globally in any norm.

\smallbreak

We first study the spectrum of the sequence of linear operators 
\begin{equation}\label{eq:Ln}
L_n : H^1_{A_n} (\R^N,\C)\to H^1_{A_n} (\R^N,\C) : v\mapsto L_n v := (-\Delta_{A_n} + 1)^{-1} W_n [v],
\end{equation}
where 
\begin{enumerate}[$(H_1)$]
\item $(A_n)_{n \in \N}$ is a sequence in $L^2_{\mathrm{loc}}(\R^N)$,
\item $(W_n)_{n \in \N}$ is a sequence in $L^q (\R^N, \Lin(\C,\C))$ with $q \ge \frac{N}{2}$ and $q > 1$ , 
\item $W_n$ is self-adjoint and $W_n \ge 0$ on $\R^N$, that is $\scalprod{z}{W_n[z]} \geq 0$ for every $z \in \mathbb{C}$.
\end{enumerate}
The next lemma is given with a proof for completeness.

\begin{lemma} \label{lemma:compact-L}
If the assumptions $(H_1)$--$(H_3)$ hold, then the operator $L_n$ is self-adjoint and compact. 
\end{lemma}

\begin{proof}
The fact that $L_n$ is self-adjoint is clear. To prove the compactness, assume that $(v_k)_{k \in \N}$ is a bounded sequence in $H^1_{A_n} (\R^N,\C)$. Therefore there exists a weak limit $v \in H^1_{A_n}(\R^N,\C)$. We write $w_k = L_n(v_k)$. For each \(k \in \N\), the function  $w_k \in H^1_{A_n} (\R^N,\C)\) satisfies the equation
\[
  -\Delta_{A_n} w_k+ w_k = W_n[v_k].
\]
Therefore
\[ 
\int_{\R^N} \abs{D_{A_n} w_k}^2 + \abs{w_k}^2 = \int_{\R^N} \scalprod{w_k}{W_n [v_k]}\le \int_{\R^N} |W_n|_{\Lin(\C,\C)} |v_k| |w_k|.
\]
Using the boundedness of $W_n$ in $L^q(\R^N)$ and Sobolev embeddings, we deduce that $(w_k)_k$ is bounded in $H^1_{A_n}(\R^N,\C)$. This implies the existence of a weak limit $w\in H^1_{A_n} (\R^N,\C)$.   
For every $\varepsilon>0$, there exists a compact set $K_\varepsilon$ such that 
\[ 
\||W_n|_{\Lin(\C,\C)}\|_{L^q(K_\varepsilon^c)}\le \frac{\varepsilon}{4 \sup_k(\|v_k\|,\|w_k\|)}
\]
so that 
\[
\left|\int_{K_\varepsilon^c}  \scalprod{w_k}{W_n [v_k]} - \scalprod{w}{W_n [v]}\right|\le \varepsilon.
\]
This estimate combined with local compactness yields
\begin{equation} \label{eq:1}
\int_{\R^N}  \scalprod{w_k}{W_n [v_k]} \to \int_{\R^N}  \scalprod{w}{W_n [v]}.
\end{equation}
By testing the equation of $w_k$ on $w$ and using in order the weak convergence $w_k \weakto w$ and \eqref{eq:1}, we conclude that
\begin{equation*}
\begin{split}
\int_{\R^N} \abs{D_{A_n} w}^2 + \abs{w}^2  & = \lim_{k\to\infty}\int_{\R^N} \scalprod{w}{W_n [v_k]}\\
& = \lim_{k\to\infty}\int_{\R^N} \scalprod{w_k}{W_n [v_k]} = \lim_{k\to\infty}\int_{\R^N} \abs{D_{A_n} w_k}^2 + \abs{w_k}^2,
\end{split}
\end{equation*}
so that the convergence is strong in $H^1_{A_n} (\R^N,\C)$. 
\end{proof}

Lemma~\ref{lemma:compact-L} and the positivity of $W_n$ imply that  $L_n$ has a nonincreasing sequence of positive eigenvalues converging to \(0\) and by Fischer's min-max principle, see for example \cite{Lax2002}*{Theorem 28.4}, one has 
\begin{equation}
\label{eqFischerMinMaxSobolev}
\lambda_k (L_n) = \sup_{\substack{E \subset H^1_{A_n} (\R^N, \C)\\ \dim E = k}}  \inf_{v \in E}
  \frac{\displaystyle \int_{\R^N} \scalprod{v}{W_n [v]} }{\displaystyle \int_{\R^N} \abs{D_{A_n} v}^2 + \abs{v}^2}.
\end{equation}
We now turn to the convergence of those eigenvalues when $A_n$ and $W_n$ have strong limits. 

\begin{proposition}%
[Convergence of eigenvalues and eigenfunctions]%
\label{propositionSpectrum}
Assume that $(H_1)-(H_3)$ hold. If $A_n \to A$ strongly in $L^2_\mathrm{loc} (\R^N)$ and $W_n \to W$ strongly in $L^q (\R^N)$ as $n \to \infty$,
then 
\[
  \lambda_k (L_n) \to \lambda_k (L),
\]
where $\lambda_k(L_n)$, $\lambda_k(L)$ are respectively the $k$-th eigenvalues of $L_n$, $L$, and $L : H^1_A(\mathbb{R}^N, \C) \to H^1_A(\R^N,\C)$ is defined as
\[
L v = (- \Delta_A + 1)^{-1} W [v].
\]
Moreover, if $u_n \in H^1_{A_n} (\R^N, \mathbb{C})$ is an eigenfunction of $L_n$ satisfying
\[
  L_n u_n = \lambda_k (L_n) u_n
\]
and 
\[
  \int_{\R^N} \abs{D_{A_n} u_n}^2  + \abs{u_n}^2 = 1,
\]
then there exist $u \in H^1_A (\R^N, \C)$ and a subsequence $(n_\ell)_{\ell \in \mathbb{N}}$ such that $u_{n_\ell} \rightarrow u$ and $D_{A_{n_\ell}} u_{n_\ell} \rightarrow D_{A}  u$ strongly in $L^2 (\R^N)$.
\end{proposition}

\begin{proof}
\resetclaim
We assume that $W \not\equiv 0$. Then, in particular, $\lambda_k (L) > 0$ for every $k \in \N_*	 \defeq \{ 1, 2, \dotsc \}$. Since $C^1_c (\R^N, \C)$ is dense in $H^1_{A_n} (\R^N, \mathbb{C})$, Fischer's min-max principle \eqref{eqFischerMinMaxSobolev} also yields
\begin{equation}
\label{eqFischerMinMaxSmooth}
 \lambda_k (L_n) = \sup_{\substack{E \subset C^1_c (\R^N, \C)\\ \dim E = k}} \inf_{v \in E}
  \frac{\displaystyle \int_{\R^N} \scalprod{v}{W_n [v]} }{\displaystyle \int_{\R^N} \abs{D_{A_n} v}^2 + \abs{v}^2}.
\end{equation}
The advantage of \eqref{eqFischerMinMaxSmooth} over \eqref{eqFischerMinMaxSobolev} is that the distinct magnetic spaces do not appear anymore in the set over which the supremum is taken.

\medbreak

\begin{claim}
\label{claimEigenLimInf}
For every $k \in \mathbb{N}_0$,
\begin{equation*}
\liminf_{n \to + \infty} \lambda_k(L_n) \geq \lambda_k (L).
\end{equation*} 
\end{claim}

\begin{proofclaim}
If the linear subspace $E \subset C^1_c (\R^N, \C)$ has finite dimension, we have
\begin{equation*}
 \frac{\displaystyle \int_{\R^N} \scalprod{v}{W_n [v]}}{\displaystyle \int_{\R^N} \abs{D_{A_n} v}^2 + \abs{v}^2}
  \to
  \frac{\displaystyle \int_{\R^N} \scalprod{v}{W [v]} }{\displaystyle \int_{\R^N} \abs{D_{A}  v}^2 + \abs{v}^2},
\end{equation*}
uniformly in $v \in E$. The convergence holds because of the strong convergences $W_n \to W$ in $L^q(\mathbb{R}^N)$ and $A_n \to A$ in $L^2_{\mathrm{loc}}(\mathbb{R}^N)$, while the uniformity follows from the finiteness of the dimension of $E$. The claimed inequality is now obvious since $\lambda_k (L_n)$ is characterized in \eqref{eqFischerMinMaxSmooth} by a supremum on $k$-dimensional spaces in $C^1_c (\mathbb{R}^N, \mathbb{C})$.
\end{proofclaim}

It remains to prove the converse inequality and the convergence of the eigenfunctions.
Let the sequence $(v_n^k)_{k \in \N}$ be an orthonormal basis in $H^1_{A_n} (\R^N, \mathbb{C})$ of eigenvectors of $L_n$, that is,
\begin{equation*}
  \int_{\R^N} \abs{D_{A_n} v_n^k}^2 + \abs{v_n^k}^2 = 1, \qquad \quad \int_{\R^N} \scalprod{v_n^k}{W_n[v_n^k]} = \lambda_k (L_n),
\end{equation*}
and, if $j \ne k$,
\[
  \int_{\R^N} \scalprod{v^k_n}{W_n [v^j_n]}
  = \int_{\R^N} \scalprod{D_{A_n} v_n^k}{D_{A_n} v_n^j} + \scalprod{v_n^k}{v_n^j}
  = 0.
\] 
From the weak compactness property across magnetic spaces of Lemma~\ref{lemma:WeakCompactness}, and going to a subsequence $(n_\ell)_{\ell \in \mathbb{N}}$, we can assume that 
\[
\lambda_k(L_{n_\ell}) \rightarrow \lambda_k^\star := \limsup_{n \to + \infty} \lambda_k (L_n)
\] 
and there exists a function $v^k \in H^1_A(\mathbb{R}^N, \mathbb{C})$ such that 
$v^k_{n_\ell} \weakto v^k$, $D_{A_{n_\ell}} v^k_{n_\ell} \weakto D v^k$ weakly in $L^2 (\R^N)$ as $\ell \to + \infty$ . 

\begin{claim}\label{Claim2}

For every $k \in \mathbb{N}_0$, we have  $\lambda_k^\star = \lambda_k (L)$, $L v^k = \lambda_k (L) v^k$ and the convergences $v_{n_\ell}^k \to v^k$, $D_{A_{n_\ell}} v_{n_\ell}^k \rightarrow D_{A}  v^k$ are strong in $L^2(\R^N)$.
\end{claim}

\begin{proofclaim}
We first proceed similarly as in the proof of Lemma~\ref{lemma:compact-L}. Applying Theorem~\ref{lemmaRellich} and Sobolev inequalities together with the diamagnetic inequality \eqref{eqDiamagnetic}, we infer that $v^k_{n_\ell} \otimes v^j_{n_\ell} \weakto v^k \otimes v^j$ weakly in $L^{q / (q - 1)} (\R^N)$ as $\ell \to \infty$.
Then, since $W_n \rightarrow W$ in $L^q (\mathbb{R}^N)$, we have 
\begin{equation} \label{eq:inequality1}
\begin{split}
\lambda_k (L) \le \liminf_{n \to + \infty} \lambda_k (L_n)
\leq \lambda_k^\star &= \lim_{\ell \to \infty} \int_{\R^N} \scalprod{v_{n_\ell}^k}{W_{n_\ell} [v_{n_\ell}^k]}\\ &= \int_{\R^N} \scalprod{v^k}{W [v^k]} 
 \leq \frac{\displaystyle \int_{\R^N} \scalprod{v^k}{W [v^k]} }{\displaystyle \int_{\R^N} \abs{D_{A}  v^k}^2 + \abs{v^k}^2}
\end{split}
\end{equation}
where we have used Claim~\ref{claimEigenLimInf}, the weak lower semi-continuity of the norm and the normalization of the eigenfunctions, and if $j \neq k$,
\[
  \int_{\R^N} \scalprod{v^k}{W [v^j]} =\lim_{\ell \to \infty} \int_{\R^N} \scalprod{v_{n_\ell}^k}{W_{n_\ell} [v_{n_\ell}^j]}
  = 0.
\]
We now proceed by induction. For $k=1$, since $v^1 \in H^1_A(\mathbb{R}^N, \mathbb{C})$ is a competitor, we have
\[ 
 \frac{\displaystyle \int_{\R^N} \scalprod{v^1}{W [v^1]} }{\displaystyle \int_{\R^N} \abs{D_{A} v^1}^2 + \abs{v^1}^2} \leq \sup_{v \in H^1_A(\mathbb{R}^N, \mathbb{C})} \frac{\displaystyle \int_{\R^N} \scalprod{v}{W [v]} }{\displaystyle \int_{\R^N} \abs{D_{A} v}^2 + \abs{v}^2} = \lambda_1(L).
\]
Therefore, \eqref{eq:inequality1} implies $L v^1 = \lambda_1 (L) v^1$ and $\lambda_1^\star=\lambda_1 (L)$. Moreover, we deduce that 
\begin{equation*}
\lambda_1(L) \geq \lambda_1(L) \int_{\mathbb{R}^N} \abs{D_{A} v^1}^2 + \abs{v^1}^2 = \int_{\mathbb{R}^N} \scalprod{v^1}{W [v^1]} \geq \liminf_{n \to + \infty} \lambda_1(L_n) \geq \lambda_1 (L),
\end{equation*}
which implies that $D_{A_{n_\ell}} v^1_{n_\ell} \rightarrow D_{A}  v^1$ and $v_{n_\ell}^1 \to v^1$ strongly in $L^2 (\R^N)$.

Now, assume the claim holds true for $j \in \{1, \dotsc, k-1\}$. Then, one has the orthogonality relations
\[
  \int_{\R^N} \scalprod{D_{A} v^{k}}{D_{A} v^j} + \scalprod{v^{k}}{v^{j}}
  = \int_{\R^N} \scalprod{v^{k}}{W [v^j]} = 0,\quad j=1,\dotsc,k-1.
\]
We therefore deduce fom the variational characterization of eigenvalues of a compact symmetric operator that 
\[ 
 \frac{\displaystyle \int_{\R^N} \scalprod{v^k}{W [v^k]} }{\displaystyle \int_{\R^N} \abs{D_{A} v^k}^2 + \abs{v^k}^2} \leq \lambda_k(L).
\]
Then \eqref{eq:inequality1} implies $L v^k = \lambda_k (L) v^k$ and $\lambda_k^\star=\lambda_k (L)$. Finally, as for $\lambda_1 (L)$, we observe that 
\begin{equation*}
\lambda_k(L) \geq \lambda_k(L) \int_{\mathbb{R}^N} \abs{D_A v^k}^2 + \abs{v^k}^2 = \int_{\mathbb{R}^N} \scalprod{v^k}{W[v^k]} \geq \liminf_{n \to + \infty} \lambda_k (L_n) \geq \lambda_k (L).
\end{equation*}
Again this implies $v_{n_\ell}^k \rightarrow v^k$ and $D_{A_{n_\ell}} v_{n_\ell}^k \to D_{A} v^k$ strongly in $L^2 (\R^N)$. 
\end{proofclaim}

\end{proof}

\subsection{Proof of the local uniqueness}

To prove the local uniqueness of the solutions, up to magnetic translations $\tau_a^A$ and rotations in $\C$, we  assume by contradiction the existence of two distinct groundstates $u_n$ and $v_n$ of \eqref{eqNLSEMag} with $A_n$. 

In the Claim~\ref{claim:convergence-forte} of the proof, we aim to show that we can assume that both entire sequences $u_n$ and $v_n$ converge strongly to the same groundstate $U$ of the limit problem \eqref{eq:limitproblem}. This relies on the uniqueness up to translations and rotations in $\mathbb{C}$ of the limit problem, see Proposition~\ref{propositionDiamagneticUniqueness}, which allows to adjust the sequences using convenient magnetic translations and multiplications by complex phases. 

In Claim~\ref{claim:ortogonality}, we prove that, by modifying slightly the phase and the magnetic translation $\tau_a^A$ used in Claim~\ref{claim:convergence-forte} (keeping the strong convergence to the limit function), $u_n$ and $v_n$ are asymptotically orthogonal in $H^1$ (in a sense that is clarified below) to the tangent space of $U$ given by the $w$ in \eqref{eq:w}, for large $n$.

In Claim~\ref{claim:tangent}, we prove that $u_n - v_n$ is the eigenfunction of a compact operator. At the limit, the equation satisfied by $u_n - v_n$ approaches \eqref{eq:NonDegeneracy}, which we know to be satisfied by the elements of the tangent space of $U$ only.

Heuristically, the combination of the two last claims shows that for large $n$, $u_n - v_n$ is orthogonal to the functions in the tangent space of $U$, and in the same time $u_n - v_n$ is almost in the tangent space of $U$, so that the only possibility is $u_n = v_n$ for $n$ large. The end of the proof relies on the spectral decomposition in eigenvalues greater, equal or smaller than $1$ of the limit operator (see \S \ref{subsec:groundstate0}).

\begin{proof}[Proof of Theorem~\ref{theoremUniqueness}]
\resetclaim
We first assume that $\mathrm{d} A_n \to 0$ as $n \to + \infty$, that is $A_n \to 0$ in $L^2_{\text{loc}}(\R^N)$ as $n \to + \infty$ since $A_n$ is skew-symmetric, and that $u_n$ and $v_n$ are groundstates solutions of \eqref{eqNLSEMag} with $A_n$. Our aim is to show that there exist $\theta_n \in \mathbb{R}$ and $a_n \in \mathbb{R}^N$ such that $u_n = \e^{i \theta_n} \tau^{A_n}_{a_n} v_n$ for $n$ large enough.

Let $U$ be a solution of the limit problem \eqref{eq:limitproblem}. By proposition~\ref{propositionDiamagneticUniqueness}, $U$ is unique up to rotations in $\C$ and translations in $\R^N$.

\begin{claim}  \label{claim:convergence-forte}
There exist sequences $(\Tilde{\theta}_n)_{n \in \N}\) in \(\R$ and $(\Tilde{a}_n)_{n \in \N}\) in \(\R^N$ such that
\[
  \lim_{n \to \infty} \int_{\R^N} \abs{D_{A_n} (\e^{i \Tilde{\theta}_n} \tau^{A_n}_{\Tilde{a}_n} u_n) -  D U}^2 + \abs{\e^{i \Tilde{\theta}_n} \tau^{A_n}_{\Tilde{a}_n} u_n - U}^2 = 0.
\]
\end{claim}

\begin{proofclaim}
By Proposition~\ref{propositionContinuityGroundState}, there exist a sequence $(b_n)_{n \in \N}$ in $\mathbb{R}^N$, a subsequence $(n_\ell)_{\ell \in \mathbb{N}}$ in $\mathbb{N}$ and a function $V \in H^1(\mathbb{R}^N, \mathbb{C})$ such that $\tau^{A_{n_\ell}}_{b_{n_\ell}} u_{n_\ell} \rightarrow V$ and $D_{A_{n_\ell}} (\tau^{A_{n_\ell}}_{b_{n_\ell}} u_{n_\ell}) \rightarrow DV$ strongly in $L^2(\mathbb{R}^N)$. Because of the uniqueness up to translations and rotations in $\mathbb{C}$ of the solution of \eqref{eq:limitproblem}, there exist $b \in \mathbb{R}^N$ and $\omega \in \mathbb{R}$ such that $V = \e^{i \omega} \tau_{b}^0 U$. We can therefore write that
\begin{multline*}
\int_{\mathbb{R}^N} | D_{A_{n_\ell}} (\tau^{A_{n_\ell}}_{b_{n_\ell}} u_{n_\ell}) - DV |^2 = 
\int_{\mathbb{R}^N} | \e^{- i \omega} \tau^{A_{n_\ell}}_{-b} D_{A_{n_\ell}} (\tau^{A_{n_\ell}}_{b_{n_\ell}} u_{n_\ell}) - \tau^{A_{n_\ell}}_{-b} \tau^{0}_{b} DU |^2 = \\
\int_{\mathbb{R}^N} | \e^{-i \omega} \e^{- i A_{n_\ell}(b_{n_\ell}) [b] } D_{A_{n_\ell}} (\tau^{A_{n_\ell}}_{b_{n_\ell}-b} u_{n_\ell}) - DU + DU - \tau^{A_{n_\ell}}_{-b} \tau^0_b U |^2 \rightarrow 0,
\end{multline*}
as \(\ell \to + \infty\).
Here, we used the commutation between the translation and the connexion and \eqref{eq:composition-translations}. Moreover, by using Lebesgue dominated convergence, we have that
\begin{equation*}
\int_{\mathbb{R}^N} | DU - \tau^{A_{n_\ell}}_{-b} \tau^0_{b} DU |^2 \rightarrow 0, \quad \text{ as } \ell \to + \infty.
\end{equation*}
By the triangle inequality, we infer that
\begin{equation*}
\int_{\mathbb{R}^N} |  \e^{-i \omega} \e^{- i A_{n_\ell}(b_{n_\ell}) [b] } D_{A_{n_\ell}} (\tau^{A_{n_\ell}}_{b_{n_\ell}-b} u_{n_\ell}) - DU |^2 \rightarrow 0, \quad \text{ as } \ell \to + \infty,
\end{equation*}
and proceeding exactly in the same way, we obtain
\begin{equation*}
\int_{\mathbb{R}^N} | \e^{-i \omega} \e^{- i A_{n_\ell}(b_{n_\ell}) [b] } \tau^{A_{n_\ell}}_{b_{n_\ell}-b} u_{n_\ell} - U |^2 \rightarrow 0, \quad \text{ as } \ell \to + \infty.
\end{equation*}
Setting $\Tilde{\theta}_{n_\ell} = -\omega - A_{n_\ell}(b_{n_\ell}) [ b] $ and $\Tilde{a}_{n_\ell} = b_{n_\ell}-b$, the conclusion of the claim follows for this subsequence.

The claim is then true for the whole sequence $n$. Indeed, if it is not the case, we would find a subsequence $n_\ell$ for which the Claim does not hold, leading to a contradiction.
\end{proofclaim}

\begin{claim} \label{claim:ortogonality}
There exist sequences $(\theta_n)_{n \in \N}$ in $\R$ and $(a_n)_{n \in \N}$ in $\R^N$ such that
\[
  \lim_{n \to \infty} \int_{\R^N} \abs{D_{A_n} (\e^{i \theta_n} \tau^{A_n}_{a_n} u_n) -  D U}^2 + \abs{\e^{i \theta_n} \tau^{A_n}_{a_n} u_n - U}^2 = 0.
\]
Moreover, when $n \in \N$ is large enough, for every $w\in \R^N$, we have the following orthogonality relations
\begin{gather*}
  \int_{\R^N} \bigscalprod{D_{A_n} (\e^{i \theta_n} \tau^{A_n}_{a_n} u_n)}{D(D U[w])} + \scalprod{\e^{i \theta_n} \tau^{A_n}_{a_n} u_n}{DU [w]} = 0,\\
  \int_{\R^N} \bigscalprod{D_{A_n} (\e^{i \theta_n} \tau^{A_n}_{a_n} u_n)}{D i U} + \scalprod{\e^{i\theta_n} \tau^{A_n}_{a_n} u_n}{i U} = 0.
\end{gather*}
\end{claim}

\begin{proofclaim}
We already proved Claim~\ref{claim:convergence-forte} with $\Tilde{\theta}_n$ and $\Tilde{a}_n$. Let us first prove the two orthogonality relations. For this, we define the map $\Phi_n \in C (\R^{N + 1}, \R^{N + 1})$ for each $(x, \tau) \in \R^{N + 1}$ and  $(w, s) \in \R^{N + 1}$ by the following scalar product
\[
\begin{split}
  \bigscalprod{(w, s)}{&\Phi_n (x, \tau)}\\
  =& \int_{\R^N} \bigscalprod{D_{A_n} (\e^{i (\Tilde{\theta}_n + \tau)} \tau^{A_n}_{x} \tau^{A_n}_{\Tilde{a}_n} u_n)}{D (DU[w])} + \bigscalprod{\e^{i (\Tilde{\theta}_n + \tau)} \tau^{A_n}_{x} \tau^{A_n}_{\Tilde{a}_n} u_n}{DU[w]}\\
  &+ \int_{\R^N} \bigscalprod{D_{A_n} (\e^{i (\Tilde{\theta}_n + \tau)} \tau^{A_n}_{x} \tau^{A_n}_{\Tilde{a}_n} u_n)}{s (D i U)} + \bigscalprod{\e^{i (\Tilde{\theta}_n + \tau)} \tau^{A_n}_{x} \tau^{A_n}_{\Tilde{a}_n} u_n}{s i U}.
\end{split}
\]
Since $D_{A_n} \circ \tau^{A_n}_{x} \tau^{A_n}_{\Tilde{a}_n}= \tau^{A_n}_{x} \tau^{A_n}_{\Tilde{a}_n} \circ D_{A_n}$, and thanks to the convergence proved in Claim~\ref{claim:convergence-forte}, the sequence $(\Phi_n)_{n \in \N}$ converges to $\Phi$ uniformly over compact subsets, where the function $\Phi \in C (\R^{N + 1}, \R^{N + 1})$ is defined for every $(x, \tau) \in \R^{N + 1}$ and  $(w, s) \in \R^{N + 1}$ by
\[
\begin{split}
  \bigscalprod{(w, s)}{\Phi (x, \tau)}
  = \int_{\R^N} \bigscalprod{D (\e^{i \tau} \tau^{0}_{x} U)}{D (DU[w])} + \bigscalprod{\e^{i \tau} \tau^{0}_{x} U}{DU[w]}\\
  + \int_{\R^N} \bigscalprod{D (\e^{i \tau} \tau^{0}_{x} U)}{s (D i U)} + \bigscalprod{\e^{i \tau} \tau^{0}_{x} U}{s i U}.
\end{split}  
\]
We first remark that $\Phi(0,0) = 0$. This is due to the fact that $D U[w] + s i U$ belongs to the tangent space of $U$, see \eqref{eq:Orthogonality}. Next, observe now that 
\begin{multline*}
  \bigscalprod{(w, s)}{D \Phi (0, 0) [z, r]}
  = \int_{\R^N} \bigscalprod{D (D u[z])}{D (DU[w])} + \bigscalprod{DU[z]}{DU[w]}\\
  + \int_{\R^N} \bigscalprod{r (D i U)}{s (D i U)} + \bigscalprod{r i U}{s i U},
\end{multline*}
meaning that $D \Phi(0,0) \geq 0$. Therefore, for every small $\rho > 0$, the Brouwer topological degree $\deg (\Phi, B_\rho, 0)$ of $\Phi$ on $B_\rho$ with respect to $0$ is well-defined, and $\deg (\Phi, B_\rho, 0) = 1$. Hence, since we have the uniform convergence on compacts of the continuous functions $\Phi_n$, for $n$ large enough, we obtain that $\deg (\Phi_n, B_\rho, 0) = 1$. We conclude to the existence of a sequence $(x_n, \tau_n)$ such that $\Phi_n (x_n, \tau_n) = 0$ for every $n$ large enough, and $(x_n, \tau_n) \to (0, 0)$ as $n \to \infty$. Finally, setting $a_n = x_n + \Tilde{a}_n$ and $\theta_n = \Tilde{\theta}_n + \tau_n + i A (\Tilde{a}_n)[x_n]$, we reach the conclusion in view of the composition formula for magnetic translations \eqref{eq:composition-translations}, and using again the Lebesgue dominated convergence.
\end{proofclaim}

\medbreak
Applying the first two claims to the sequence $(v_n)_n$ and renaming $\Tilde{u}_n = \e^{i \theta_n} \tau^{A_n}_{a_n} u_n$ and $\Tilde{v}_n = \e^{i \varphi_n} \tau^{A_n}_{c_n} v_n$ (where the couple $(\varphi_n, c_n) \in \mathbb{R}^{N+1}$ is given by the claims), we can assume that $\Tilde{u}_n$ satisfied
\[
  \lim_{n \to \infty} \int_{\R^N} \abs{D_{A_n} \Tilde{u}_n -  D U}^2 + \abs{\Tilde{u}_n- U}^2 = 0,
\]
and for every $w\in \R^N$,
\begin{gather*}
  \int_{\R^N} \scalprod{D_{A_n} \Tilde{u}_n}{D( D U[w])} + \scalprod{\Tilde{u}_n}{DU[w]} = 0, \\
  \int_{\R^N} \scalprod{D_{A_n} \Tilde{u}_n}{D i U} + \scalprod{ \Tilde{u}_n}{iU} = 0,
\end{gather*}
and the same for $\Tilde{v}_n$.

\begin{claim} \label{claim:tangent}
There exists $W_n \in L^q (\R^N, \Lin(\C, \C))$ such that 
\[
  -\Delta_{A_n} (\Tilde{u}_n - \Tilde{v}_n) + (\Tilde{u}_n - \Tilde{v}_n) = W_n [\Tilde{u}_n - \Tilde{v}_n] \qquad \text{in \(\R^N\)},
\]
and 
\[
  W_n \to \abs{U}^{p - 2} + (p - 2) \abs{U}^{p - 4} U \otimes U  
\]
in $L^{q} (\R^N)$ for every $2 \leq q (p - 2) \leq \frac{2N}{N-2}$.
\end{claim}

We recall that the tensor product $\otimes$ has been defined in \eqref{eq:otimes}.

\begin{proofclaim}
We define $W_n : \R^N \to \Lin(\C, \C)$ by
\[
  \scalprod{w}{W_n[z]} = \int_0^1 D f \bigl((1 - t)\Tilde{u}_n + t \Tilde{v}_n\bigr)[w, z]\dif t,
\]
for $f(u) = \abs{u}^{p - 2} u$. Claim~\ref{claim:convergence-forte} and Lemma~\ref{lemmaSobolev} imply that $\Tilde{u}_n \rightarrow U$ and $\Tilde{v}_n \rightarrow U$ in $L^{q (p - 2)}(\mathbb{R}^N)$, for $2 \leq q (p - 2) \leq \frac{2N}{N-2}$. Then, it is clear that $W_n \in L^q (\R^N)$ and $W_n \rightarrow D f (U) = \abs{U}^{p - 2} + (p - 2) \abs{U}^{p - 4} U \otimes U$ in $L^q (\R^N)$ as $n \to \infty$.
\end{proofclaim}

\noindent \textbf{Conclusion} The compact operator defined by $L_n = (-\Delta_{A_n} + 1)^{-1} W_n$ enters in the hypothesis of Proposition~\ref{propositionSpectrum}. We know that the spectrum converges, i.e., $\lambda_k (L_n) \rightarrow \lambda_k (L)$. Moreover, since the limit equation is \eqref{eq:NonDegeneracy}, we also know that $\lambda_1(L) = p-1 > 1$, $\lambda_i (L) = 1$, for $i=2, \dotsc, N+2$, and $\lambda_i (L) < 1$, for $i \geq N+3$. 

We define the orthogonal projection operator $P^+_n$ on the first eigenvector, $P^0_n$ the projection on the eigenspace $E^0_n$ made by the $N + 1$ following eigenvectors, and $P^-_n = I - P^+_n - P^0_n$. We observe that $L_n$ commutes with $P^-_n$ and $P^+_n$ and $L_n (\Tilde{u}_n - \Tilde{v}_n) = \Tilde{u}_n - \Tilde{v}_n$. Moreover,
\[
\begin{split}
  \norm{P^n_+ (\Tilde{u}_n - \Tilde{v}_n)}^2_{H^1_{A_n} (\R^N, \C)} 
  &=\scalprod{P^+_n (\Tilde{u}_n - \Tilde{v}_n)}{P^+_n L_n(\Tilde{u}_n - \Tilde{v}_n)}_{H^1_{A_n} (\R^N, \C)}\\
  &=\scalprod{P^+_n (\Tilde{u}_n - \Tilde{v}_n)}{L_n P^+_n (\Tilde{u}_n - \Tilde{v}_n)}_{H^1_{A_n} (\R^N, \C)}\\
  &= \lambda_1 (L_n) \bignorm{P^+_n (\Tilde{u}_n - \Tilde{v}_n)}^2_{H^1_{A_n} (\R^N, \C)}.
\end{split}
\]
Then, since $\lim_{n \to + \infty} \lambda_1 (L_n) > 1$, $P^+_n (\Tilde{u}_n - \Tilde{v}_n) = 0$ for $n$ large enough. Similarly,
\[
\begin{split}
  \norm{P^-_n (\Tilde{u}_n - \Tilde{v}_n)}^2_{H^1_{A_n} (\R^N, \C)} 
  &=\scalprod{P^-_n (\Tilde{u}_n - \Tilde{v}_n)}{P^-_n L_n(\Tilde{u}_n - \Tilde{v}_n)}_{H^1_{A_n} (\R^N, \C)}\\
  &=\scalprod{P^-_n (\Tilde{u}_n - \Tilde{v}_n)}{L_n P^-_n (\Tilde{u}_n - \Tilde{v}_n)}_{H^1_{A_n} (\R^N, \C)}\\
  &\le \lambda_{N + 3}(L_n) \norm{P^-_n (\Tilde{u}_n - \Tilde{v}_n)}^2_{H^1_{A_n} (\R^N, \C)}.
\end{split}
\]
Thus, since $\lim_{n \to \infty} \lambda_{N + 3} (L_n) < 1$, $P^-_n (\Tilde{u}_n - \Tilde{v}_n) = 0$ for $n$ large enough.
Assume now by contradiction that, for every $n \in \N$, $\Tilde{u}_n \ne \Tilde{v}_n$. Then, the function
\[
  z_n = \frac{\Tilde{u}_n - \Tilde{v}_n}{\norm{\Tilde{u}_n-\Tilde{v}_n}_{H^1_{A_n} (\R^N, \C)}}
\]
is in the eigenspace $E_n^0$ and is a linear combination of eigenvectors. By Proposition~\ref{propositionSpectrum}, there exists $z = DU[w] + \lambda i U \in E^0$, where $E^0$ is the eigenspace of $L$ corresponding the eigenvalue $1$ (see \S \ref{subsec:groundstate0}), such that, up to a subsequence still denoted by $n$, 
\[
  \lim_{n \to \infty} \int_{\R^N} \abs{D_{A_n} z_{n} - D z}^2 + \abs{z_{n} - z}^2 = 0.
\]
By Claim~\ref{claim:ortogonality}, we also know that
\[
  \int_{\R^N} \scalprod{D_{A_n} z_n}{D z} + \scalprod{z_n}{z},
\]
for $n$ large enough. Finally
\begin{equation*}
\begin{split}
 \| \Tilde{u}_n - \Tilde{v}_n \|^2_{H^1_{A_n}(\mathbb{R}^N, \mathbb{C})} & = \int_{\mathbb{R}^N} \scalprod{ D_{A_n}(\Tilde{u}_n - \Tilde{v}_n)}{D_{A_n} (\Tilde{u}_n-\Tilde{v}_n)} + \scalprod{\Tilde{u}_n-\Tilde{v}_n}{\Tilde{u}_n-\Tilde{v}_n} \\
  & = \| \Tilde{u}_n - \Tilde{v}_n \|^2_{H^1_{A_n}(\mathbb{R}^N, \mathbb{C})} \int_{\mathbb{R}^N} \scalprod{ D_{A_n} z_n}{ D_{A_n} z_n} + \scalprod{ z_n }{z_n} \\
  & = \| \Tilde{u}_n - \Tilde{v}_n \|^2_{H^1_{A_n}(\mathbb{R}^N, \mathbb{C})} \int_{\mathbb{R}^N} \scalprod{ D_{A_n} z_n - D z}{ D_{A_n} z_n} + \scalprod{ z_n -z }{z_n} \\
  & \leq \| \Tilde{u}_n - \Tilde{v}_n \|^2_{H^1_{A_n}(\mathbb{R}^N, \mathbb{C})} \int_{\mathbb{R}^N} \abs{D_{A_{n}} z_{n} - D z}^2 + \abs{z_{n} - z}^2.
\end{split}
\end{equation*}
This is impossible. Then, $\Tilde{u}_n = \Tilde{v}_n$ for $n$ large.
\end{proof}

\subsection{Non-degeneracy for small magnetic fields}
As an application of the method to prove Theorem~\ref{theoremUniqueness} on the essential uniqueness of groundstates, we also obtain nondegeneracy of groundstates.

\begin{proposition}[Non-degeneracy for small magnetic fields]
\label{propositionNonDegeneracyA}
For every \(N \ge 2\) and \(p \in (2, \frac{2N}{N - 2})\), there exists \(\varepsilon > 0\) (given in Theorem~\ref{theoremUniqueness}) such that if \(\abs{dA} \le \varepsilon\), $u$ is a solution of \eqref{eqNLSEMag} with \(\mathcal{I}_A (u) \le \mathcal{E} (0)+\varepsilon\) and if $w \in H^1_A (\R^N,\C)$ satisfies
\begin{equation*}
 -\Delta_{A} w + w = \abs{u}^{p - 2} w + (p - 2) \abs{u}^{p - 4} \scalprod{u}{w} u, 
\end{equation*}
then there exist $y \in \R^N$ and $\lambda \in \R$ such that 
\begin{equation*}
  w = D_A u[y] + \lambda i u. 
\end{equation*}
\end{proposition}
\begin{proof}
This follows from Proposition~\ref{propositionNonDegeneracyWithoutMagneticField}, Lemma~\ref{lemmaPalaisSmale}, Proposition~\ref{propositionContinuityGroundState} and Proposition~\ref{propositionSpectrum}, with arguments similar to those in the proof of Theorem~\ref{theoremUniqueness}.
\end{proof}

\section{Symmetry of solutions} \label{section:symmetryof}

\subsection{Invariance under the rotations that preserve the magnetic field} \label{sec:symmetry}

In this section, we use the uniqueness up to magnetic translations and rotations in $\mathbb{C}$ proved in Theorem~\ref{theoremUniqueness} to deduce symmetry properties of the groundstate of \eqref{eqNLSEMag}.

\begin{proposition}%
[Symmetry of groundstates with vanishing center-of-mass]
\label{prop:symmetric-solution}%
Assume that $A \in \Lin( \mathbb{R}^N, \bigwedge^1 \mathbb{R}^N)$ is skew-symmetric and that $\abs{\mathrm{d}A} < \varepsilon$, where $\varepsilon > 0$ is given in Theorem~\ref{theoremUniqueness}. If $u$ is a groundstate solution of \eqref{eqNLSEMag} and if 
\begin{equation} \label{eq:moment}
  \int_{\R^N} x \abs{u (x)}^2 \dif x = 0,
\end{equation}
then for every linear isometry $R : \R^N\to \R^N$ such that $R_\# A = A$, we have
\[
  u \circ R = u.
\]
\end{proposition}

The moment condition \eqref{eq:moment} makes sense. Indeed by classical regularity estimates
\(u\) is continuous and tends to \(0\) at infinity. Indeed, since \(u\) is a solution to the 
equation \eqref{eqNLSEMag}, we have, by Kato's inequality,
\[
 -\Delta \abs{u} + \abs{u} \le \abs{u}^p,
\]
so that the function \(u\) decays exponentially at infinity and thus the integral in \eqref{eq:moment} converges absolutely.

\begin{proof}[Proof of Proposition~\ref{prop:symmetric-solution}]
Let $\varepsilon>0$ be given by Theorem~\ref{theoremUniqueness} so that uniqueness up to magnetic translations and rotations in $\C$ holds for the equation \eqref{eqNLSEMag}. Consider the function $v = u \circ R$. Since we have assumed that $R_\# A = A$, we deduce that $v$ is also a solution of \eqref{eqNLSEMag} and therefore from Theorem~\ref{theoremUniqueness} that $v (x) = \e^{i \theta} \tau^A_h u$, for some $\theta \in \mathbb{R}$ and $h \in \mathbb{R}^N$. By change of variables, we have 
\[
\begin{split}
  \int_{\R^N} x \abs{u (x)}^2\dif x &
   = \int_{\R^N} R (y) \abs{v (y)}^2\dif y = \int_{\R^N} R (y) \abs{u (y - h)}^2\dif y\\
   &= \int_{\R^N} R (z + h) \abs{u (z)}^2 \dif z = R \Bigl(\int_{\R^N} z \abs{u (z)}^2 \dif z \Bigr) + R (h) \int_{\R^N} \abs{u}^2.
\end{split}
\]
From this and using the assumption, we deduce that $R (h) = 0$, so that $h=0$ since $R$ is an isometry and therefore $u( R(x)) = \e^{i \theta} u(x)$.

Let $u_0$ be the positive and radial solution of \eqref{eq:limitproblem}. By a change of variables, we have
\[
  \int_{\R^N} u_0(x) \, u(x) \, \dif x  = \int_{\R^N} u_0( R y) \, u (R y) \, \dif y = \e^{i \theta} \int_{\R^N} u_0(y) u(y) \dif y.
\]
Moreover, taking $\varepsilon$ smaller if necessary, we infer from Proposition~\ref{propositionContinuityGroundState} that
\[
  \int_{\R^N} u_0 u \ne 0.
\]
This clearly implies that $\e^{i \theta} = 1$.
\end{proof}

\begin{remark} \label{rem:A}
In fact Proposition~\ref{prop:symmetric-solution} implies directly the seemingly stronger statement that for every linear isometry $R : \R^N\to \R^N$ such that $\abs{A \circ R}^2 = \abs{A}^2$ as quadratic forms, we have
\(u \circ R = u\).
This follows from the structure of the group \(G\) of isometries that satisfy the condition. We define the linear operator \(\Hat{A} \in \Lin (\Rset^N, \Rset^N)\) so that \(\Hat{A} (v) \cdot w = A (v)[w]\) ($\cdot$ being the standard scalar product in $\R^N$). We remark that $\Hat{A}$ is in fact an antisymmetric matrix since $A$ is skew-symmetric.
The operator \(-\Hat{A}^2\) is self-adjoint and semi-definite positive. 
We set \(W_\lambda = \ker (\lambda^2 + \Hat{A}^2)\) the eigenspaces corresponding to the eigenvalues $\lambda^2 \geq 0$. Moreover, we have that $\mathbb{R}^N = \oplus_{\lambda} W_\lambda$ and 
$\abs{A(x)}^2 = \sum_{\lambda \neq 0} \lambda^2 \abs{P_{W_\lambda}(x)}^2$, where $P_{W_\lambda}$ is the projection on $W_\lambda$.
We have \(R_\# A\) if and only if \(\Hat{A} \circ R = R \circ \Hat{A}\). 
In particular, we have \(\Hat{A}^2 \circ R = R \circ \Hat{A}^2\) and thus \(R (W_\lambda) = W_\lambda\). The group of isometries can thus be written as a product of groups acting on \(W_\lambda\).

If \(\lambda = 0\), we have \(\Hat{A} = 0\) on \(W_0\) and thus the condition \(\Hat{A} \circ R = R \circ \Hat{A}\) is trivially satisfied and the group \(G\) acts on \(W_0\) as the orthogonal group \(O (W_0)\). 
If \(\lambda \ne 0\), the space \(W_\lambda\) has the structure of a complex vector space defined as follows: if \(\alpha, \beta \in \Rset\) and \(w \in W_\lambda\), we set \((\alpha + i \beta) w = \alpha w + \beta A(w)/\lambda\). 
The condition \(\Hat{A} \circ R = R \circ \Hat{A}\) means then that map \(R\) is \(\C\)--linear. 
The group \(G\) acts thus on \(W_\lambda\) as \(U (W_\lambda)\). 
We have thus \(G \simeq O (W_0) \times \prod_{\lambda \neq 0} U (W_\lambda)\). The symmetry announced in the introduction follows from the fact that the linear groups \(O (W_\lambda)\) and \(U (W_\lambda)\) have the same orbits.
\end{remark}

\subsection{Decoupling of the linear operator}

The next lemma shows that the symmetry obtained in Proposition~\ref{prop:symmetric-solution} is strong enough to cancel out the coupling in \eqref{eqNLSEMag}.

\begin{lemma}%
[Decoupling by symmetry]
\label{prop:symmetriesCoupling}
Let \(u \in H^1_\mathrm{loc} (\Rset^N, \C)\) and assume that \(A \in \Lin (\Rset^N, \bigwedge^1 \Rset^N)\) is skew-symmetric.
If for every linear isometry $R : \R^N\to \R^N$ such that $R_\# A = A$, we have
\[
  u \circ R = u ,
\]
then 
\[
 \scalprod{i A u}{D u} = 0 \qquad \text{almost everywhere in \(\Rset^N\).}
\]
\end{lemma}
\begin{proof}
Let \(\Hat{A} \in \Lin (\Rset^N, \Rset^N)\) be the linear operator, given in Remark~\ref{rem:A}, representing \(A\) with respect to the Euclidean metric in the space \(\Rset^N\), and denote by \(W_\lambda = \ker (\lambda^2 + \Hat{A}^2) \subseteq \Rset^N\) the eigenspaces of $-\hat{A}^2$, corresponding to eigenvalues $\lambda^2  \geq 0$. We also have that $\Hat{A} (W_\lambda) = W_\lambda$. Moreover, if \(P_{W_\lambda}\) denotes the orthogonal projection on \(W_\lambda\), we have $P_{W_\lambda} \circ \Hat{A} = \Hat{A} \circ P_{W_\lambda}$.

If \(\lambda = 0\) and if \(w \in W_\lambda\), then \(\abs{\Hat{A} (w)}^2 = -\Hat{A}^2 (w) \cdot w = \lambda^2 \abs{w}^2 = 0\). Hence, \(\Hat{A} \circ P_{W_\lambda} = 0\) and thus 
\[
 D u (x) [\Hat{A} (P_{W_\lambda}(x)] = 0.
\]

If \(\lambda \ne 0\). We define for \(\theta \in \Rset\) the linear operator $R_\lambda(\theta) \, : \, \mathbb{R}^N \to \mathbb{R}^N$ by
\[
 R_\lambda (\theta)
 = (I - P_{W_\lambda}) + \bigl(\cos \theta \,\text{ I} + \sin \theta \,\Hat{A}/\lambda\bigr) \circ P_{W_\lambda}
\]
where \(I : \Rset^N \to \Rset^N\) is the identity map.
We first show that for every \(\theta \in \Rset\), the map \(R_\lambda (\theta)\) is an isometry.
We observe that for every \(\theta \in \Rset^N\) and \(v \in \Rset^N\), we have 
\[
  \abs{R_\lambda (\theta)(v)}^2 = \abs{(I - P_{W_\lambda})(v)}^2 + \bigabs{\cos \theta P_{W_\lambda} (v) + \sin \theta \,\Hat{A} (P_{W_\lambda} (v))}^2.
\]
Since for every \(w \in W_\lambda\), \(\abs{\Hat{A} (w)}^2 = \lambda^2 \abs{w}\) and since \(\Hat{A}(w) \cdot w = 0\) by antisymmetry of $\Hat{A}$, we conclude that \(R_\lambda (\theta)\) is an isometry of the Euclidean space \(\Rset^N\).
Next we show that \(R_\lambda(\theta)_\# A = A\), or equivalently that $R_\lambda(\theta) \circ \Hat{A} = \Hat{A} \circ R_\lambda(\theta)$. Since $R_\lambda(\theta)^{-1} = R_\lambda(-\theta)$, we have to compute
\[
  R_\lambda (-\theta) \circ \Hat{A} \circ R_\lambda (\theta) 
  = \Hat{A} + \bigl( 2 (\cos \theta - 1) + ((\cos \theta - 1)^2 + (\sin \theta)^2)\bigr) \Hat{A} \circ P_{W_\lambda} = \Hat{A}.
\]
By our assumption, we have 
\[
 u (R_\lambda (\theta)(x)) = u (x).
\]
By differentiating with respect to \(\theta\) at \(0\), we deduce that 
\[
 D u (x)[\Hat{A} (P_{W_\lambda} (x))] = 0.
\]
This concludes the proof.
\end{proof}

\subsection{Decoupled problem}

We now consider the decoupled problem
\begin{equation}
\label{eq:neweq}
 -\Delta u + (1 + \abs{A}^2) u = \abs{u}^{p - 2} u \qquad \text{ in } \mathbb{R}^N.
\end{equation}
Solutions of this problem are critical points of the functional 
\[
 \Tilde{\mathcal{I}}_{\abs{A}^2} = \frac{1}{2} \int_{\Rset^N} \abs{D u}^2 + (1 + \abs{A}^2)\abs{u}^2
 - \frac{1}{p} \int_{\Rset^N} \abs{u}^p,
\]
defined on the natural function space 
\[
 \Tilde{H}^1_{\abs{A}^2} (\Rset^N)
 = \Bigl\{ u \in H^1 (\Rset^N) \st \int_{\Rset^N} \abs{A}^2 \abs{u}^2 < \infty \Bigr\}.
\]

\begin{proposition}
[Symmetry of groundstates of the modified problem]
\label{propSymmetryModified}
If \(u\) is a groundstate solution of \eqref{eq:neweq} and if \(A \in \Lin (\Rset^N, \bigwedge^1 \Rset^N)\) can be written as
\[
  \abs{A (x)}^2 = \sum_{j = 1}^k \lambda_j^2 \abs{P_{W_j} (x)}^2,
\]
with \(\lambda_1, \dotsc, \lambda_k \in \mathbb{R}_0\) and \(\Rset^N = W_0 \oplus \dotsb \oplus W_k\),
then there exist \(a \in W_0\) and \(v : [0, \infty)^{k + 1} \to [0, \infty)\) nonincreasing with respect to each of its variable such that 
\[
  u(x) = v (\abs{P_{W_0} (x) - a}, \abs{P_{W_1} (x)}, \dotsc, \abs{P_{W_k} (x)}).
\]
\end{proposition}

The assumption on \(A\) can be reformulated by saying that \(W_0\) is the (possibly trivial) kernel of $- \Hat{A}^2$, while $W_i$ are the eigenspaces corresponding to positive eigenvalues $\lambda_i^2 > 0$, for $i = 1, \dotsc, k$. 
Since \(\abs{A}^2\) is a nonnegative quadratic form, this decomposition is always possible (see Remark~\ref{rem:A}).

In order to prove Proposition~\ref{propSymmetryModified}, we will rely on the notion of polarisation (or two-point rearrangement). If $H \subset \mathbb{R}^N$ is a closed half-space, and $\sigma_H$ is the reflection with respect to $\partial H$, the polarization of the function $u : \Rset^N \to \mathbb{R}$ is
\begin{equation*}
  u^H(x) 
  = \begin{cases} \max \{ u(x), u(\sigma_H(x) \}, & \text{if \(x \in H\)}, \\
                                 \min \{ u(x), u(\sigma_H(x) \}, & \text{if \(x \in \Rset \setminus H\)}.
  \end{cases} 
\end{equation*}

We shall rely on the following lemma.

\begin{lemma}%
[Behavior of the potential under polarization]%
\label{lemma:polarisation}
Let $H$ be a closed half-space of $\mathbb{R}^N$. 
If \(u : \Rset^N \to \R\) is nonnegative and measurable and if for every \(x \in H\), \(\abs{A(\sigma_H (x))}^2 \ge \abs{A (x)}^2\), then 
\[
\int_{\mathbb{R}^N} \abs{A}^2\, \abs{u^H}^2 
\leq \int_{\mathbb{R}^N} \abs{A}^2\, \abs{u}^2;
\]
equality holds if and only if either \(u^H = u\) or \(\abs{A \circ \sigma_H}^2 = \abs{A}^2\).
\end{lemma}

\begin{proof}
We have 
\[
 \int_{\mathbb{R}^N} \abs{A}^2\, \abs{u}^2 - \int_{\mathbb{R}^N} \abs{A}^2\, \abs{u^H}^2 
 = \int_{E} \bigl(\abs{A (\sigma_H (x))}^2 - \abs{A (x)}^2 \bigr)\,
 \bigl(\abs{u (\sigma_H (x))}^2 - \abs{u (x)}^2 \bigr)\dif x,
\]
where the set \(E \subset \Rset^N\) is defined by
\[
  E = \bigl\{ x \in H \st u (x) < u(\sigma_H(x))\bigr\}.
\]
The integral on the right-hand side is clearly nonnegative.

Assume now that the integral on the right-hand side is \(0\) and that \(u^H \not\equiv u\).
In particular, the set \(E\) has positive Lebesgue measure and for every \(x \in E\), 
we have \(\abs{A (\sigma_H (x))}^2 = \abs{A (x)}\). Since the function \(\abs{A}^2\) is a quadratic form, 
this implies that \(\abs{A \circ \sigma_H}^2 = \abs{A}^2\) on \(\Rset^N\).

The other possibility is that $E$ has zero Lebesgue measure, so that $u \equiv u^H$.
\end{proof}

A function which is invariant under polarizations with respect to large set of hyperplanes is known to be symmetric, see \citelist{\cite{BrockSolynin2000}*{Lemma 6.3}}.

\begin{lemma}[Symmetry by invariance under polarization]
\label{lemma:schwarz}
Let $u \in L^2 (\mathbb{R}^N)$ be nonnegative and let \(W \subset \Rset^N\) be a linear subspace of \(\Rset^N\). 
If for every closed half-space $H \subset \mathbb{R}^N$ such that $W^\perp \subset \operatorname{int} H$, $u^H = u$, then 
for every $x, y \in \mathbb{R}^N$ such that $x - P_{W} (x) = y - P_W (y)$ and $\abs{P_W (x)} \le \abs{P_W (y)}$, $u (x) \ge u (y)$.
\end{lemma}

Under the weaker condition that either the function or its reflection coincides with the polarization but assumed for a larger set of polarization, the same symmetry holds up to a suitable translation.

\begin{lemma}[Symmetry by invariance or reflection under polarization] \label{lemma:schwarz2}
Let $u \in L^2 (\mathbb{R}^N)$ be nonnegative and let \(W \subset \Rset^N\) be a linear subspace of \(\Rset^N\). 
If for every closed half-space $H \subset \mathbb{R}^N$ such that $W^\perp$ is parallel to $\partial H$, either $u^H = u$ or $u^H = u \circ \sigma_H$, 
then there exists $a \in W$ such that for every $x, y \in \mathbb{R}^N$ such that $x - P_{W} (x) = y - P_W (y)$ and $\abs{P_W (x) - a} \le \abs{P_W (y) - a}$, $u (x) \ge u (y)$.
\end{lemma}

\begin{proof}
The proof is a straightforward adaptation of the corresponding statement when \(W = 0\) \citelist{\cite{MorozVanSchaftingen2013-2}*{Lemma 5.6}\cite{VanSchaftingen2009}*{Lemma 5}}. 
The point \(a \in W\) is a minimizer of the function 
\[
 x \in W^\perp \longmapsto \int_{\Rset^N} w (y - x)\, u (y) \dif y \in \Rset,
\]
where the function \(w \in L^2 (\Rset^N)\) is positive, radial and radially decreasing.
\end{proof}

\begin{proof}[Proof of proposition~\ref{propSymmetryModified}]
We first prove that \(u\) has constant phase. We have by the formula for the derivative of the modulus \eqref{eqDerivativeModulus}
\begin{equation}
\label{eq:IntegrateDia}
  \int_{\Rset^N} \bigabs{\nabla \abs{u}}^2 = \int_{\Rset^N} \bigabs{\scalprod{\nabla u}{\sign u}}^2
  \le \int_{\Rset^N} \abs{\nabla u}^2.
\end{equation}
Therefore, for every \(t \in [0, \infty)\),
\[
  \Tilde{\mathcal{I}}_{\abs{A}^2} (t\abs{u}) \le \Tilde{\mathcal{I}}_{\abs{A}^2} (tu),
\]
and it follows, taking $t=1$, that \(\abs{u}\) is a groundstate of the equation \eqref{eq:neweq}.
By classical regularity estimates and the strong maximum principle, the function \(\abs{u}\) is smooth and positive. Since we have equality in \eqref{eq:IntegrateDia}, following the argument of the proof of Proposition~\ref{propositionDiamagneticUniqueness} it follows that \(u = \e^{i \theta} \abs{u}\) for some \(\theta \in \Rset\).
In the sequel of the proof, we shall fix without loss of generality \(\theta = 0\).

Let \(i \in \{1, \dotsc, k\}\). For every closed halfspace \(H \subset \Rset^N\) such that 
\(W_i^\perp \subset \operatorname{int} H\), we have for every \(x \in \operatorname{int} H\)
\[
  \abs{P_{W_i} (x)} < \abs{P_{W_i} (\sigma_H (x))},
\]
and for each \(j \in \{0, \dotsc, k\} \setminus \{i\}\),
\[
 \abs{P_{W_j} (x)} = \abs{P_{W_j} (\sigma_H (x))}.
\]
Since \(\lambda_i^2 > 0\) and \(\lambda_j^2 > 0\), we have  for every $x \in H$
\[
  |A (\sigma_H (x))| > |A (x)|.
\]
It follows thus from Lemma~\ref{lemma:polarisation} that either $u = u^H$ or for every \(t \in (0, \infty)\),
\[
 \Tilde{\mathcal{I}}_{\abs{A}^2} (tu^H) < \Tilde{\mathcal{I}}_{\abs{A}^2} (tu) \le \Tilde{\mathcal{I}}_{\abs{A}^2} (u).
\]
The last situation cannot occur since $u$ is a groundstate of \eqref{eq:neweq}. Therefore $u=u^H$.
By Lemma~\ref{lemma:schwarz}, we conclude that for every $x, y \in \Rset^N$ such that for $j \in \{0, \dotsc, k\} \setminus \{i\}$, $P_{W_j} (x) = P_{W_j} (y)$ and $\abs{P_{W_i} (x)} \le \abs{P_{W_i} (y)}$, then $u (x) \ge u (y)$.

We now treat the case $i = 0$, following the strategy of Bartsch, Weth and Willem in \cite{BartschWethWillem} (see also \cite{VanSchaftingenWillem2008}).
We observe that if \(\partial H\) is parallel to \(W_0^\perp\), we have for each \(x \in \Rset^N\), 
\[
  \abs{A (\sigma_H (x))}^2 = \abs{A (x)}^2,
\]
and then by Lemma~\ref{lemma:polarisation} for every \(t \in (0, \infty)\),
\[
 \Tilde{\mathcal{I}}_{\abs{A}^2} (tu^H) \le \Tilde{\mathcal{I}}_{\abs{A}^2} (tu)
 \le \Tilde{\mathcal{I}}_{\abs{A}^2} (u).
\]
In particular, since \(u\) is a groundstate of the problem \eqref{eq:neweq}, its polarization 
\(u^H\) is also a groundstate. We then have 
\[
  \abs{u - u \circ \sigma_H} = 2 u^H - u - u \circ \sigma_H \qquad \text{almost everywhere in } H.
\]
Therefore, using the equation of $u$, $u^H$, 
\[
 -\Delta \abs{u - u \circ \sigma_H} + (1 + \abs{A}^2) \abs{u - u \circ \sigma_H}
 = 2 \abs{u_H}^{p-1} - \abs{u}^{p-1} - \abs{u \circ \sigma_H}^{p-1} 
 \ge 0, \qquad \text{ on } H.
\]
By the strong maximum principle, this implies that \(\abs{u - u \circ \sigma_H}\) does not vanish inside $H$, so that \(u - u \circ \sigma_H\) does not change sign inside $H$. We have therefore either \(u^H = u\) or \(u^H = u \circ \sigma_H\).
By Lemma~\ref{lemma:schwarz2}, there exists $a \in W_0$, such that, for $j \in \{1, \dotsc, k\}$, $P_{W_j} (x) = P_{W_j} (y)$ and $\abs{P_{W_0} (x) - a} \le \abs{ P_{W_0} (y) - a}$, then $u (x) \ge u (y)$.
\end{proof}                                                                   

\subsection{Conclusion}

We first prove that when the magnetic field is weak enough, the groundstates of \eqref{eqNLSEMag} 
correspond to groundstates of \eqref{eq:neweq}.

\begin{proposition}%
[Groundstates of the magnetic problem are groundstates of the modified problem]
\label{proposition:realGS}%
Let \(\varepsilon > 0\) be given by Theorem~\ref{theoremUniqueness}.
If \(A\) and \(u\) are as in the conclusion of Theorem~\ref{theoremUniqueness}, then there exists \(a \in \Rset^N\) such that \(\tau^A_a u\) is a groundstate of the problem \eqref{eq:neweq}.
\end{proposition}
\begin{proof}
We set 
\begin{equation*}
a = - \frac{\displaystyle\int_{\mathbb{R}^N} x \abs{u(x)}^2 \dif x}{\displaystyle \int_{\mathbb{R}^N} \abs{u (x)}^2 \dif x}
\end{equation*}
to obtain \eqref{eq:moment}, that is 
\[
 \int_{\mathbb{R}^N} x \abs{\tau^A_a v (x)}^2 \dif x = 0.
\]
The function $v = \tau^A_a u$ is then a groundstate of \eqref{eqNLSEMag} verifying the assumptions of Proposition~\ref{prop:symmetric-solution}.

We first note that by the symmetry properties of Proposition~\ref{prop:symmetric-solution} and the decoupling property of Lemma~\ref{prop:symmetriesCoupling}, we have for every \(t \in (0, \infty)\), 
\[
 \mathcal{I}_A (t v) = \Tilde{\mathcal{I}}_{\abs{A}^2} (t v),
\]
and thus 
\[
 \max_{t \ge 0} \mathcal{I}_A (t v) \ge \inf \, \bigl\{ \max_{t \ge 0} \Tilde{\mathcal{I}}_{\abs{A}^2} (t w) 
 \st w \in \Tilde{H}^1_{\abs{A}^2} (\Rset^N) \setminus \{0\} \bigr\}.
\]

On the other hand, if \(\tilde{v}\) is a groundstate of \eqref{eq:neweq}, then by the symmetry properties of groundstates of Proposition~\ref{propSymmetryModified} and again by the decoupling property of Lemma~\ref{prop:symmetriesCoupling}, for every \(t \in [0, \infty)\),
\[
 \mathcal{I}_A (t \tilde{v}) = \Tilde{\mathcal{I}}_{\abs{A}^2} (t \tilde{v}),
\]
and thus 
\[
 \max_{t \ge 0} \Tilde{\mathcal{I}}_{\abs{A}^2}(t \tilde{v}) \ge \inf\, \bigl\{ \max_{t \ge 0}  \mathcal{I}_A (t w) 
 \st w \in H^1_{A} (\Rset^N,\C) \setminus \{0\} \bigr\}.
\]
This implies that the groundstates levels of both problems coincide and that 
$v=\tau_a^A u$ is a groundstate of \eqref{eq:neweq}.
\end{proof}

\begin{proof}[Proof of Theorem~\ref{theoremSymmetry}]
The conclusion follows from the proof of Proposition~\ref{proposition:realGS}, in which the required symmetry was deduced from Proposition~\ref{prop:symmetric-solution}, and by using the uniqueness result of Theorem~\ref{theoremUniqueness}.
\end{proof}

\section{Asymptotic behavior of groundstates} \label{section:asymptotic}

In this section we first prove Theorem~\ref{thm:asymptotic} about the equivalence of the asymptotics of the solutions of the magnetic semilinear Schr\"odinger equation \eqref{eqNLSEMag} and a linear problem. We then apply this result to describe completely the asymptotics in the planar case \(N = 2\) and obtain a Gaussian upper bound in higher dimensions \(N \ge 3\).

\subsection{Comparison with solution of a linear problem}
Theorem~\ref{thm:asymptotic} will follow from the following proposition about the 
equivalence of asymptotics between linear and semilinear problems.

\begin{proposition}%
[Decay of solutions to a semilinear problem with a potential]
\label{prop:nonLinDecay}
For $R > 0$, let \(V \in C (\Rset^N \setminus B_R)\) be such that \(\inf_{\Rset^N \setminus B_R} V > 0\)
and \(p \in (2, \infty)\).
If \(u, v \in C^2 (\Rset^N \setminus B_R)\) satisfy 
\[
\left\{
\begin{aligned}
 -\Delta u + V u & = u^{p - 1} & & \text{in \(\Rset^N \setminus \overline{B_R }\)},\\
 u & > 0 & & \text{on \(\Rset^N \setminus B_R\)},\\
 u (x) & \to 0 & & \text{as \(\abs{x} \to \infty\)}
\end{aligned}
\right.
\]
and 
\[
\left\{
\begin{aligned}
 -\Delta v + V v & = 0 & & \text{in \(\Rset^N \setminus \overline{B_R}\)},\\
 v & > 0 & & \text{on \(\Rset^N \setminus B_R\)},\\
 v (x) & \to 0 & & \text{as \(\abs{x} \to \infty\)},
\end{aligned}
\right.
\]
then there exist \(C, c \in (0, \infty)\) such that 
\[
  c v \le u \le C v \qquad \text{in \(\Rset^N \setminus B_R\)}. 
\]
\end{proposition}

Here \(B_R \) denotes the \emph{open ball} of radius \(R\) centered in $0$.

\begin{proof}[Proof of Proposition~\ref{prop:nonLinDecay}]%
We first observe that 
\[
 -\Delta u + V u \ge 0 \qquad \text{in \(\Rset^N \setminus \overline{B_R}\)}.
\]
Since the functions \(u\) and \(v\) are both continuous and positive on \(\partial B_R \),
by Weiserstrass's Theorem, there exists \(c \in (0, \infty)\) 
such that \(cv \le u\) on \(\partial B_R\). The inequality extends to \(\Rset^N \setminus B_R\) by the maximum principle.

For the converse inequality, we observe that, since \(p > 2\),
\[
 -\Delta \bigl(u^{p - 1}\bigr) = - (p - 1) u^{p - 2} \Delta u - (p - 1) (p - 2) u^{p - 3} \abs{\nabla u}^2 \le - (p - 1) u^{p - 2} \Delta u\qquad \text{in \(\Rset^N \setminus \overline{B_R}\)}.
\]
Therefore, we have for \(\lambda > 0\), 
\[
 -\Delta \bigl(u + \lambda u^{p - 1}\bigr) + V \bigl(u + \lambda u^{p - 1}\bigr)
 \le -u^{p - 1} (\lambda (p - 2) V - 1 - \lambda (p-1) u^{p - 2}) \qquad \text{in \(\Rset^N \setminus \overline{B_R}\)}.
\]
Since \(\lim_{\abs{x} \to \infty} u (x) = 0\) and since \(u\) is continuous, we have, if \(\lambda\) and $R$ are large enough,
\[
 -\Delta \bigl(u + \lambda u^{p - 1}\bigr) + V \bigl(u + \lambda u^{p - 1}\bigr)
 \le 0\qquad \text{in \(\Rset^N \setminus \overline{B_R}\)}.
\]
We choose \(C \in (0, \infty)\) sucht that \(u + \lambda u^{p - 1} \le C v\). 
By the maximum principle, it follows that 
\(u \le u + \lambda u^{p - 1} \le C v\) in \(\R^N \setminus B_R\), and the conclusion follows.
\end{proof}

We are now in position to prove Theorem~\ref{thm:asymptotic}.

\begin{proof}[Proof of Theorem~\ref{thm:asymptotic}]
From the reformulation of Proposition~\ref{proposition:realGS} of the nonlinear Schr\"o\-din\-ger equation \eqref{eqNLSEMag} as the decoupled problem~\eqref{eq:neweq}, the solution \(u\) is a groundstate of \eqref{eq:neweq}. In view of the properties of the groundstates of \eqref{eq:neweq} of Proposition~\ref{propSymmetryModified}, we can assume that \(u\) is real and positive. The conclusion follows then from Proposition~\ref{prop:nonLinDecay}.
\end{proof}

\subsection{The planar case}

The asymptotics can be described precisely in the two-dimensio\-nal case. To prove this, we need to find the exact asymptotics of the linear problem, in order to use Proposition~\ref{prop:nonLinDecay}, we rely on the next lemma  \cite{MorozVanSchaftingen2013}*{Proposition 6.1} (see also \cite{Agmon1984}*{Theorem 3.3})

\begin{lemma}
\label{lemmaAsymptoticsLinear2d}
Let $R > 0$ and let \(V (r) = 1 + \frac{B^2}{4} r^2\), then there exists a nonnegative radial function $H : \mathbb{R}^2 \setminus \overline{B_R} \to \mathbb{R}$ such that
\begin{equation*}
- \Delta H (x)+ V (\abs{x}) H (x) = 0, \qquad \text{for  } x \in \mathbb{R}^2 \setminus \overline{B_{R}},
\end{equation*}
and $H$ has the following asymptotics as \(\abs{x} \to \infty\)
\[
  H(x) = \frac{\exp \Bigl(-\displaystyle \int_0^{\abs{x}} \sqrt{V (s)} \dif s\Bigr)}{\abs{x}^{\frac{1}{2}} V (\abs{x}) ^{\frac{1}{4}}} \bigl(1 + o (1)\bigr),
\]
where for every \(x \in \Rset^2\)
\[
 \int_0^{\abs{x}} \sqrt{V (s)} \dif s
 = \tfrac{1}{2} \abs{x} \sqrt{1 + \tfrac{B^2}{4} \abs{x}^2} + \tfrac{1}{\abs{B}} 
 \ln \Bigl(\tfrac{\abs{B}}{2} \abs{x} + \sqrt{1 + \tfrac{B^2}{4} \abs{x}^2}\Bigr).
\]
\end{lemma}
The last integral follows by direct calculation and using the identity $\operatorname{arsinh} s = \log (s + \sqrt{1+s^2})$. We adapt the proof to our particular case for the reader's convenience.

\begin{proof}
To prove this, we proceed in as the proof of \cite{MorozVanSchaftingen2013}*{Proposition 6.1}  and we introduce for \(\tau \in \R\) the functions \(\Phi_\tau : \mathbb{R}^{2} \setminus \overline{B_\rho} \to \mathbb{R}\) defined for \(x \in \mathbb{R}^{2} \setminus \overline{B_\rho} \to \mathbb{R}\) by
\[
\Phi_\tau(x) = \abs{x}^{- \frac{1}{2}} V(\abs{x})^{- \frac{1}{4}} \exp \Bigl(\frac{- \tau \abs{x}^{- \beta}}{\beta} - \int_{\rho}^{\abs{x}} \sqrt{V(s)} \dif s \Bigr),
\]
where $\beta > 0$ is a parameter that will be fixed later. By computing explicitly the Laplacian, we get
\[
- \Delta \Phi_\tau + V(\abs{x}) \Phi_\tau = \Bigl( \frac{2 \tau (V(\abs{x}))^{1/2} + \omega_\tau(\abs{x})}{\abs{x}^{1+\beta}} \Bigr) \Phi_\tau,
\]
where $\omega_\tau : \mathbb{R}^{2} \setminus \overline{B_\rho} \to \mathbb{R}$ is given by
\begin{multline*}
\omega_\tau(\abs{x}) = - \frac{\,\abs{x}^{\beta - 1}}{4} + \frac{\tau (\beta + 1)}{\abs{x}} + \frac{V''(\abs{x}) \abs{x}^{1+\beta}}{4 V(\abs{x})} - \frac{5 (V'(\abs{x}))^2 \abs{x}^{1 + \beta}}{16 V^2(\abs{x})} \\
+ \frac{\tau V'(\abs{x})}{2 V(\abs{x})} - \frac{\tau^{2}}{\,\abs{x}^{\beta + 1}}.
\end{multline*}
If we choose $0 < \beta < 1$, we obtain that $\lim_{\abs{x} \to + \infty} \omega_\tau(\abs{x}) = 0$, thanks to the explicit shape of $V$.

By choosing $\tau_- < 0$ and $\tau_+ > 0$, we conclude that $\Phi_{\tau_-}$ and $\Phi_{\tau_+}$ are respectively  sub- and super-solutions of the equation in $\R^2 \setminus B_R$ for $R$ large enough. Moreover,
\[
\lim_{\abs{x} \to + \infty} \frac{\Phi_{\tau_-}(\abs{x})}{\Phi_{\tau_+}(\abs{x})} = 1.
\]
To conclude to the existence of a solution $H$, we proceed as in the proof of \cite{MorozVanSchaftingen2013}*{Proposition 6.1}. The asymptotic behavior at infinity follows from the one of $\Phi_{\tau_-}$ and $\Phi_{\tau_+}$.
\end{proof}

\begin{proposition}
\label{proposition2dAsymptotics}
If $A \in \Lin(\R^N,\bigwedge^1 \R^N)$ is given by $A (x)[v] = \frac{B}{2} x \wedge v$ and $ \abs{B} \le \varepsilon$ with $\varepsilon > 0$ given in Theorem~\ref{theoremUniqueness} and if $u$ is a solution of \eqref{eqNLSEMag} such that $\mathcal{I}_{A} (u) \le \mathcal{E} (0)+\varepsilon$,
then there exist \(a \in \Rset^2\) and \(c \in \C \setminus \{0\}\) such that, as \(\abs {x} \to \infty\),
\[  
u (x)
  = \frac{\exp \bigl(- \frac{\abs{B}\,\abs{x - a}^2}{4}\bigr)}{\abs{x - a}^{\frac{1}{2}} \bigl(1 + \abs{B} \, \abs{x - a}\bigr)^{\frac{1}{2}+\frac{1}{\abs{B}}}} (c + o (1)),
\]
\end{proposition}

Here, $a \wedge b \defeq a_1 b_2 - a_2 b_1$ for $a = (a_1, a_2), b = (b_1, b_2) \in \R^2$. The magnetic potential $A$ is therefore linear and skew-symmetric. 

\begin{proof}[Proof of Proposition~\ref{proposition2dAsymptotics}]
We first observe that if the solution \(u\) is radial, then for \(R > 0\) sufficiently large
\(C\) and \(c\) become close to each other in the conclusion of Proposition~\ref{prop:nonLinDecay} 
and thus the superior and inferior limits coincide.

From Lemma~\ref{lemmaAsymptoticsLinear2d} on the asymptotics of solutions to the linear problem, we deduce
\[
  u (x)
  = \frac{\exp \bigl(- \frac{\abs{x - a}}{2} \sqrt{1 + \frac{B^2}{4} \abs{x - a}^2}\bigr)}{\abs{x - a}^{1/2} (1 + \frac{B^2}{4} \abs{x - a}^2)^{\frac{1}{4}} (\sqrt{1 + \frac{B^2}{4} \abs{x - a}^2} +  \frac{\abs{B}}{2} \abs{x - a})^\frac{1}{\abs{B}}} (c + o (1)),
\]
and the conclusion follows by standard properties of limits and in view of the fact that, as \(\abs{x - a} \to \infty\) 
\[
\begin{split}
  \frac{\abs{x - a}}{2} \sqrt{1 + \frac{B^2}{4} \abs{x - a}^2}
  &=\frac{\abs{B} \, \abs{x - a}^2}{4} \sqrt{1 + \frac{4}{B^2 \abs{x - a}^2}}\\
  &= \frac{\abs{B} \, \abs{x - a}^2}{4} \biggl(1  + \frac{2}{B^2 \abs{x - a}^2} + O \Bigl(\frac{1}{\abs{x - a}^4} \Bigr)\biggr)\\
  &= \frac{\abs{B} \, \abs{x - a}^2}{4} + \frac{1}{2 \abs{B}} + o (1),
\end{split}
\]
and that, again as \(\abs{x - a} \to \infty\),
\[
\begin{split}
 \ln \Bigl(\frac{\abs{B}\abs{x - a}}{2} + \sqrt{1 + \frac{B^2 \abs{x - a}^2}{4}} \Bigr)
 &= \ln (\abs{B} \, \abs{x - a})
 + \ln \Bigl(\frac{1}{2} + \frac{1}{2}\sqrt{1 + \frac{4}{B^2 \abs{x - a}^2}}\Bigr)\\
 &= \ln (\abs{B} \, \abs{x - a}) + o (1).\qedhere
\end{split}
\]
\end{proof}
We remark that in the limit $|B| \to 0$, we recover the asymptotics of the groundstates of the nonlinear Schrödinger equation without magnetic field.

\subsection{Higher-dimensional case}

In higher dimensions \(N \ge 3\), the operator \(-\Delta + (1 + \abs{A}^2)\) is \emph{anisotropic} in general, that is, it does not commute with linear isometries of \(\Rset^N\). 
It is thus technically difficult to have closed-form expression for the asymptotics. 
It is however possible to obtain Gaussian asymptotics in the transversal directions to the magnetic field.

\begin{proposition}
\label{propositionHigherDimensionGaussian}
If \(u\) is a groundstate solution of \eqref{eqNLSEMag} and if \(A \in \Lin (\Rset^N, \bigwedge^1 \Rset^N)\) can be written as
\[
  \abs{A (x)}^2 = \sum_{j = 1}^k \lambda_j^2 \abs{P_{W_j} (x)}^2,
\]
with \(\lambda_1, \dotsc, \lambda_k \in \mathbb{R}_0\) and \(\Rset^N = W_0 \oplus \dotsb \oplus W_k\), then there exists $c > 0$ such that
\[
 |u (x)| \le (c + o (1)) \exp \Bigl(- \sum_{j = 1}^k \frac{\abs{\lambda_j}}{2} \abs{P_{W_j} (x)}^2\Bigr).
\]
\end{proposition}

The asymptotic estimates can be written in terms of the square-root of a linear operator as
\[
  |u (x)| \le (1 + o (1)) \exp \Bigl( - \frac{x \cdot \sqrt{A^* \circ A} x}{2} \Bigr),
\]
where \( \sqrt{A^* \circ A}\) is semi-definite positive and \((\sqrt{A^* \circ A})^2 = A^* \circ A = - \Hat{A}^2 \), $\Hat{A}$ being defined in Remark~\ref{rem:A}.

\begin{proof}[Proof of Proposition~\ref{propositionHigherDimensionGaussian}]
This follows from the asymptotics of Theorem~\ref{thm:asymptotic} and the fact that if 
\[
 w (x) = \exp \Bigl(- \sum_{j = 1}^k \frac{\abs{\lambda_j}}{2} \abs{P_{W_j} (x)}^2\Bigr),
\]
then for each \(x \in \Rset^N \setminus \{0\}\), 
\[
 -\Delta w (x) + (1 + \abs{A (x)}^2) w (x)
 = \Bigl(1 + \sum_{j = 1}^k \abs{\lambda_j} \dim W_j \Bigr) w (x) \ge 0.\qedhere
\]
\end{proof}

In particular, if \(N = 3\) and \(B \in \bigwedge^2 \Rset^3 \simeq \Rset^3\), we can take \(\lambda_1 = \abs{B}/2\) and \(\lambda_2 = 0\). We have \(\abs{P_{W_1} (x)} = {\abs{B \times x}}/{{\abs{B}}}\)
and we obtain
\eqref{eq3dDecay}.

\section{Differentiability} \label{sec:diff}

We assume that $\abs{dA} \le \varepsilon$, where \(\varepsilon\) is given by Theorem~\ref{theoremUniqueness} and we choose \(u_A \in H^1_A (\Rset^N, \C)\) to be the \emph{real even} groundstate of \eqref{eqNLSEMag} with magnetic potential $A$, that is, for every $x \in \R^N$
\[
 u_A (-x)=u_A (x). 
\]
Such a normalization always exists since for every $A \in \Lin (\R^N, \bigwedge^1 \R^N)$ skew-symmetric, $(-I)_\# A = A$. Since the problem \eqref{eqNLSEMag} is invariant under rotations of both the function \(u_A\) and the magnetic potential, we shall apply a rotation so that \(A\) has a more convenient structure. We shall thus assume that \(A\) is invariant under \(SO (2)^{\floor{\frac{N}{2}}}\), where each factor acts on \(2\) canonical coordinates of \(\Rset^N\). 

This implies in particular that \(\abs{A}^2\) diagonalizes with respect to the canonical basis, that is, for every \(x = (x_1, \dotsc, x_N) \in \Rset^N\),
\[
 \abs{A (x)}^2 = \sum_{i = 1}^{\floor{\frac{N}{2}}} \lambda_i^2 \bigl(\abs{x_{2i - 1}})^2 + \abs{x_{2 i}}^2\bigr),
\]
with \(\lambda_1, \dotsc, \lambda_{\floor{\frac{N}{2}}} \in \mathbb{R}\). We note that the $\lambda_i$ can possibly be zero.

We define 
\begin{multline*}
 \mathcal{Q} = \Bigl\{ Q : \Rset^N \to \Rset \st \text{there exist \(\alpha_1, \dotsc, \alpha_{\floor{\frac{N}{2}}} \in \Rset\)} \\[-1.5em]
 \text{such that for each \(x \in \Rset^N\), } Q (x) = \sum_{i= 1}^{\floor{\frac{N}{2}}} \alpha_i (\abs{x_{2i - 1}}^2 + \abs{x_{2 i}}^2)\Bigr\}.
\end{multline*}

\subsection{Differentiability of the groundstate}

If we let \(\mathcal{Q}\) denote the \(N\)--dimensional space of diagonal quadratic forms defined above, we then have \(\abs{A}^2 \in \mathcal{Q}\). The space \(\mathcal{Q}\) is a finite dimensional space endowed with a norm \(\norm{\cdot}\). 
Since the space \(\Lin(\mathbb{R}^N, \bigwedge^1 \mathbb{R}^N)\) is also finite dimensional, all the norms on this space are equivalent and give the \(L^2_\mathrm{loc}\) topology that has been used above; we will thus not specify the convergence in the following.

\begin{proposition}[Differentiability of the groundstate with respect to the magnetic field]
\label{propositionDifferentiability}
If $A_* \in \Lin(\mathbb{R}^N, \bigwedge^1 \mathbb{R}^N)$ is such that \(\abs{dA_*}\le \varepsilon\), where \(\varepsilon > 0\) is given in Theorem~\ref{theoremUniqueness} and if \(\abs{A_*}^2 \in \mathcal{Q}\), then 
\[
 u_A = u_{A_*} + w_{A_*}\bigl[\abs{A}^2 - \abs{A_*}^2\bigr] + o \bigl(\big\|\abs{A}^2 - \abs{A_*}^2\big\|\bigr),
\]
in $H^1_{A_*}(\R^N , \C)$ as $A \to A_*$ in $L^2_{\text{loc}}(\R^N)$,
where $w_{\abs{A_*}^2} \in \Lin (\mathcal{Q}, H^1_{A_*}(\R^N , \C))$ 
is the unique linear map such that for every \(Q \in \mathcal{Q}\) 
\begin{equation} 
\label{eq:wA*}
\left\{
\begin{aligned}
-\Delta_{A_*} w_{A_*}[Q] +  \bigl( 1 - (p - 1) \abs{u_{A_*}}^{p - 2} \bigr) w_{A_*}[Q] &= - Q u_{A_*} &&\text{in \(\Rset^N\)},\\
 \scalprod{D_{A_*} u_{A_*}[h]}{w_{A_*}[Q]}_{H^1_{A_*} (\Rset^N,\C)} &= 0
&& \text{for each \(h \in \Rset^N\),}\\
 \scalprod{i u_{A_*}}{w_{A_*}[Q]}_{H^1_{A_*} (\Rset^N,\C)} &= 0.
\end{aligned}
\right.
\end{equation}
\end{proposition}

\begin{proof}
\resetclaim

\begin{claim} \label{claim:1estimate}
Let $\abs{d A_*} \le \varepsilon$ be small enough and $f \in L^2(\mathbb{R}^N, \mathbb{C})$. If 
\begin{equation*}
\left\{
\begin{aligned}
 \int_{\Rset^N} \scalprod{D_{A_*} u_{A_*}[h]}{f}&= 0
&& \text{for each \(h \in \Rset^N\),}\\
 \int_{\Rset^N} \scalprod{i u_{A_*}}{f}&= 0,
\end{aligned}
\right.
\end{equation*}
then the problem 
\begin{equation*}
\left\{
\begin{aligned}
 -\Delta_{A_*} v + v - (p- 1) \abs{u_{A_*}}^{p - 2} v &= f, && \text{in \(\Rset^N\)},\\
 \scalprod{D_{A_*} u_{A_*}[h]}{v}_{H^1_{A_*}(\Rset^N,\C)} &= 0
&& \text{for each \(h \in \Rset^N\),}\\
 \scalprod{i u_{A_*}}{v}_{H^1_{A_*} (\Rset^N,\C)} &= 0,
\end{aligned}
\right.
\end{equation*}
has a unique solution \(v \in H^1_{A_*}(\mathbb{R}^N, \mathbb{C})\).
Moreover, there exists \(C > 0\) such that \(v\) satisfies the following estimate
\[
  \norm{ v }_{H^1_{A_*}(\mathbb{R}^N, \mathbb{C})} \le C \norm{f}_{L^2(\mathbb{R}^N, \mathbb{C})}.
\]
\end{claim}

\begin{proofclaim}
Let $L_{A_*} \,: \, H^1_{A_*} (\R^N,\mathbb{C}) \to H^1_{A_*} (\R^N, \mathbb{C})$ be the operator defined by 
\[
L_{A_*} u \defeq (p - 1) (-\Delta_{A_*} + I)^{-1} \left( \abs{u_{A_*}}^{p - 2} u \right).
\]
The equation of $v$ then reads 
\[
 v - L_{A_*} (v) = (-\Delta_{A_*} + I)^{-1}f.
\]
Therefore, thanks to the orthogonality condition on $v$ and Proposition~\ref{propositionNonDegeneracyA}, we have the existence $v \in \left( \ker (I - L_{A_*}) \right)^\perp$ in $H^1_{A_*}(\R^N,\C)$. We can inverse the operator and it follows that 
\begin{equation*}
\begin{split}
\norm{v}_{H^1_{A_*} (\R^N, \mathbb{C})} & \le C \bignorm{(-\Delta_{A_*} + I)^{- 1} f}_{H^1_{A_*} (\R^N, \mathbb{C})}  \\
& \le C \sqrt{\bignorm{(-\Delta_{A_*} + I)^{- 1} f}_{L^2 (\R^N, \mathbb{C})} \norm{f}_{L^2 (\R^N, \mathbb{C})}} \le C \norm{f}_{L^2(\mathbb{R}^N, \mathbb{C})}.
\end{split}
\end{equation*}
Finally, if we assume by contradiction the existence of two solutions $v_1$, $v_2$, by subtracting the equations of $v_1$, $v_2$, we obtain that $v_1 - v_2$ has to solve the linear problem, for which we know by Proposition~\ref{propositionNonDegeneracyA} that the only solutions are given by $i u_{A_*}$ and $D_{A_*}u_{A_*}[h]$. We obtain a contradiction using the orthogonality conditions.
\end{proofclaim}

\begin{claim}
For each \(Q \in \mathcal{Q}\), the problem \eqref{eq:wA*} has a unique even and real-valued solution \(w_{A_*}[Q]\).
\end{claim}
\begin{proofclaim}
We remark that since \(u_{A_*}\) and \(Q\) are real-valued and even, we have immediately
\begin{equation*}
\left\{
\begin{aligned}
 \int_{\Rset^N} \scalprod{D_{A_*} u_{A_*}[h]}{-Qu_{A_*}}&= 0
&& \text{for each \(h \in \Rset^N\),}\\
 \int_{\Rset^N} \scalprod{i u_{A_*}}{-Qu_{A_*}}&= 0,
\end{aligned}
\right.
\end{equation*}
so that the problem has a unique complex-valued solution in view of Claim~\ref{claim:1estimate}.

By the symmetry properties of the functions \(u_{A_*}\), \(A_*\) and \(Q\), and by uniqueness, \(w_{A_*}[Q]\) is even and is invariant under the isometries that are compatible with \(A_*\). In view of Proposition~\ref{prop:symmetriesCoupling}, we have
\[
   -\Delta_{A_*} w_{A_*} = -\Delta w_{A_*} + \abs{A_*}^2 w_{A_*}.
\]
Since the functions \(u_{A_*}\) and \(Q\) are real valued, it follows then by uniqueness that \(w_{A_*}[Q]\) is real-valued for each \(Q \in \mathcal{Q}\).
\end{proofclaim}

We define now the function $v_A = u_A - u_{A_*} - w_{A_*}\bigl[\abs{A}^2 - \abs{A_*}^2\bigr]$. By using \eqref{eq:wA*} and by the equations satisfied by $u_A$,  $u_{A_*}$ and $w_{A_*}\bigl[\abs{A}^2 - \abs{A_*}^2\bigr]$, the function $v_A$ satisfies the equation
\begin{equation} \label{eq:vA}
\begin{split}
& -\Delta_{A_*} v_A + v_A - (p- 1) \abs{u_{A_*}}^{p - 2} v_A \\ 
& = \abs{u_A}^{p - 2} u_A  - \abs{u_{A_*}}^{p - 2} u_{A_*} + (p - 1) \abs{u_{A_*}}^{p - 2} (u_{A_*} - u_{A}) - \bigl(\abs{A}^2 - \abs{A_*}^2\bigr) (u_{A} - u_{A_*}).
\end{split}
\end{equation}
We remark that the mixed term vanishes: $i (A - A_*) \cdot \nabla u_A = 0$ because of the symmetries of $A$, $A_*$ and $u_A$ and the fact that $|A|^2$ and $|A_*|^2$ are in $\mathcal{Q}$.

\begin{claim} \label{claimExponentialDecay}
One has that
\[
u_A \to u_{A_*} \quad \text{ uniformly on } \R^N \text{ as } A \to A_* \, .
\]
Moreover, for every $\eta > 0$, there exists $\delta > 0$ small such that if $\norm{A - A_*} < \delta$, then for every $x \in \R^N$
\[
  \abs{u_A (x) } \le C \e^{-(1 - \eta) \abs{x}}.
\]
\end{claim}

\begin{proofclaim} 
By the symmetry properties of \(u_A\) and \(u_{A_*}\) and by Proposition~\ref{prop:symmetriesCoupling}, 
we can rewrite the equation of $u_A - u_{A_*}$ as
\begin{equation*}
\begin{split}
& - \Delta (u_A - u_{A_*}) + (u_A - u_{A_*}) \\
& = |u_A|^{p - 2} (u_A - u_{A_*}) + (|u_{A}|^{p - 2} - |u_{A_*}|^{p - 2} ) u_{A_*} + \abs{A}^2 (u_{A_*} - u_A) - (\abs{A}^2 - \abs{A_*}^2) u_{A_*}.
\end{split}
\end{equation*}
By classical regularity estimates, we first obtain that $u_{A} \to u_{A_*}$ uniformly on compact subsets of $\mathbb{R}^N$ as $A \to A_*$.

Next, since $u_A \to u_{A_*}$ in $L^p(\mathbb{R}^N)$ for $2 \leq p \leq 2^*$ (see Proposition~\ref{propositionContinuityGroundState} and Lemma~\ref{lemmaSobolev}), for every $\xi > 0$, there exist $R > 0$ and $\delta > 0$ such that if $\norm{A - A_*} < \delta$ then
\[
 \int_{\R^N \setminus B_{R }} \abs{u_A}^p \le \xi,                           
\]
from which we conclude that $\|u_A\|_{L^\infty(\mathbb{R}^N \setminus B_R)} \leq \eta$ (since we also know the exponential decay).

Next, for $\frac{1}{p} > \frac{1}{2} - \frac{1}{N}$, there exist $\xi > 0$ and $r > 0$, such that if $a \in \R^N$,
\[
 \int_{B_r (a)} \abs{u_A}^p \le \xi,
\]
by \cite{DiCosmoVanSchaftingen2015}*{proposition 4.6}, we conclude that $\|u_A\|_{L^\infty(B_{\frac{r}{2}}(a))} \leq \eta$. We can recover $B_R$ by a finite number of those small balls. 

Finally, by the classical Kato inequality
\begin{equation*}
 -\Delta \abs{u_{A}} + \abs{u_{A}} \le \abs{u_{A}}^{p - 1},
\end{equation*}
it follows that (by renaming $\eta$ if necessary)
\[
   \abs{u_A (x) } \le C \e^{- (1 - \eta)\abs{x}},
\]
for every $x \in \mathbb{R}^N$.
\end{proofclaim}

\begin{claim} \label{claim:3}
The following estimate holds
\[
 \norm{v_A}_{H^1_A (\R^N, \mathbb{C})} = o \bigl(\bignorm{\abs{A}^2 - \abs{A_*}^2}\bigr), \qquad \text{ as } A \to A_*. 
\]
\end{claim}

\begin{proofclaim}
We have to verify that $v_A$ satisfy the hypothesis of Claim~\ref{claim:1estimate}.

First, we remark that the equation of $v_A$ always makes sense since the right-hand-side of \eqref{eq:vA} is orthogonal in $L^2(\R^N)$ to $D_{A_*} u_{A_*}$ and $i u_{A_*}$. This is due to the evenness condition on $u_{A}$ and $u_{A_*}$, and their reality.

Next, we observe that the family satisfies 
\begin{equation*}
 \begin{split}
& \bigabs{\abs{u_A}^{p - 2} u_{A} -\abs{u_{A_*}}^{p - 2} u_{A_*} + (p - 1) \abs{u_{A_*}}^{p - 2} (u_{A_*} - u_{A})} \\
& \qquad \le C (\abs{u_A} + \abs{u_{A_*}})^{p - 2 - \gamma} \abs{u_A - u_{A_*}}^{\gamma} (\abs{v_A} + \abs{w_{A_*}[\abs{A}^2 - \abs{A_*}^2]}),
\end{split}
\end{equation*}
for $\gamma = \min (p - 2, 1)$ and $C > 0$. This object is in $L^2(\mathbb{R}^N)$ by using the uniform bounds in $\mathbb{R}^N$ of Claim~\ref{claimExponentialDecay} and the facts that $w_{A_*}[\abs{A}^2 - \abs{A_*}^2]$ and $v_A$ are in $H^1_{A_*}(\mathbb{R}^N, \mathbb{C})$.

Finally, we also have that 
\begin{equation}
\label{eq:vAb}
\scalprod{v_A}{D_{A_*} u_{A_*}}_{H^1_{A_*} (\Rset^N,\C)} = 0  \qquad \text{and} \qquad \scalprod{ v_A}{iu_{A_*}}_{H^1_{A_*} (\Rset^N,\C)} = 0, 
\end{equation}
still by using the evenness and reality conditions. This last condition is the orthogonality condition required in Claim~\ref{claim:1estimate}. Therefore, we can use the estimate of Claim~\ref{claim:1estimate}
\begin{multline*}
\norm{v_A}_{H^1_A (\R^N,\C)}\\
\le C \bigl(\norm{u_A}_{L^\infty} +  \norm{u_{A_*}}_{L^\infty}\bigr)^{p - 2 - \gamma}  \norm{u_A - u_{A_*}}_{L^\infty}^\gamma  \Bigl(\norm{v_{A_*}}_{L^2} + \bignorm{w_{A_*}[\abs{A}^2 - \abs{A_*}^2]}_{L^2} \Bigr) \\
  + C \| u_A - u_{A_*}\|_{L^\infty}^{1/2} \bignorm{\abs{A}^2 - \abs{A_*}^2},
\end{multline*}
where we used Claim~\ref{claimExponentialDecay} for the last term.

Since $\norm{w_{A_*}[\abs{A}^2 - \abs{A_*}^2]}_{L^2 (\R^N)}= O \bignorm{\abs{A}^2 - \abs{A_*}^2}$, $u_A \to u_{A_*}$ in $L^\infty(\R^N)$ by Claim~\ref{claimExponentialDecay} as $A \to A_*$, and $\gamma > 0$, we conclude that $\norm{v_A}_{H^1_A (\R^N,\C)} = o \left(\norm{\abs{A}^2 - \abs{A_*}^2} \right)$ as $A \to A_*$.
\end{proofclaim}

\begin{claim} \label{Claim:4}
Finally, we prove that
\[
 v_A (x) = o (\bignorm{\abs{A}^2 - \abs{A_*}^2}) \e^{-(1 - 2 \delta) \abs{x}}, \qquad \text{ as } A \to A_*, 
\]
uniformly in $x \in \R^N$.
\end{claim}

\begin{proofclaim}
By classical regularity estimates, we have $w_{A_*}[\abs{A}^2 - \abs{A_*}^2] \in C (\R^N)$. 
Starting again from \eqref{eq:vA} and using the fact that $v_A = o (\|\abs{A}^2- \abs{A_*}^2\|)$ in $L^2 (\R^N)$, we deduce that $v_A = o (\|\abs{A}^2- \abs{A_*}^2\|)$ uniformly on compact subsets of $\R^N$, see \cite{GilbargTrudinger1983}. 

Since the function $u_A$ has a uniform exponential bound (Claim~\ref{claimExponentialDecay}), we deduce from \eqref{eq:vA} and \eqref{eq:vAb} that if \(R > 0\) is large enough
\[
-\Delta \abs{v_A} + \frac{1}{2} \abs{v_A} \le o ( \| \abs{A}^2 - \abs{A_*}^2 \| ) \abs{x}^2 \e^{-(1 - \eta)\abs{x}}
\qquad \text{for each } x \in \R^N \setminus B_R,
\]
from which we conclude if \(2 \delta > \eta\) that
\[
v_A (x) = o ( \| \abs{A}^2 - \abs{A_*}^2 \| ) \e^{-(1 - 2 \delta) \abs{x}},
\]
uniformly for $x \in \R^N$.
\end{proofclaim}

We can now conclude the proof of the proposition. From Claim~\ref{Claim:4}, it follows that
\[
\norm{v_A}_{H^1_{A_*}(\R^N , \C)}  = o \bigl(\bignorm{ \abs{A}^2 - \abs{A_*}^2}) 
\]
and the conclusion follows.
\end{proof}

\subsection{Asymptotic expension of the ground-energy}

For a self-adjoint linear operator \(L \in \Lin (\Rset^N, \Rset^N)\), we let 
\(\lambda_1 (L), \dotsc, \lambda_N (L)\) denote the eigenvalues of \(L\), including
multiplicity. 
By the classical Courant--Fischer minimax principle, the functions \(\lambda_1, \dotsc, \lambda_N\)
can be taken to be continuous. Whereas simple eigenvalues are differentiable \cite{HornJohnson1985}*{\S 6.3},
eigenvalues are in general not differentiable. The simple following classical example shows that the first eigenvalue
\[
 \lambda_1 \Biggl( \begin{pmatrix} 0 & t \\ t & 0 \end{pmatrix} \Biggr) = -\abs{t}
\]
is not differentiable at $t = 0$, where it is in fact not simple.
However, it will be sufficient for our purpose to know that suitable \emph{averages of eigenvalues 
are differentiable}. Such results are well-known \citelist{\cite{Kato1976}*{II.2.2}\cite{Chu1990}*{\S 4}}.
We propose here for the convenience of the reader a direct proof 
following the strategy of the proof of differentiability of a function of a matrix \cite{HornJohnson1991}*{\S 6.6}.

\begin{lemma}[Differentiability of averaged eigenvalues]
\label{lemmaDiffAverEigenval}
Let \(L, M \in \Lin (\Rset^N, \Rset^N)\), \(\varepsilon > 0\) and \(j \in \{1, \dotsc, N\}\).
Assume that for every \(\ell \in \{1, \dotsc, N\}\) 
such that \(\lambda_{\ell} (L) \neq \lambda_j (L)\), 
one has \(\abs{\lambda_{\ell} (L) - \lambda_j (L)} > \varepsilon\).
Then,
\[
 \sum_{\ell \, : \, \abs{\lambda_\ell (M) - \lambda_j (L)} \le \varepsilon}
 \bigl(\lambda_\ell (M) - \lambda_j (L)\bigr)
 = \sum_{\ell \, : \, \abs{\lambda_\ell ({L}) - \lambda_j (L)} \le \varepsilon} e_\ell \cdot (M - L) (e_\ell)
 + o \bigl(\norm{M - L}\bigr),
\]
where \(e_1, \dotsc, e_N\) form an orthogonal basis of the space \(\Rset^N\) and are
eigenvectors of the operator \(L\) associated to the eigenvalues \(\lambda_1 (L), \dotsc, \lambda_N (L)\).
\end{lemma}
\begin{proof}
We consider the real-valued function \(f\) defined on a self-adjoint linear operator \(M\) by 
\[
 f (M) = \trace \Biggl(\bigl(M - \lambda_j (L) I\bigr)\circ \prod_{\ell \, : \, \abs{\lambda_\ell (L) - \lambda_j (L)} > \varepsilon}
 \frac{\bigl(M - \lambda_\ell(L) I \bigr)^2}{\bigl(\lambda_j(L) - \lambda_\ell (L)\bigr)^2}\Biggr).
\]
By computation of the trace with respect to an orthogonal basis of eigenvectors of the self-adjoint linear map \(M\), we have 
\[
\begin{split}
 f (M) &= \sum_{m = 1}^N \Biggl((\lambda_m (M) - \lambda_j (L)) \prod_{\ell \, : \,  \abs{\lambda_\ell (L) - \lambda_j (L)} > \varepsilon}
 \frac{\bigl(\lambda_m (M) - \lambda_\ell(L)\bigr)^2}{\bigl(\lambda_j(L) - \lambda_\ell (L)\bigr)^2}\Biggr)\\
 & = \sum_{\ell \, : \, \abs{\lambda_\ell (M) - \lambda_j (L)} \le \varepsilon}
 \bigl(\lambda_\ell (M) - \lambda_j (L)\bigr) + o \bigl(\norm{M  - L}\bigr).
\end{split}
\]
We note in particular that \(f (L) = 0\).
By classical differential calculus applied to a polynomial of linear operators, we have, on the other hand
\begin{multline*}
 f (M) = f (L) + \trace \Biggl((M - L) \circ \prod_{\ell \, : \, \abs{\lambda_\ell (L) - \lambda_j (L)} > \varepsilon}
 \frac{(L - \lambda_\ell(L) I )^2}{(\lambda_j(L) - \lambda_\ell (L))^2}\Biggr)\\
 + \sum_{m \, : \, \abs{\lambda_m (L) - \lambda_j (L) } > \varepsilon} 
 \trace  \Biggl(\bigl(L - \lambda_j (L) I \bigr)\circ \biggl(\prod_{\ell \, : \, \substack{\abs{\lambda_\ell (L) - \lambda_j (L)} > \varepsilon\\ \ell < m}}
 \frac{(L - \lambda_\ell(L) I )^2}{(\lambda_j(L) - \lambda_\ell (L))^2}\biggr)\\
 \circ \frac{\bigl(L - \lambda_m (L) I \bigr)\circ (M - L) + (M - L) \circ \bigl(L - \lambda_m (L) I \bigr)}{(\lambda_j (L) - \lambda_m (L))^2}\\
 \circ \biggl(\prod_{\ell \, : \, \substack{\abs{\lambda_\ell (L) - \lambda_j (L)} > \varepsilon\\ \ell > m}}\frac{(L - \lambda_\ell(L) I )^2}{(\lambda_j(L) - \lambda_\ell (L))^2}
 \biggr)\Biggr)
 + o \bigl(\norm{M - L}\bigr).
\end{multline*}
By computing the trace in the basis of eigenvectors \(e_1, \dotsc, e_N\) of $L$, we obtain
\[
 f (M) = \sum_{\ell \, : \, \abs{\lambda_\ell ({L}) - \lambda_j (L)} \le \varepsilon} e_\ell \cdot (M - L) (e_\ell) + o \bigl( \norm{M - L}\bigr);
\]
the conclusion follows.
\end{proof}

Using the differentiability result of Proposition~\ref{propositionDifferentiability}, we can then prove that the ground-energy is differentiable in a neighbourhood of $A = 0$.

\begin{proof}[Proof of Theorem~\ref{thm:energy}]
First assume that the magnetic fields \(B\) and \(B_*\) have a common basis of eigenvectors. 
The same then holds for the corresponding vector potentials \(A\) and \(A_*\) which commute. 
We then have by Proposition~\ref{propositionDifferentiability},
\[
u_{A} = u_{A_*} + w_{A_*}[\abs{A}^2 - \abs{A_*}^2] + o\bigl(\bignorm{\abs{A}^2 - \abs{A_*}^2} \bigr) \qquad \text{ in } H^1_{A_*}(\R^N,\C).
\]
From this we deduce that 
\begin{equation}
\label{eqFirstExpansion}
\begin{split} 
 \mathcal{E} (B) & = \mathcal{E} (B_*) + 
\dualprod{\mathcal{I}_{A_*}' (u_{A_*})}{w_{A_*} [\abs{A}^2 - \abs{A_*}^2]}
\\
& + \frac{1}{2} \int_{\Rset^N}( \abs{A}^2 - \abs{A_*}^2) \abs{u_{A_*}}^2
+ o\bigl(\bignorm{\abs{A}^2 - \abs{A_*}^2} \bigr)\\
&= \mathcal{E} (B_*) + \frac{1}{2} \int_{\Rset^N}( \abs{A}^2 - \abs{A_*}^2) \abs{u_{A_*}}^2+ o\bigl(\bignorm{\abs{A}^2 - \abs{A_*}^2} \bigr),
\end{split}
\end{equation}
since $\mathcal{I}_{A_*}'(u_{A_*}) = 0$. If \(\lambda_1 (A^* \circ A), \dotsc, \lambda_N (A^* \circ A)\) denote the eigenvalues in nonincreasing order of the operator \(A^* \circ A\) and if \(e_1, \dotsc, e_N\) form a common basis of eigenvectors of \(A_*^* \circ A_*\) and \(A^* \circ A\), we have 
\begin{equation*}
 \int_{\Rset^N}( \abs{A}^2 - \abs{A_*}^2) \abs{u_{A_*}}^2 
 = \sum_{j=1}^N \int_{\Rset^N} \bigl(\lambda_j (A^* \circ A) - \lambda_j (A_*^* \circ A_*) \bigr) \abs{e_j \cdot x}^2 \abs{u_{A_*} (x)}^2 \dif x.
\end{equation*}
We now decompose \(\{1, \dotsc, N\} = \bigcup_{k = 1}^K J_k\), so that \(\lambda_j (A_*^* \circ A_*) = \lambda_\ell (A_*^* \circ A_*)\) if and only if there exists \(k \in \{1, \dotsc, N\}\) such that 
\(j, \ell \in J_k\).
We then have, in view of the symmetry properties of \(u_{A_*}\), see Theorem~\ref{theoremSymmetry}, that
\begin{multline*}
 \int_{\Rset^N}( \abs{A}^2 - \abs{A_*}^2) \abs{u_{A_*}}^2\\
 = \sum_{k = 1}^K \frac{1}{\# J_k} \sum_{j \in J_k} \sum_{\ell \in J_k} 
 \int_{\Rset^N} \bigl(\lambda_j (A^* \circ A) - \lambda_j (A_*^* \circ A_*) \bigr)
 |e_\ell \cdot x|^2 \abs{u_{A_*} (x)}^2 \dif x.
\end{multline*}
By using again the rotational invariance of Theorem~\ref{theoremSymmetry}, as well as Lemma~\ref{lemma:propertyGSE}, we have in general, that is even if \(A\) and \(A_*\) do not have a common basis of complex eigenvectors, that there exists an isometry \(R : \Rset^N \to \Rset^N\) such that, if we define \(\Tilde{A} \in \Lin (\Rset^N, \bigwedge^1 \Rset^N)\) by \(\Tilde{A}(x)[v] = A (R (x))[R (v)]\), \(\Tilde{A}\) and \(A_*\) have a common basis of complex eigenvectors, which implies that \(\Tilde{A}^* \circ \Tilde{A}\) and \(A_*^* \circ A_*\) have a common orthonormal basis of real eigenvectors. Moreover, these common eigenvectors of \(\Tilde{A}^* \circ \Tilde{A}\) and \(A_*^* \circ A_*\) appear in nonincreasing order of associated eigenvalues.
We then have, if we also order the eigenvalues of \(A^* \circ A\) in nondecreasing order, that \(\lambda_j (\Tilde{A}^* \circ \Tilde{A}) =  \lambda_j (A^* \circ A)\) and by the rotational invariance of the \(\mathcal{E}\) (Lemma~\ref{lemma:propertyGSE}), \(\mathcal{E} (B) = \mathcal{E} (dA) = \mathcal{E} (d\Tilde{A})\). Therefore, using what we did above
\begin{multline*}
\mathcal{E} (B) 
= 
  \mathcal{E} (B_*) + \frac{1}{2}\sum_{k = 1}^K \frac{1}{\# J_k} \sum_{j \in J_k} \sum_{\ell \in J_k} 
 \int_{\Rset^N} \bigl(\lambda_j (A^* \circ A) - \lambda_j (A_*^* \circ A_*) \bigr)
 |e_\ell \cdot x|^2 \abs{u_{A_*} (x)}^2 \dif x \\
 + o \bigl(\bignorm{\abs{A_*}^2 - \abs{A}^2} \bigr).
\end{multline*}
In view of the differentiability of averaged eigenvalues of Lemma~\ref{lemmaDiffAverEigenval}, we deduce that
\begin{multline*}
\mathcal{E} (B) 
=
  \mathcal{E} (B_*) + \frac{1}{2}\sum_{k = 1}^K \frac{1}{\# J_k} \sum_{j \in J_k} \sum_{\ell \in J_k} 
 \int_{\Rset^N} (\abs{A (e_j)}^2 - \abs{A_* (e_j)}^2)
 |e_\ell \cdot x|^2 \abs{u_{A_*} (x)}^2 \dif x\\ + o \bigl(\bignorm{\abs{A_*}^2 - \abs{A}^2} \bigr).
\end{multline*}
By the rotation invariance in the eigenspaces, this becomes 
\[
\mathcal{E} (B) 
=
  \mathcal{E} (B_*) + \frac{1}{2} \sum_{j= 1}^N
 \int_{\Rset^N} (\abs{A (e_j)}^2 - \abs{A_* (e_j)}^2)
 |e_j \cdot x|^2 \abs{u_{A_*} (x)}^2 \dif x + o \bigl(\bignorm{\abs{A_*}^2 - \abs{A}^2} \bigr).
\]
Since \(u_{A_*}\) is invariant under reflections in the directions \(e_1, \dotsc, e_N\), we have
\[
 \int_{\Rset^N}  (e_j \cdot x) (e_\ell \cdot x) \abs{u_{A_*} (x)}^2 \dif x = 0,
\]
so that
\[
\int_{\Rset^N} \bigl(\abs{A}^2 - \abs{A_*}^2\bigr) \abs{u_{A_*}}  = \sum_{j=1}^N \int_{\R^N} \bigl( \abs{A(e_j)}^2 - \abs{A_*(e_j)}^2 \bigr) |x \cdot e_j|^2 \abs{u_{A_*}(x)}^2 \dif x.
\]
Therefore,
\begin{equation}
\label{eqDerivativeEnergy}
\begin{split}
 \mathcal{E} (B) 
 &=\mathcal{E} (B_*) + \frac{1}{2} \int_{\Rset^N} \bigl(\abs{A}^2 - \abs{A_*}^2\bigr) \abs{u_{A_*}}^2  + o \bigl(\bignorm{\abs{A}^2 - \abs{A_*}^2} \bigr) \\
 &= \mathcal{E} (B_*) + \int_{\Rset^N}  A_* \cdot (A - A_*) \abs{u_{A_*}}^2 + o \bigl(\norm{A - A_*}\bigr)\\
 &= \mathcal{E} (B_*) + \frac{1}{4} \int_{\Rset^N} B_* (x) \cdot (B (x) - B_* (x)) \abs{u_{A_*} (x)}^2 \dif x + o \bigl(|B - B_*|\bigr),
\end{split}
\end{equation}
where we set $B(x)[v] = B[x,v]$. Therefore, the differentiability of the function \(\mathcal{E}\) follows in a neighbourhood of $0$.

For the convexity, given \(B, B_* \in \bigwedge^2 \Rset^N\), we use the formula of the derivative of the energy
to write
\[
\begin{split}
 \dualprod{\mathcal{E}'(B) - \mathcal{E}' (B_*)}{B - B_*}
 &= \frac{1}{4} \int_{\Rset^N} B (x) \cdot (B (x) - B_* (x)) \abs{u_{A} (x)}^2 \dif x \\
 &\qquad - \frac{1}{4} \int_{\Rset^N} B_* (x) \cdot (B (x) - B_* (x)) \abs{u_{A_*} (x)}^2 \dif x\\
 &\ge \frac{1}{4}  \int_{\Rset^N} \abs{B (x) - B_* (x)}^2 \abs{u_{A_*} (x)}^2 \dif x\\
 &\qquad- C | B - B_* | \int_{\Rset^N} \abs{x}^2 \bigabs{\abs{u_{A_*} (x)}^2 - \abs{u_A (x)}^2}\dif x\\
 &\ge \frac{1}{4}  \int_{\Rset^N} \abs{B (x) - B_* (x)}^2 \abs{u_{A_*} (x)}^2 \dif x - o (\abs{B - B_*}^2),
\end{split}
\]
in view of Proposition~\ref{propositionDifferentiability}, when \(B \to B_*\), uniformly in \(B_*\)
in a sufficiently small neighbourhood of \(0\). It follows then that if \(B\) and \(B_*\) are close enough
to \(0\) 
\[
 \dualprod{\mathcal{E}'(B) - \mathcal{E}' (B_*)}{B - B_*} \ge 0,
\]
from which the convexity follows.

By the diamagnetic inequality \eqref{eqDiamagnetic}, the function \(\mathcal{E}\) achieves its global minimum at \(0\).

Finally, by \eqref{eqFirstExpansion} with \(B_* = 0\), we have 
\[
\begin{split}
\mathcal{E}(B) &= \mathcal{E}(0) + \frac{1}{2} \int_{\mathbb{R}^N} \abs{A}^2 \abs{u_0}^2 + o(\norm{A}^2)\\
& = \mathcal{E}(0) + \frac{1}{8} \int_{\mathbb{R}^N} \abs{B (x)}^2 \abs{u_0 (x)}^2 \dif x + o (\abs{B}^2).
\end{split}
\]
Since the function \(u_0\) is radial, we have
\[
\begin{split}
  \int_{\mathbb{R}^N} \abs{B (x)}^2 \abs{u_0 (x)}^2 \dif x
  &= \sum_{j, \ell, m = 1}^N \int_{\mathbb{R}^N}(e_j \cdot x) (e_\ell \cdot x) B (e_j, e_m) B (e_\ell, e_m) \abs{u_0 (x)}^2 \dif x\\
  &= \sum_{j, m = 1}^N (B (e_j, e_m))^2 \int_{\Rset^N} (e_j \cdot x)^2 \abs{u_0 (x)}^2 \dif x\\
  &= \sum_{j, m = 1}^N (B (e_j, e_m))^2 \frac{1}{N} \int_{\Rset^N} \abs{x}^2 \abs{u_0 (x)}^2 \dif x.
\end{split}
\]
The conclusion follows from the fact that the Euclidean norm on \(\bigwedge^2 \Rset^N\) 
which is compatible with the exterior product is given by 
\[
 \abs{B}^2 = \frac{1}{2} \sum_{j, m = 1}^N (B (e_j, e_m))^2.\qedhere
\]
\end{proof}


\begin{bibdiv}
\begin{biblist}

\bib{Agmon1984}{article}{
   AUTHOR = {Agmon, Shmuel},
   TITLE = {Bounds on exponential decay of eigenfunctions of Schr\"odinger operators},
   CONFERENCE = {
                  title = {Schr\"odinger operators},
                  address = {Como},
                  date = {1984},
                 },
   BOOK = {
           series = {Lecture Notes in Math.},
           volume = {1159},
           publisher = {Springer}, 
           address = {Berlin},
           },
   DATE = {1985},
   PAGES = {1--38},
}

\bib{AmbrosioDalMaso1990}{article}{
   AUTHOR = {Ambrosio, Luigi},
   AUTHOR = {Dal Maso, Gianni},
   TITLE = {A general chain rule for distributional derivatives},
   JOURNAL = {Proc. Amer. Math. Soc.},
   VOLUME = {108},
   YEAR = {1990},
   NUMBER = {3},
   PAGES = {691--702},
   ISSN = {0002-9939},
}

\bib{AmbrosettiMalchiodi}{book}{
    AUTHOR = {Ambrosetti, Antonio},
    AUTHOR = {Malchiodi, Andrea},
    TITLE = {Perturbation methods and semilinear elliptic problems on {${\bf R}^n$}},
    SERIES = {Progress in Mathematics},
    VOLUME = {240},
    PUBLISHER = {Birkh\"auser}, 
    address={Basel},
    YEAR = {2006},
    PAGES = {xii+183},
    ISBN = {978-3-7643-7321-4; 3-7643-7321-0},
}


\bib{AmbrosettiMalchiodiRuiz}{article}{
   AUTHOR = {Ambrosetti, Antonio},
   AUTHOR = {Malchiodi, Andrea},
   AUTHOR = {Ruiz, David},
   TITLE = {Bound states of nonlinear Schr\"odinger equations with potentials vanishing at infinity},
   JOURNAL = {J. Anal. Math.},
   VOLUME = {98},
   YEAR = {2006},
   PAGES = {317--348},
   ISSN = {0021-7670},
}




\bib{BartschWethWillem}{article}{
   AUTHOR = {Bartsch, Thomas},
   AUTHOR = {Weth, Tobias},
   AUTHOR = {Willem, Michel},
   TITLE = {Partial symmetry of least energy nodal solutions to some variational problems},
   JOURNAL = {J. Anal. Math.},
   VOLUME = {96},
   YEAR = {2005},
   PAGES = {1--18},
   ISSN = {0021-7670},
}
		
\bib{BonheureBouchezGrumiau2009}{article}{
   AUTHOR = {Bonheure, Denis},
   AUTHOR = {Bouchez, Vincent},
   AUTHOR = {Grumiau, Christopher},
   TITLE = {Asymptotics and symmetries of ground-state and least energy nodal solutions for boundary-value problems with slowly growing superlinearities},
   JOURNAL = {Differential Integral Equations},
   VOLUME = {22},
   YEAR = {2009},
   NUMBER = {9-10},
   PAGES = {1047--1074},
   ISSN = {0893-4983},
}

\bib{BonheureBouchezGrumiauVanSchaftingen2008}{article}{
   AUTHOR = {Bonheure, Denis},
   AUTHOR = {Bouchez, Vincent},
   AUTHOR = {Grumiau, Christopher},
   AUTHOR = {Van Schaftingen, Jean},
   TITLE = {Asymptotics and symmetries of least energy nodal solutions of Lane-Emden problems with slow growth},
   JOURNAL = {Commun. Contemp. Math.},
   VOLUME = {10},
   YEAR = {2008},
   NUMBER = {4},
   PAGES = {609--631},
   ISSN = {0219-1997},
}


\bib{BrockSolynin2000}{article}{
   AUTHOR = {Brock, Friedemann},
   AUTHOR = {Solynin, Alexander Yu.},
   TITLE = {An approach to symmetrization via polarization},
   JOURNAL = {Trans. Amer. Math. Soc.},
   VOLUME ={352},
   YEAR = {2000},
   NUMBER = {4},
   PAGES = {1759--1796},
   ISSN = {0002-9947},
}
	
		
\bib{Chu1990}{article}{
   AUTHOR = {Chu, King-Wah Eric},
   TITLE = {On multiple eigenvalues of matrices depending on several parameters},
   JOURNAL = {SIAM J. Numer. Anal.},
   VOLUME = {27},
   YEAR = {1990},
   NUMBER = {5},
   PAGES = {1368--1385},
   ISSN = {0036-1429},
}


\bib{Coffman1972}{article}{
   AUTHOR = {Coffman, Charles V.},
   TITLE = {Uniqueness of the ground state solution for $\Delta u-u+u^{3}=0$\ and a variational characterization of other solutions},
   JOURNAL = {Arch. Rational Mech. Anal.},
   VOLUME = {46},
   YEAR = {1972},
   PAGES = {81--95},
   ISSN = {0003-9527},
}


\bib{DiCosmoVanSchaftingen2015}{article}{
   AUTHOR = {Di Cosmo, Jonathan},
   AUTHOR = {Van Schaftingen, Jean},
   TITLE = {Semiclassical stationary states for nonlinear Schr\"odinger equations under a strong external magnetic field},
   JOURNAL = {J. Differential Equations},
   VOLUME = {259},
   YEAR = {2015},
   NUMBER = {2},
   PAGES = {596--627},
   ISSN = {0022-0396},
}


\bib{Erdos1996}{article}{
   AUTHOR = {Erd{\H{o}}s, László},
   TITLE = {Gaussian decay of the magnetic eigenfunctions},
   JOURNAL = {Geom. Funct. Anal.},
   VOLUME = {6},
   DATE = {1996},
   NUMBER = {2},
   PAGES = {231--248},
   ISSN = {1016-443X},
}


\bib{EstebanLions1999}{incollection}{
   AUTHOR = {Esteban, Maria J.},
   AUTHOR = {Lions, Pierre-Louis},
   TITLE = {Stationary solutions of nonlinear {S}chr\"odinger equations with an external magnetic field},
   BOOKTITLE = {Partial differential equations and the calculus of variations,
              {V}ol.\ {I}},
   SERIES = {Progr. Nonlinear Differential Equations Appl.},
   VOLUME = {1},
   PAGES = {401--449},
   PUBLISHER = {Birkh\"auser Boston},
   ADDRESS = {Boston, MA},
   YEAR = {1989},
}


\bib{FournaisLeTreustRaymondVanSchaftingen}{article}{
   AUTHOR = {Fournais, Soeren},
   AUTHOR = {Le Treust, Lo\"\i c},
   AUTHOR = {Raymond, Nicolas},
   AUTHOR = {Van Schaftingen, Jean},
   TITLE = {Semiclassical Sobolev constants for the electro-magnetic Robin Laplacian},
   EPRINT = {arXiv:1603.02810},
}


\bib{FournaisRaymond}{article}{
   AUTHOR = {Fournais, Soeren},
   AUTHOR = {Raymond, Nicolas},
   TITLE = {Optimal magnetic Sobolev constants in the semiclassical limit},
   JOURNAL = {Ann. Inst. H. Poincar\'e Anal. Non Lin\'eaire},
  doi={10.1016/j.anihpc.2015.03.008},
}


\bib{GilbargTrudinger1983}{book}{
   AUTHOR = {Gilbarg, David},
   AUTHOR = {Trudinger, Neil S.},
   TITLE = {Elliptic partial differential equations of second order},
   SERIES = {Grundlehren der Mathematischen Wissenschaften},
   VOLUME = {224},
   EDITION = {2},
   PUBLISHER = {Springer}, 
   ADDRESS = {Berlin},
   YEAR = {1983},
   PAGES = {xiii+513},
   ISBN = {3-540-13025-X},
}


\bib{HornJohnson1985}{book}{
   AUTHOR = {Horn, Roger A.},
   AUTHOR = {Johnson, Charles R.},
   TITLE = {Matrix analysis},
   PUBLISHER = {Cambridge University Press}, 
   ADDRESS = {Cambridge},
   YEAR = {1985},
   PAGES = {xiii+561},
   ISBN = {0-521-30586-1},
}


\bib{HornJohnson1991}{book}{
   AUTHOR = {Horn, Roger A.},
   AUTHOR = {Johnson, Charles R.},
   TITLE = {Topics in matrix analysis},
   PUBLISHER = {Cambridge University Press}, 
   ADDRESS = {Cambridge},
   YEAR = {1991},
   PAGES = {viii+607},
   ISBN = {0-521-30587-X},
}
                



\bib{Kato1976}{book}{
   AUTHOR = {Kato, Tosio},
   TITLE = {Perturbation theory for linear operators},
   EDITION = {2},
   series = {Grundlehren der Mathematischen Wissenschaften}, 
   VOLUME = {132},
   PUBLISHER = {Springer}, 
   ADDRESS = {Berlin--New York},
   DATE = {1976},
   PAGES = {xxi+619},
}


\bib{Kwong1989}{article}{
   AUTHOR = {Kwong, Man Kam},
   TITLE = {Uniqueness of positive solutions of $\Delta u-u+u^p=0$ in ${\bf R}^n$},
   JOURNAL = {Arch. Rational Mech. Anal.},
   VOLUME = {105},
   YEAR = {1989},
   NUMBER = {3},
   PAGES = {243--266},
   ISSN = {0003-9527},
}


\bib{Lax2002}{book}{
   AUTHOR = {Lax, Peter D.},
   TITLE = {Functional analysis},
   SERIES = {Pure and Applied Mathematics},
   PUBLISHER = {Wiley-Interscience},
   PLACE = {New York},
   YEAR = {2002},
   PAGES = {xx+580},
   ISBN = {0-471-55604-1},
}


\bib{LiebLoss}{book}{
   AUTHOR = {Lieb, Elliott H.},
   AUTHOR = {Loss, Michael},
   TITLE = {Analysis},
   SERIES = {Graduate Studies in Mathematics},
   VOLUME = {14},
   EDITION = {2},
   PUBLISHER = {American Mathematical Society},
   ADDRESS = {Providence, R.I.},
   YEAR = {2001},
   PAGES = {xxii+346},
   ISBN = {0-8218-2783-9},
}


\bib{Lions1984CC2}{article}{
   AUTHOR = {Lions, Pierre-Louis},
   TITLE = {The concentration-compactness principle in the calculus of variations. The locally compact case. II},
   JOURNAL = {Ann. Inst. H. Poincar\'e Anal. Non Lin\'eaire},
   VOLUME = {1},
   YEAR = {1984},
   NUMBER = {4},
   PAGES = {223--283},
   ISSN = {0294-1449},
}


\bib{McLeodSerrin1987}{article}{
   AUTHOR = {McLeod, Kevin},
   AUTHOR = {Serrin, James},
   TITLE = {Uniqueness of positive radial solutions of $\Delta u+f(u)=0$ in ${\bf R}^n$},
   JOURNAL = {Arch. Rational Mech. Anal.},
   VOLUME = {99},
   YEAR = {1987},
   NUMBER = {2},
   PAGES = {115--145},
   ISSN = {0003-9527},
}


\bib{MorozVanSchaftingen2013}{article}{
    AUTHOR = {Moroz, Vitaly},
    AUTHOR = {Van Schaftingen, Jean},
    TITLE= {Nonexistence and optimal decay of supersolutions to {C}hoquard equations in exterior domains},
    JOURNAL = {J. Differential Equations},
    VOLUME = {254},
    YEAR = {2013},
    NUMBER = {8},
    PAGES = {3089--3145},
    ISSN = {0022-0396},
}


\bib{MorozVanSchaftingen2013-2}{article}{
    AUTHOR = {Moroz, Vitaly},
    AUTHOR ={Van Schaftingen, Jean},
    TITLE = {Groundstates of nonlinear {C}hoquard equations: existence, qualitative properties and decay asymptotics},
    JOURNAL = {J. Funct. Anal.},
    VOLUME= {265},
    YEAR = {2013},
    NUMBER = {2},
    PAGES = {153--184},
    ISSN = {0022-1236},
    URL = {http://dx.doi.org/10.1016/j.jfa.2013.04.007},
}


\bib{Oh1990}{article}{
   AUTHOR = {Oh, Yong-Geun},
   TITLE = {On positive multi-lump bound states of nonlinear Schr\"odinger equations under multiple well potential},
   JOURNAL = {Comm. Math. Phys.},
   VOLUME = {131},
   YEAR = {1990},
   NUMBER = {2},
   PAGES = {223--253},
   ISSN = {0010-3616},
}


\bib{ShiojiWatanabe2016}{article}{
   author={Shioji, Naoki},
   author={Watanabe, Kohtaro},
   title={Uniqueness and nondegeneracy of positive radial solutions of ${\rm
   div}(\rho\nabla u)+\rho(-gu+hu^p)=0$},
   journal={Calc. Var. Partial Differential Equations},
   volume={55},
   date={2016},
   number={2},
   pages={art. 32},
   issn={0944-2669},
}

\bib{Shirai2008}{article}{
   AUTHOR = {Shirai, Shin-ichi},
   TITLE = {Existence and decay of solutions to a semilinear Schr\"odinger equation with magnetic field},
   JOURNAL = {Hokkaido Math. J.},
   VOLUME = {37},
   YEAR = {2008},
   NUMBER = {2},
   PAGES = {241--273},
   ISSN = {0385-4035},
}


\bib{VanSchaftingen2009}{article}{
   AUTHOR = {Van Schaftingen, Jean},
   TITLE = {Explicit approximation of the symmetric rearrangement by polarizations},
   JOURNAL = {Arch. Math. (Basel)},
   VOLUME = {93},
   YEAR = {2009},
   NUMBER = {2},
   PAGES = {181--190},
   ISSN = {0003-889X},
}    


\bib{VanSchaftingen2014}{article}{
   AUTHOR = {Van Schaftingen, Jean},
   TITLE = {Interpolation inequalities between Sobolev and Morrey-Campanato spaces: a common gateway to concentration-compactness and Gagliardo-Nirenberg interpolation inequalities},
   JOURNAL = {Port. Math.},
   VOLUME = {71},
   YEAR = {2014},
   NUMBER = {3--4},
   PAGES = {159--175},
   ISSN = {0032-5155},
}
                

\bib{VanSchaftingenWillem2008}{article}{
    AUTHOR = {Van Schaftingen, Jean},
    AUTHOR = {Willem, Michel},
    TITLE = {Symmetry of solutions of semilinear elliptic problems},
    JOURNAL = {J. Eur. Math. Soc. (JEMS)},
    VOLUME = {10},
    YEAR = {2008},
    NUMBER = {2},
    PAGES = {439--456},
    ISSN = {1435-9855},
}


\bib{Weinstein1985}{article}{
   AUTHOR = {Weinstein, Michael I.},
   TITLE = {Modulational stability of ground states of nonlinear Schr\"odinger equations},
   JOURNAL = {SIAM J. Math. Anal.},
   VOLUME = {16},
   YEAR = {1985},
   NUMBER = {3},
   PAGES = {472--491},
   ISSN = {0036-1410},
}

\bib{Willem1996}{book}{
    AUTHOR = {Willem, Michel},
    TITLE = {Minimax theorems},
    SERIES = {Progress in Nonlinear Differential Equations and their Applications, 24},
    PUBLISHER = {Birkh\"auser},
    ADDRESS = {Boston, Mass.},
    YEAR = {1996},
    PAGES = {x+162},
    ISBN = {0-8176-3913-6}
}

\end{biblist}
\end{bibdiv}
\end{document}